\RequirePackage{fix-cm}
\documentclass{amsart} 
\usepackage[utf8]{inputenc}
\usepackage{mathptmx}
\usepackage{latexsym}
\usepackage{amsmath}
\usepackage{amssymb}
\usepackage{amsfonts} 
\usepackage{eucal} 
\usepackage{amsbsy}
\usepackage{amsthm}
\usepackage[all]{xy}
\usepackage[pagebackref=true]{hyperref}
\usepackage{graphicx}
\usepackage{tikz}
\usepackage{color, colortbl}
\usepackage{stmaryrd}

\definecolor{Gray}{gray}{0.9}

\definecolor{Gray1}{gray}{0.9}
\definecolor{Gray2}{gray}{0.8}
\definecolor{Gray3}{gray}{0.7}

\newtheorem{thm}{Theorem}[section]
\newtheorem*{thm*}{Theorem}
\newtheorem{lemma}[thm]{Lemma}
\newtheorem{prop}[thm]{Proposition}

\newtheorem*{cor*}{Corollary}

\theoremstyle{definition}
\newtheorem{defn}[thm]{Definition}
\newtheorem{example}[thm]{Example}

\theoremstyle{remark}
\newtheorem{remark}[thm]{Remark}

\newcommand {\dual}  {\ensuremath{\mathcal{D}}}

\newcommand {\Sa}    {\ensuremath{\mbox{$\mathcal{S}$}}}

\newcommand {\real}  {\ensuremath{\mathbb{R}}}
\newcommand {\nat}  {\ensuremath{\mathbb{N}}}
\newcommand {\intg}  {\ensuremath{\mathbb{Z}}}

\newcommand {\cplx}  {\ensuremath{\mathbb{C}}}
\newcommand {\rat}   {\ensuremath{\mathbb{Q}}}

\newcommand {\Tor}   {\ensuremath{\operatorname{Tor}}}

\newcommand {\im}    {\operatorname{im}}

\newcommand {\colim} {\ensuremath{\operatorname{colim}}}

\newcommand {\smlhf} {\ensuremath{\mbox{$\frac{1}{2}$}}}

\newcommand {\pro}   {\ensuremath{\operatorname{pr}}}

\newcommand {\pt}    {\ensuremath{\operatorname{pt}}}
\newcommand {\id}    {\ensuremath{\operatorname{id}}}

\newcommand {\BO}   {\ensuremath{\operatorname{BO}}}

\newcommand {\BPL}   {\ensuremath{\operatorname{BPL}}}

\newcommand {\BBPL}   {\ensuremath{\operatorname{B}\widetilde{\operatorname{PL}}}}

\newcommand {\SO}   {{\ensuremath{\operatorname{SO}}}}

\newcommand {\Th}   {\ensuremath{\operatorname{Th}}}

\newcommand {\KO}   {{\ensuremath{\operatorname{KO}}}}
\newcommand {\K}   {{\ensuremath{\operatorname{K}}}}

\newcommand {\Or}    {\ensuremath{\operatorname{O}}}
\newcommand {\reg}    {{\ensuremath{\operatorname{reg}}}}
\newcommand {\BM}    {{\ensuremath{\operatorname{BM}}}}
\newcommand {\Ad}    {{\ensuremath{\operatorname{Ad}}}}
\newcommand {\wgr}    {{\ensuremath{\widetilde{\operatorname{Gr}}}}}

\newcommand {\pr}  {\ensuremath{\mathbb{P}}}

\newcommand {\Gr}    {\ensuremath{\operatorname{Gr}}}

\makeatletter
\@namedef{subjclassname@2020}{%
  \textup{2020} Mathematics Subject Classification}
\makeatother

\begin{document}


\title[The Equivariant L-Class]
  {The Equivariant L-Class of Pseudomanifolds}

\author{Markus Banagl}

\address{Institut f\"ur Mathematik, Universit\"at Heidelberg,
  Im Neuenheimer Feld 205, 69120 Heidelberg, Germany}

\email{banagl@mathi.uni-heidelberg.de}

\thanks{This work is funded in part by a research grant of the
 Deutsche Forschungsgemeinschaft (DFG, German Research Foundation)
 -- Projektnummer 495696766.}
 
\date{December 2024}





\begin{abstract}
We construct an equivariant L-class for orientation preserving actions 
of a compact Lie group on a Whitney stratified compact oriented pseudomanifold
that satisfies the Witt condition, for example on a compact pure-dimensional 
complex algebraic variety. The class lies in equivariant rational homology and 
its restriction to the trivial group is the Goresky-MacPherson L-class.
For a smooth action on a manifold, the class is equivariantly Poincar\'e dual
to the Hirzebruch L-class of the Borel homotopy quotient of the
tangent bundle. We also provide a product formula under the equivariant
Künneth isomorphism. If the group acts freely, then the equivariant L-class
identifies with the Goresky-MacPherson L-class of the orbit space.
The construction method rests on establishing
Whitney (B)-regularity for finite-dimensional compact pseudomanifold approximations to the
Borel construction, and on the author's Verdier-Riemann-Roch type formulae for
the Goresky-MacPherson L-class.
\end{abstract}

\maketitle


\tableofcontents


\section{Introduction}

In \cite{gmih1},
Goresky and MacPherson used intersection homology to construct
an $L$-class $L_* (X) \in H_* (X;\rat)$ for compact, oriented Whitney stratified 
pseudomanifolds $X$ in an ambient smooth manifold $M$.
This was initially done for spaces without odd-codimensional strata
(such as complex algebraic varieties), but then extended by 
Siegel to Witt spaces (\cite{siegel}) and by the author to more general
spaces whose strata of odd codimension admit Lagrangian structures in
their middle dimensional link cohomology sheaves 
(\cite{banagl-mem}, \cite{banagl-lcl}, \cite{banagltiss}).
For a smooth manifold, this class is the Poincar\'e dual of Hirzebruch's
cohomological $L$-class of the tangent bundle.

Up to now, no construction of an equivariant $L$-class
for singular spaces has been given, as has been pointed out for example in
\cite{cmssequivcharcl}.
We propose here the construction of a $G$-equivariant $L$-class
$L^G_* (X) \in H^G_* (X;\rat)$ in equivariant rational homology for
compact, oriented Whitney stratified pseudomanifolds $X \subset M$ that are
invariant under the smooth action of a compact Lie group
$G$ on an ambient smooth manifold $M$. 
The pseudomanifold $X$ is assumed to satisfy the Witt condition.

Our method rests on the Verdier-Riemann-Roch (VRR) type formulae for the
Goresky-MacPherson $L$-class established in 
\cite{banaglnyjm} (immersive case) and 
\cite{banaglbundletransfer} (submersive case).
Let $X \subset M$ be a Whitney stratified subset of a smooth manifold $M$
and let $G$ be a compact Lie group acting smoothly on $M$ such that
$X$ is $G$-invariant and the induced action of $G$
on $X$ is compatible with the stratification, 
i.e. every stratum is $G$-invariant and the induced action on the stratum is smooth.
If $E$ is a smooth manifold, then $E\times X$ is Whitney (B)-regular in
$E\times M$ (Lemma \ref{lem.productwhitneystrat}).
We apply this principle to compact smooth Stiefel manifolds $EG_k$, 
equipped with a free and smooth $G$-action, that serve
as finite dimensional $(k-1)$-connected approximations to a free contractible
$G$-space $EG$ as $k\to \infty$.
Thus $EG_k \times X \subset EG_k \times M$ is Whitney (B)-regular.
We go on to show that if $G$ acts freely on $M$, then the embedding
$X/G \subset M/G$ is Whitney (B)-regular (Theorem \ref{thm.orbitspaceoffreeisbregular}),
and if $X$ is a pseudomanifold, then $X/G$ is a pseudomanifold
(Lemma \ref{lem.quotientoffreeispseudomfd}). 
Applying this to the free $G$-space $EG_k \times X \subset EG_k \times M$,
we conclude that $X_G (k) = EG_k \times_G X$ is a Whitney stratified
subset of the smooth manifold $M_G (k) = EG_k \times_G M$
(Theorem \ref{thm.xgkwhitneystratpsdmfdwitt}).
Moreover, if $X$ is a compact pseudomanifold that satisfies the Witt condition,
then $X_G (k)$ inherits these properties.
Orientability questions are treated in Section \ref{sec.orientability}. 
The main result there is that the $X_G (k)$ are orientable
if $G$ is bi-invariantly orientable and its action preserves the
orientation of $X$. We will subsequently assume that $G$ is
bi-invariantly orientable. This includes all (compact) connected groups,
all finite groups and all (compact) abelian groups.
Thus $X_G (k)$ possesses a Goresky-MacPherson $L$-class, which can be used,
after correction by the cohomological $L$-class of $BG_k = EG_k/G$,
to define a stage-$k$ equivariant $L$-class $L^G_{*,k} (X)$
(Definition \ref{def.stagekequivlclass}).
The VRR-Theorem is then used (Proposition \ref{prop.stagekp1lmapstostagekl}) to prove 
that these classes $\{ L^G_{*,k} (X) \}_k$
constitute an element of the inverse limit defining equivariant homology $H^G_*$.
This element is $L^G_* (X)$ (Definition \ref{def.equivlclasssingular}, which is central
to this paper).
Equivariant homology is reviewed in Section \ref{sec.equivarianthomology}. 
It is generally nontrivial in negative degrees and thus not the nonequivariant homology of $EG\times_G X$.
Oriented compact $G$-pseudomanifolds of dimension $m$ have an equivariant fundamental
class $[X]_G \in H^G_m (X)$ (Section \ref{ssec.equivfundclass}).
If $X=M$ is smooth, then equivariant Poincar\'e duality is the
isomorphism $H^{m-i}_G (M) \cong H^G_i (M)$ given by capping with $[M]_G$
(see Section \ref{ssec.equivpd}). Here, equivariant cohomology $H^*_G (-)$
means ordinary cohomology of the Borel space $(-)_G = EG\times_G (-)$.

The above construction of $L^G_* (X)$ involves choices of models for $EG_k$.
We prove, using our VRR-formula in the submersive case, that $L^G_* (X)$ is independent 
of these choices (Theorem \ref{thm.stageklclassindepofgrpemb}).

The equivariant $L$-class has the following properties:
The top-degree class is the equivariant fundamental class $[X]_G$.
If $X=M$ is smooth, then $L^G_* (M)$ is the equivariant Poincar\'e dual
of the usual equivariant cohomology class $L^*_G (M) = L^* ((TM)_G)$,
where $(TM)_G$ is the homotopy quotient vector bundle of the equivariant tangent 
bundle $TM$ (Theorem \ref{thm.manifoldcase}).
Functoriality in the group variable is established in
Theorem \ref{thm.equivlchangeingroup}:
An inclusion $G' \subset G$ of a closed subgroup induces a
restriction map $H^{G}_* (X;\rat) \to H^{G'}_* (X;\rat)$. We prove that this map
sends $L^{G}_* (X)$ to $L^{G'}_* (X)$.
The class $L^G_* (X)$ contains all of the information of the
Goresky-MacPherson class $L_* (X)$, since the 
restriction map $H^G_* (X;\rat) \to H^{\{ 1 \}}_* (X;\rat) \cong H_* (X;\rat)$ 
induced by the inclusion of
the trivial group $\{ 1 \}$ into $G,$ sends $L^G_* (X)$ to $L_* (X)$
(Proposition \ref{prop.equivlmapstononequivl}).
Equivariant homology satisfies a Künneth theorem,
Proposition \ref{prop.equivariantkunneth}, which asserts that
the cross product induces an isomorphism
$H^{G}_* (X;\rat) \otimes H^{G'}_* (X';\rat)
   \cong H^{G\times G'}_* (X\times X';\rat).$
We show in Theorem \ref{thm.productequivl} that
the equivariant $L$-class satisfies the product formula
$L^{G \times G'}_* (X\times X') 
      = L^G_* (X) \times L^{G'}_* (X')$
under this isomorphism.       
If $G$ acts trivially on $X$, then
the equivariant Künneth theorem shows that
$H^G_* (X;\rat) \cong H^G_* (\pt;\rat) \otimes H_* (X;\rat)$.
In this case, the above product formula implies that
$L^G_* (X) = [\pt]_G \times L_* (X)$,
see Proposition \ref{prop.trivialaction}.
For free actions, there is a canonical identification 
$H^G_* (X;\rat) \cong H_{*-\dim G} (X/G;\rat)$ under which
$L^G_* (X)$ corresponds to the
Goresky-MacPherson class $L_* (X/G)$ of the orbit space
(Theorem \ref{thm.freeaction}).

The $L$-class VRR-formulae of \cite{banaglnyjm} and 
\cite{banaglbundletransfer} have been provided in a
piecewise linear (PL) context. The technical Appendix
\ref{sec.plstructures} establishes this context for the $G$-actions considered
in the present paper. The key point is this: the ambient
action on $M$ is smooth and one may consider the
smooth principal $G$-bundle 
$EG_k \times M \to M_G (k)$. Then one
forgets that this bundle is principal and only considers its
underlying smooth fiber bundle.
Under piecewise differentiable (PD) homeomorphisms of base and total space, the
bundle is a PL bundle, whose fiber is the unique PL manifold
that underlies the smooth manifold $G$ (forgetting the group structure).
We use the fact that Whitney stratified
subsets can be triangulated (Goresky \cite{goreskytriang}, Verona \cite{verona}).

The above construction method for $L^G_*$ is consistent with constructions
of equivariant characteristic classes in complex algebraic geometry.
The equivariant version of Chern-Schwartz-MacPherson classes
was developed by Ohmoto in \cite{ohmoto1} for algebraic actions
of a complex reductive linear algebraic group $\mathbb{G}$
on a possibly singular complex algebraic variety $X$.
Equivariant Todd classes for such actions were introduced by
Brylinski and Zhang in \cite{brylinskizhang}.
Equivariant intersection theory in equivariant Chow groups for actions of 
linear algebraic groups on algebraic spaces has been developed by
Edidin and Graham in \cite{edidingraham}.
The work of Brasselet, Schürmann and Yokura (\cite{bsy}) explains how Hirzebruch's 
generalized Todd class $T^*_y$ can be homologically extended to singular varieties.
In \cite{weber}, Weber developed an equivariant version of that
class for algebraic actions on possibly singular complex algebraic varieties,
with a focus on actions of the torus $\mathbb{G} = (\cplx^*)^r$.
In all of the above complex algebraic situations, the construction rests on an
algebraic approximation of the classifying space $E\mathbb{G}$
used by Totaro in \cite{totaro}.

For the trivial group, the Goresky-MacPherson class $L_* (X)$ is the
Pontrjagin character of a $K$-theoretic orientation class $\Delta (X)
\in \KO_* (X) \otimes \intg [\smlhf]$, 
\cite{banaglko}, \cite{siegel}.
$K$-theoretic questions such as the relation of $L^G_* (X)$ to
$G$-signatures, to equivariant orientation classes $\Delta^G$
in $\KO^G_* [\smlhf]$, and to analytic $\K$-homology classes defined
by the signature operator (Albin, Leichtnam, Mazzeo, Piazza, \cite{almpsigpack})
will be explored elsewhere.
For a finite group $G$ acting on a Witt space $X$
(satisfying weak regularity properties on the fixed point sets),
Cappell, Shaneson and Weinberger indicated the construction of a $G$-equivariant
class $\Delta^G (X)$ and a corresponding $G$-signature theorem in \cite{csw}.
Free group actions on Witt spaces were considered by Curran in \cite{curran}.
While we do not use equivariant Thom-Mather control data in the present paper,
the basic equivariant setup of Whitney (B)-regular stratified spaces
used by Pflaum and Wilkin in \cite{pflaumequivcontrol} served as a 
valuable inspiration.

\section{Stratifications}

The main results of the present paper are established for
Whitney stratified subsets of smooth manifolds.
This concept involves both purely topological assumptions on the stratified
space, and regularity requirements on how the space is embedded.
In this section, we recall the former assumptions, while a review of the latter
requirements are the subject of Section \ref{sec.whitneystratspaces}.
Any closed subspace of a smooth manifold is automatically separable, locally compact
and Hausdorff.
Hence, with a view towards Whitney stratified subsets of manifolds, 
we may focus on spaces with these properties.
Thus let $X$ be a separable locally compact Hausdorff space and
let $P$ be a poset with order relation $<$. Given $i,j\in P$,
we will write $i\leq j$ if $i<j$ or $i=j$.
A \emph{stratification} $\Sa$ of $X$ over $P$ is a partition
$\Sa = \{ S_i ~|~ i\in P \}$ of $X$ into locally closed subspaces
$S_i \subset X$ (called the \emph{strata}) such that\\

\noindent (S1) $\Sa$ is locally finite,

\noindent (S2) $S_i$ is a smooth manifold for every $i\in P$, and

\noindent (S3) $\Sa$ satisfies the \emph{condition of frontier}:
 $S_i$ intersects the closure $\overline{S}_j$ of $S_j$ nontrivially if and only if
 $S_i \subset \overline{S}_j$. This is to be the case if and only if
 $i\leq j$.\\
 
\noindent The pair $(X,\Sa)$ will then be called a \emph{stratified space}.
If $X$ is compact, then (S1) implies that $\Sa$ contains only finitely many
strata: Every point has an open neighborhood that intersects only finitely
many strata, and $X$ is covered by a finite number of such neighborhoods.
If $\Sa$ is a stratification of $X$ and $N$ a nonempty smooth manifold, then
$\Sa \times N$ will denote the partition $\{ S_i \times N ~|~ i\in P \}$
of $X\times N$. The proof of the following lemma is a straightforward
verification of the stratification axioms (S1) -- (S3).

\begin{lemma} \label{lem.productstrat}
The partition $\Sa \times N$ is a stratification of $X\times N$.
\end{lemma}

The stratification $\Sa \times N$ is called the
\emph{product stratification} on $X\times N$.

\section{Whitney Stratified Spaces}
\label{sec.whitneystratspaces}

In order to recall Whitney's condition of \emph{(B)-regularity}, we begin
by setting up notation concerning the parallel translation of 
planes in Euclidean space. 
If $a,b \in \real^k$ are points in Euclidean space, we shall write
$[a,b]$ for the line segment consisting of the points $ta+(1-t)b$, $t\in [0,1]$.
If $a\not= b$, then the parallel translate of the affine line
$\{ ta + (1-t)b ~|~ t\in \real \}$ to the origin $0\in \real^k$ will be
denoted by $G[a,b]$. We may then regard $G[a,b]$
as a linear subspace of $\real^k$, or as a point of
projective space $\real \pr^{k-1}$.
If we wish to emphasize the ambient dimension, we shall also write
$G_k [a,b]$ for $G[a,b]$.
As a one-dimensional linear subspace of $\real^k$, $G[a,b]$ is spanned by
the nonzero vector $b-a$.
The following simple observation will assist us in proving 
Lemma \ref{lem.productwhitneystrat} on the product of a 
Whitney stratified subset with a manifold factor.
\begin{lemma} \label{lem.projoflinesegments}
Let $\pi: \real^k \times \real^j \to \real^k$ be the projection onto the
first factor. Let $a,b\in \real^{k+j} = \real^k \times \real^j$ be points
such that $b-a \not\in \ker \pi$. Then
\[ \pi [a,b] = [\pi (a), \pi (b)]~ \text{ and }~
  G_k [\pi (a), \pi (b)] = \pi (G_{k+j} [a,b]). \]
\end{lemma}
More generally,
a $d$-dimensional \emph{affine plane} in Euclidean space $\real^k$ is a subset
$P\subset \real^k$ which can be written in the form
$P = p + P_0,$
where $p\in \real^k$ and $P_0 \subset \real^k$ is a linear subspace of
dimension $d$.
The linear subspace $P_0$ is uniquely determined by $P$, the point $p$ is not.
The \emph{Gauss map} $G_k$, defined on the set of 
$d$-dimensional affine planes in $\real^k,$ 
is given by $G_k (P) := P_0.$
We may regard the translate $G_k (P)$ as a linear subspace of $\real^k$ or
as a point of the Grassmannian $G(d,k)$ of $d$-planes through the origin in $\real^k$.
If $P$ happens to be a linear subspace of $\real^k$, then
$G_k (P)=P.$ In particular, $G_k (\real^k)=\real^k.$
The behavior of Gauss maps under splittings, described in the following lemma,
will be useful in establishing both Lemma \ref{lem.productwhitneystrat} and 
Theorem \ref{thm.orbitspaceoffreeisbregular} on orbit spaces of free actions.
\begin{lemma} \label{lem.gausssplitting}
If $P_j \subset \real^j$ and  
$P_n \subset \real^n$ are affine planes, then $P = P_j \times P_n$ is an
affine plane in $\real^k = \real^j \times \real^n$ and
$G_{j+n} (P) = G_j (P_j) \times G_n (P_n).$
\end{lemma}

\begin{example} \label{exple.gaussmapsofaffinelines}
Let $\ell \subset \real^j \times \real^n$ be an affine line
of the form
$\ell = \{ 0_{\real^j} \} \times \ell_n,$
where $\ell_n$ is an affine line in $\real^n$.
Then by Lemma \ref{lem.gausssplitting},
\[ G_{j+n} (\ell) = G_j (\{ 0_{\real^j} \}) \times G_n (\ell_n) 
   = \{ 0_{\real^j} \} \times G_n (\ell_n). \]
\end{example}

\begin{example} \label{exple.gaussmapsofsplitaffineplanes}
Let $T \subset \real^j \times \real^n$ be an affine plane
of the form
$T = \real^j \times T_n,$
where $T_n$ is an affine plane in $\real^n$.
Then by Lemma \ref{lem.gausssplitting},
\[ G_{j+n} (T) = G_j (\real^j) \times G_n (T_n) 
   = \real^j \times G_n (T_n). \]
\end{example}
The tangent space to a smooth manifold $M$ at a point $p\in M$ will
be denoted by $T_p M$.

Suppose that $(X,\Sa)$ is a stratified space whose underlying topological space
$X$ is a closed subset $X\subset M$ of a smooth manifold $M$ of dimension $n$ such 
that the strata of $\Sa$ are smooth submanifolds of $M$.
Let $R,S$ be two different strata in $\Sa$ such that
$R \subset \overline{S}$ and let $x$ be a point in $R$.
Recall that the pair $(R,S)$ is said to be \emph{Whitney (B) regular at $x$},
if there exists a smooth chart 
$\chi: U \to \real^n$ of $M$ around $x$ such that
the following condition holds:
Whenever 
\begin{itemize}
\item $(x_k)$ is a sequence of points in $R \cap U$ and 
\item $(y_k)$ is a sequence of points in $S \cap U$,
\end{itemize}
such that\\

\noindent (R1) $x_k \to x$ and $y_k \to x$ as $k\to \infty$,

\noindent (R2) $G[\chi (x_k), \chi (y_k)]$ converges in 
  $\real \pr^{n-1}$ to a point $\ell$, and

\noindent (R3) $G(T_{\chi (y_k)} \chi (S\cap U))$ converges in
 the Grassmannian $G(d,n)$, $d=\dim S,$ to a point $\tau$,\\

\noindent then
\[ \ell \subset \tau, \]
where both $\ell$ and $\tau$ are regarded as linear subspaces of $\real^n$.
(Note that in the above context, $x_k \not= y_k$ for every $k$ since
$R$ and $S$ are disjoint.)
If $(R,S)$ is Whitney $(B)$-regular at $x$ with respect to some chart $\chi$,
then it is Whitney $(B)$-regular with respect to any other smooth chart
(Pflaum \cite[Lemma 1.4.4]{pflaumhabil}) around $x$.
This is an important principle that we will rely on in the proof
of Theorem \ref{thm.orbitspaceoffreeisbregular}.
We say that $\Sa$ is \emph{Whitney (B)-regular} with respect to 
$X\subset M$, if $(R,S)$ is Whitney (B)-regular at every point $x\in R$
for every pair $(R,S)$ of different strata in $\Sa$ such that
$R \subset \overline{S}$. In this case, the pair
$(X,\Sa)$ is also called a \emph{Whitney stratified subset of $M$}.

Since every stratum of a Whitney stratified subset of $M$ is a smooth submanifold
of $M$, the set of dimensions of strata is bounded by the dimension of $M$.
We say that a Whitney stratified subset $(X,\Sa),$ $X\subset M,$ 
has \emph{dimension} $n$,
if $\dim S =n$ for some $S\in \Sa$ and $\dim S \leq n$ for every $S\in \Sa$.
For a general stratified space $(X,\Sa)$, there is no useful link
between the condition of frontier and the dimensions of strata.
Simple examples, arising for instance from the topologist's
sine curve, show that there may be strata $R,S$ with
$R\subset \overline{S} -S$, yet $\dim R = \dim S$ or possibly even
$\dim R > \dim S$ (\cite[p. 18, 1.1.12]{pflaumhabil}).
Mather noticed that this phenomenon cannot occur in the presence 
of (B)-regularity:

\begin{lemma} \label{lem.matherfrontierdimbound}
(Mather \cite[p. 6, Prop. 2.5; p. 14]{mather})
Let $(X,\Sa)$ be a Whitney stratified subset of $M$. If $R,S\in \Sa$
are strata with $R\subset \overline{S} - S$, then
$\dim R < \dim S.$
\end{lemma}
Let $(X,\Sa)$ be a Whitney stratified subset of $M$.
For $d\in \nat$, we define its \emph{$d$-skeleton} to be
\[ X_d := \bigcup \{ S \in \Sa ~|~ \dim S \leq d \}. \]
Mather's lemma implies the closedness of skeleta:
\begin{lemma} \label{lem.skeletaclosed}
Suppose $(X,\Sa)$ is a Whitney stratified subspace of $M$.
Then the $d$-skeleta $X_d$ are closed subsets of $X$.
\end{lemma}
For Whitney stratified spaces $(X,\Sa)$ of dimension $n$, the complement
of the $(n-1)$-skeleton,
$X^\reg := X - X_{n-1},$
of $X$ will be called the \emph{regular set} of $(X,\Sa)$.
Lemma \ref{lem.skeletaclosed} shows that the regular set is open.

It is well-known that Whitney stratified subsets are
stable under taking products (Trotman \cite[p. 7]{trotman}).
For the sake of completeness and for the reader's convenience, we 
provide a detailed proof in the case where one of the factors
is a manifold.
\begin{lemma} \label{lem.productwhitneystrat}
Let $(X,\Sa),$ $X \subset M,$ be a Whitney stratified subset of a smooth manifold $M$
and let $E$ be any smooth manifold. Then 
$(X\times E, \Sa \times E),$
$X\times E \subset M\times E$, is a Whitney stratified subset of the smooth
manifold $M\times E$.
\end{lemma}
\begin{proof}
By Lemma \ref{lem.productstrat},
the partition $\Sa \times E$ is a stratification of the
separable, locally compact Hausdorff space $X\times E$.
It remains to verify that the product strata fit together in
a Whitney (B)-regular manner.
Let $R\times E,$ $S\times E$ be two different strata in $\Sa \times E$
such that $R\times E \subset \overline{S\times E} = \overline{S} \times E$
and let $(x,e)$ be a point in $R\times E$.
Since $R\subset \overline{S}$, we may use (B)-regularity of $(R,S)$ at $x$
to obtain a smooth chart $\chi: U \to \real^n$ of $M$ around $x$
such that whenever
$(x_k)$ is a sequence of points in $R \cap U$ and 
$(y_k)$ a sequence of points in $S \cap U$ such that (R1) -- (R3) hold,
we have $\ell \subset \tau$.

Let $\epsilon: V \to \real^d$ be a smooth chart in $E$ about $e$, $d=\dim E$.
A smooth chart $\widetilde{\chi}: \widetilde{U} \to \real^{n+d}$
in $M\times E$ around $(x,e)$
is given by $\widetilde{U} := U \times V$, $\widetilde{\chi} = \chi \times \epsilon$.
We shall verify (B)-regularity in this chart.
Let 
\begin{itemize}
\item $(x_k, a_k)$ be a sequence of points in $(R\times E) \cap \widetilde{U}$ and 
\item $(y_k, b_k)$ a sequence of points in $(S\times E) \cap \widetilde{U}$
\end{itemize}
such that\\

\noindent (R1) $(x_k, a_k) \to (x,e)$ and $(y_k, b_k) \to (x,e)$ as $k\to \infty$,

\noindent (R2) $G[\widetilde{\chi} (x_k, a_k), 
                  \widetilde{\chi} (y_k, b_k)]$ converges in 
  $\real \pr^{n+d-1}$ to a point $\widetilde{\ell}$, and

\noindent (R3) $G(T_{\widetilde{\chi} (y_k, b_k)} 
                \widetilde{\chi} ((S\times E)\cap \widetilde{U}))$ converges in
 the Grassmannian $G(s+d,n+d)$, $s=\dim S,$ to a point $\widetilde{\tau}$.\\

We must show that $\widetilde{\ell} \subset \widetilde{\tau}$.
Let $\pi: \real^n \times \real^d \to \real^n$ denote the projection onto the
first factor.
Let $\Sigma \subset G(s+d, n+d)$ be the Schubert cycle
$\Sigma := \{ P ~|~ P \supset \ker \pi \}.$
We claim that the kernel $\{ 0_n \} \times \real^d$ of $\pi$ is contained
in $\widetilde{\tau}$, i.e. that $\widetilde{\tau} \in \Sigma$.
Indeed, writing
\begin{align*}
T_{\widetilde{\chi} (y_k, b_k)} 
                \widetilde{\chi} ((S\times E)\cap \widetilde{U})
&= T_{\widetilde{\chi} (y_k, b_k)} 
                \widetilde{\chi} ((S\cap U) \times V) \\                
&= T_{(\chi (y_k), \epsilon (b_k))} 
                (\chi (S\cap U) \times \epsilon (V)) \\
&= T_{\chi (y_k)} (\chi (S\cap U)) \times 
     T_{\epsilon (b_k)} \real^d,                
\end{align*}
we may apply Lemma \ref{lem.gausssplitting} to obtain
\begin{align*} 
G_{n+d} (T_{\widetilde{\chi} (y_k, b_k)} 
                \widetilde{\chi} ((S\times E)\cap \widetilde{U}))
&= G_{n+d} (T_{\chi (y_k)} (\chi (S\cap U)) \times 
     T_{\epsilon (b_k)} \real^d) \\
&= G_n (T_{\chi (y_k)} (\chi (S\cap U))) \times \real^d,                     
\end{align*}
which contains $\{ 0_n \} \times \real^d$ for every $k$.
In other words,
$G_{n+d} (T_{\widetilde{\chi} (y_k, b_k)} 
                \widetilde{\chi} ((S\times E)\cap \widetilde{U}))
       \in \Sigma$
for all $k$.
Therefore, the limit $\widetilde{\tau}$ contains
$\{ 0_n \} \times \real^d = \ker \pi$ as well because
$\Sigma$ is closed in $G(s+d, n+d)$.
This establishes the claim.

In the case where $\widetilde{\ell} \subset \ker \pi$, the claim implies that
$\widetilde{\ell} \subset \widetilde{\tau}$, as was to be shown.
Let us then assume that $\widetilde{\ell} \not\subset \ker \pi$.
In this case, the image
\[ \ell := \pi (\widetilde{\ell}) \]
under the linear map $\pi$ is a one-dimensional linear subspace of $\real^n$.
By (R1) above, the sequences $(x_k) \subset R\cap U$ and $(y_k) \subset S\cap U$ 
both converge to $x\in R$
(by continuity of the first factor projection).
This establishes (R1) for $(x_k),$ $(y_k)$.
We will verify that condition (R2) holds for $(x_k),$ $(y_k)$ with respect to the
line $\ell$, that is, we will show that
$G_n [\chi (x_k), \chi (y_k)]$ converges in $\real \pr^{n-1}$ to $\ell$.
Indeed, by definition of the chart $\widetilde{\chi},$
$\pi \widetilde{\chi} (x_k, a_k) = \chi (x_k)$
and
$\pi \widetilde{\chi} (y_k, b_k) = \chi (y_k).$
Since $\ker \pi = \{ 0_n \} \times \real^d$ is closed in $\real^n \times \real^d$,
the condition $\widetilde{\ell} \not\subset \ker \pi$
is an open condition, i.e. $\widetilde{\ell}$ has an open
neighborhood in $\real \pr^{n+d-1}$ such that every line in that
neighborhood is not contained in $\ker \pi$.
Thus there is a $k_0$ such that
\[ \widetilde{\chi} (y_k, b_k) - \widetilde{\chi} (x_k, a_k) \not\in \ker \pi \]
for all $k \geq k_0$.
So for these $k$, we may apply 
Lemma \ref{lem.projoflinesegments} to obtain
\[
[\chi (x_k), \chi (y_k)]
 = [\pi \widetilde{\chi} (x_k, a_k), \pi \widetilde{\chi} (y_k, b_k)]
 = \pi [\widetilde{\chi} (x_k, a_k), \widetilde{\chi} (y_k, b_k)]
\]
and
\[ G_n [\chi (x_k), \chi (y_k)]
   = G_n [\pi \widetilde{\chi} (x_k, a_k), \pi \widetilde{\chi} (y_k, b_k)] 
   = \pi (G_{n+d} [\widetilde{\chi} (x_k, a_k), \widetilde{\chi} (y_k, b_k)]). \]
By continuity of the projection $\pi$, the right hand side converges
to $\pi (\widetilde{\ell}) = \ell$ as $k\to \infty$. 
Hence the left hand side converges to $\ell$, which proves (R2)
for $(x_k),$ $(y_k)$ and $\ell$.

We move on to verifying (R3) for $(y_k)$.
The projection $\pi: \real^n \times \real^d \to \real^n$
induces a map
\[ \pi_\Sigma: \Sigma \longrightarrow G(s,n),~ P \mapsto \pi (P).  \]
For if $P\in \Sigma$, then we may choose a direct sum
decomposition $P = (\ker \pi) \oplus P_+$. The dimension of 
$P_+$ is $s$
and the restriction of $\pi$ to $P_+$ is injective.
Therefore, $\pi (P_+)$ is an $s$-dimensional linear subspace of $\real^n$.
Furthermore, $\pi (P_+) = \pi (P)$. Hence $\pi_\Sigma$ exists as claimed.
Since $\pi_\Sigma$ is continuous, the sequence of image points
\[ \pi_\Sigma (G_{n+d} (T_{\widetilde{\chi} (y_k, b_k)} 
                \widetilde{\chi} ((S\times E)\cap \widetilde{U}))) \]
converges to 
$\pi_\Sigma (\widetilde{\tau})$ in $G(s,n)$ as $k\to \infty$.
Now, we showed above that
\begin{equation} \label{equ.gndtwcgntctrd} 
G_{n+d} (T_{\widetilde{\chi} (y_k, b_k)} 
                \widetilde{\chi} ((S\times E)\cap \widetilde{U}))
    = G_n (T_{\chi (y_k)} (\chi (S\cap U))) \times \real^d. 
\end{equation}    
Applying $\pi$, we get
\begin{align*} 
\pi_\Sigma G_{n+d} (T_{\widetilde{\chi} (y_k, b_k)} 
                \widetilde{\chi} ((S\times E)\cap \widetilde{U}))
&= \pi_\Sigma (G_n (T_{\chi (y_k)} (\chi (S\cap U))) \times \real^d) \\
&= G_n (T_{\chi (y_k)} (\chi (S\cap U))).
\end{align*}    
We conclude that $G_n (T_{\chi (y_k)} (\chi (S\cap U)))$
converges to 
\[ \tau := \pi_\Sigma (\widetilde{\tau}) \]
in $G(s,n)$ as $k\to \infty$.
So the data $(x_k),$ $(y_k)$, $\ell$ and $\tau$ satisfy
(R1) -- (R3).

From (B)-regularity of $(R,S)$ at $x$ we deduce that
$\ell \subset \tau$, i.e. 
$\pi (\widetilde{\ell}) \subset \pi (\widetilde{\tau})$.
Identity (\ref{equ.gndtwcgntctrd}) implies that
\begin{align*}
G_{n+d} (T_{\widetilde{\chi} (y_k, b_k)} 
                \widetilde{\chi} ((S\times E)\cap \widetilde{U}))
 &= \pi^{-1} (G_n (T_{\chi (y_k)} (\chi (S\cap U)))) \\
 &= \pi^{-1} (\pi G_{n+d} (T_{\widetilde{\chi} (y_k, b_k)} 
                \widetilde{\chi} ((S\times E)\cap \widetilde{U})))                
\end{align*}
This shows that in the limit
$\widetilde{\tau} = \pi^{-1} (\pi (\widetilde{\tau})).$
Consequently,
\[ \widetilde{\ell} \subset
   \pi^{-1} (\pi (\widetilde{\ell}))
   \subset \pi^{-1} (\pi (\widetilde{\tau})) = \widetilde{\tau},
\]
as was to be shown.
\end{proof}

\section{Free Actions on Whitney Stratified Spaces}

The main purpose of the present section is to explain that
if a compact Lie group $G$ acts smoothly and freely on a manifold 
which contains a $G$-invariant Whitney stratified subspace $X$,
then the orbit space $X/G$ is a Whitney stratified subspace of
the smooth manifold $M/G$ (Theorem \ref{thm.orbitspaceoffreeisbregular}).

Let $G$ be a compact Lie group and let
$(X,\Sa)$ be a stratified space.
We say that an action of $G$ on $X$ is
\emph{compatible with the stratification $\Sa$}, or
\emph{stratum preserving}, if every stratum
$S\in \Sa$ is a $G$-invariant subspace and the induced action $G\times S \to S$ of
$G$ on $S$ is smooth.
In this case we will also say that $\Sa$ is a \emph{$G$-stratification}
and that $(X,\Sa)$ is a \emph{$G$-stratified space}.

Suppose that $M$ is a smooth $G$-manifold and
$(X,\Sa),$ $X\subset M,$ is a Whitney stratified subset
such that $X$ is $G$-invariant and the induced action of $G$
on $X$ is compatible with the stratification $\Sa$.
Let $\Sa/G$ denote the collection of orbit spaces
$\{ S/G ~|~ S\in \Sa \}$. 
Since every stratum $S$ is nonempty, the spaces $S/G$ are not empty.
If for some $R,S\in \Sa$, $R/G$ and $S/G$ have a point in common,
then $R$ and $S$ have a point in common and are therefore equal.
Thus $R/G = S/G$. Furthermore, any point of $X/G$ lies in some $S/G$
since every point of $X$ lies in some stratum $S$. 
Thus $\Sa/G$ is a set-theoretic partition of $X/G$. 
 
\begin{lemma} \label{lem.closureofstratummodg}
The closure $\overline{S}$ in $X$ of a stratum $S\in \Sa$ is
a $G$-invariant subspace of $X$ and
the closure of $S/G \in \Sa/G$ in $X/G$
can be computed by
\begin{equation} \label{equ.closuresmodg}
 \overline{S/G} = \overline{S}/G.  
\end{equation}
\end{lemma}
\begin{proof}
We show first that $\overline{S}$
is indeed a $G$-invariant subspace of $X$:
Let $s$ be a point of the closure $\overline{S}$ and $g\in G$.
For $gs$ to lie in $\overline{S}$, it is necessary and sufficient
that every open neighborhood $V\subset X$ of $gs$ has nonempty intersection
with $S$. The set $U:= g^{-1} V$ is open and
\[ s = g^{-1} (gs) \in g^{-1} V =U. \]
Thus $U$ intersects $S$ nontrivially and there exists a point
$s' \in U\cap S$. The point $gs'$ lies in $S$, since $S$ is $G$-invariant.
Moreover, $gs' \in g (g^{-1} V)=V$. We conclude that 
$V\cap S$ is not empty, as was to be shown.
Thus $\overline{S}$ is $G$-invariant and $\overline{S}/G$ is defined.
We prove the equality claimed in (\ref{equ.closuresmodg}). 
Since $G$ is compact, the orbit projection
$\pi: X\to X/G$ is closed (\cite[p. 38, Thm. 3.1.(2)]{bredontransfgroups}).
Consequently, $\pi (\overline{S})$ is closed in $X/G$.
As $\overline{S}$ is $G$-invariant, $\pi (\overline{S}) = \overline{S}/G$.
Therefore, $\overline{S}/G$ is closed and since it contains $S/G$, it must
also contain the closure of $S/G$,
\[ \overline{S/G} \subset \overline{S}/G. \]
To establish the converse inclusion, let
$Gs_0 \in \overline{S}/G$ be any point, $s_0 \in \overline{S}$.
This point is in $\overline{S/G}$ precisely when
every open neighborhood $U \subset X/G$ of $Gs_0$ intersects
$S/G$ nontrivially.
The set $V:= \pi^{-1} (U)$ is open in $X$ and
$s_0 \in Gs_0 \subset \pi^{-1} (U) =V.$
Since $s_0 \in \overline{S}$, its open neighborhood $V$ must intersect
$S$ nontrivially, i.e. there exists a point $s_1 \in V\cap S$.
Its orbit $\pi (s_1) = Gs_1$ is a point of $U$, and a point of $S/G$
(as $S$ is $G$-invariant).
Hence $U \cap (S/G)$ is not empty. This places $Gs_0$ in $\overline{S/G}$ and
finishes the proof of (\ref{equ.closuresmodg}).
\end{proof}

\begin{lemma} \label{lem.smodgisstratn}
If $G$ acts freely (and smoothly) on $M$, then
the partition $\Sa/G$ is a stratification of $X/G$.
\end{lemma}
\begin{proof}
Since $G$ is compact, the orbit projection
$\pi: M\to M/G$ is closed.
Consequently, $\pi (X)$ is closed in $M/G$.
Since $X$ is $G$-invariant, $\pi (X)=X/G$.
Therefore,  
the closed equivariant embedding $X\subset M$ induces a
closed embedding $X/G \subset M/G$, with $M/G$ a smooth manifold,
as $G$ is compact and acts freely and smoothly.
This implies in particular that $X/G$ is separable, locally compact, and Hausdorff,
since these properties are inherited from the smooth manifold $M/G$.

We prove next that the subspaces $S/G,$ $S\in \Sa,$ are locally closed in $X/G$.
We will use the following criterion:
A subspace is locally closed if and only if it is open in its closure.
As $S$ is locally closed in $X$ by assumption,
$S$ is an open subset of $\overline{S}$.
Since $\overline{S}$ is $G$-invariant by Lemma \ref{lem.closureofstratummodg},
there is an orbit projection
$\pi: \overline{S} \to \overline{S}/G$, which is an open map.
We conclude that $\pi (S)$ is open in $\overline{S}/G$.
Now $\pi (S)=S/G$ (since $S$ is $G$-invariant) and
$\overline{S}/G = \overline{S/G}$ by (\ref{equ.closuresmodg}) of the Lemma.
Hence $S/G$ is open in $\overline{S/G}$, as was to be verified.

It remains to check the stratification conditions (S1) -- (S3).
To verify (S1) we must show that $\Sa/G$ is locally finite.
Let $Gx$ be a point of $X/G$, $x\in X$. 
Using (S1) for $\Sa$, we find an open neighborhood $U\subset X$
of $x$ in $X$ which intersects nontrivially only 
$S_{i_1},\ldots, S_{i_k}\in \Sa$, that is,
$U\cap S_j = \varnothing$ for all $j\in J := P- \{ i_1, \ldots, i_k \}$.
Since $GU \subset X$ is the union of the open sets $gU$, $g\in G$,
it is open in $X$. Since $GU$ is also a union of orbits, we can
regard it as a subset $GU \subset X/G$ of the orbit space of $X$.
As such, it is also open, since the quotient map $X\to X/G$ is an
open map. The point $Gx$ is contained in $GU$. Thus the latter constitutes
an open neighborhood of $Gx$ in $X/G$.
We claim that $GU$ intersects only finitely many strata in $\Sa /G$.
Indeed, suppose that 
$(GU) \cap (S_j/G) \not= \varnothing$ for some $j\in J$.
Then there is some $u\in U$ and $s \in S_j$ such that
$Gu = Gs$. Thus $u=gs$ for some $g\in G$. This would imply that
$u\in S_j$, since $S_j$ is $G$-invariant. We thus arrive at a contradiction
to $U\cap S_j = \varnothing$. Therefore, $(GU) \cap (S_j/G) = \varnothing$ 
for all $j\in J$, showing that $\Sa/G$ is locally finite.

For (S2), we have to show that $S/G$ is a smooth manifold for every
$S\in \Sa$. By compatibility of the action with $\Sa$,
the induced action of $G$ on $S$ is smooth. Since the induced action
is also free and $G$ is compact, $S/G$ is a manifold with a unique
smooth structure such that the quotient map $S \to S/G$ is a smooth
submersion. In fact, $S/G$ is then a smooth submanifold of $M/G$.

Lastly, we verify the condition of frontier (S3)
using Lemma \ref{lem.closureofstratummodg}.
Suppose then that $R/G$ intersects the closure 
of $S/G$ nontrivially, $R,S\in \Sa$.
Then by (\ref{equ.closuresmodg}),
$(R/G) \cap (\overline{S}/G) \not= \varnothing.$
There exists thus a point
$r\in R$ and a point $s_0 \in \overline{S}$ such that
$Gr = Gs_0 \in (R/G) \cap (\overline{S}/G)$.
Some group element $g\in G$ moves $r$ to $s_0$, $gr=s_0$.
As $R$ is $G$-invariant, $gr \in R$. Therefore, $s_0 \in R\cap \overline{S}$.
By the condition of frontier for $\Sa$,
$R$ is contained in $\overline{S}$.
This implies $R/G \subset \overline{S}/G = \overline{S/G}$.
\end{proof}

As pointed out earlier, the closed equivariant inclusion $X\subset M$ induces a 
closed inclusion $X/G \subset M/G$.
If $G$ acts freely (and smoothly), then $M/G$ is a manifold with a unique
smooth structure such that the quotient map $\pi: M \to M/G$ is
a smooth submersion (Lee \cite[p. 218, Theorem 9.16]{lee}).
In fact, $\pi$ is a smooth principal $G$-bundle
(Tu \cite[p. 22, Theorem 3.3]{tuintrolectequivcoh}).

\begin{thm} \label{thm.orbitspaceoffreeisbregular}
Let $G$ be a compact Lie group which acts smoothly and freely on
a smooth manifold $M$. Suppose that
$(X,\Sa),$ $X\subset M,$ is a Whitney stratified subset
such that $X$ is $G$-invariant and the induced action of $G$
on $X$ is compatible with the stratification $\Sa$.
Then $(X/G, \Sa/G)$, $X/G \subset M/G$, is a Whitney stratified subset.
\end{thm}
\begin{proof}
By Lemma \ref{lem.smodgisstratn}, $\Sa/G$ is a stratification of $X/G$.
Let $R,S\in \Sa$ be strata, $R\not= S,$ such that
$R/G \subset \overline{S/G}$
in $X/G,$ and let $\xi$ be a point of $R/G$. 
We shall prove that $(R/G, S/G)$ is Whitney (B)-regular at $\xi$.

To begin with, we shall construct a suitable chart
$\chi: U \to \real^n$ of $M/G$ around $\xi$, where $n=\dim (M/G)$.
The regularity is subsequently to be verified in this chart.
The smooth submersion $\pi: M \to M/G$ is proper, as $G$ is compact.
Hence, by Ehresmann's fibration theorem, $\pi$ is smoothly locally trivial, i.e.
there exists an open neighborhood
$U \subset M/G$ of $\xi$ and a diffeomorphism
$\phi: \pi^{-1} (U) \to G \times U$ such that
\[ \xymatrix{
\pi^{-1} (U) \ar[rd]_{\pi|} \ar[rr]^\phi_\cong & & G\times U \ar[ld]^{\pro_2} \\
& U &
} \]
commutes.
The inverse image $\pi^{-1} (U)$ is $G$-invariant. 
Making $U$ smaller if necessary, we may assume that there is a 
diffeomorphism $\chi: U \to \real^n$ with $\chi (\xi)=0$.

Now let
\begin{itemize}
\item $(\xi_k)$ be a sequence of points in $(R/G) \cap U$ and 
\item $(\eta_k)$ a sequence of points in $(S/G) \cap U$,
\end{itemize}
such that (R1) -- (R3) hold in the chart $\chi$, that is,\\

\noindent (R1) $\xi_k \to \xi$ and $\eta_k \to \xi$ as $k\to \infty$,

\noindent (R2) $G[\chi (\xi_k), \chi (\eta_k)]$ converges in 
  $\real \pr^{n-1}$ to a point $\ell$, and

\noindent (R3) $G(T_{\chi (\eta_k)} \chi ((S/G)\cap U))$ converges in
 the Grassmannian $G(d,n)$, $d=\dim S/G,$ to a point $\tau$.\\

We must show that
$\ell \subset \tau,$
where both $\ell$ and $\tau$ are regarded as linear subspaces of $\real^n$.
The idea will be to lift the points $\xi, \xi_k, \eta_k$ to points in $M$
in such a way that (B)-regularity for $(X,\Sa)$, $X\subset M,$ can be applied.
Then we will deduce the desired containment relation from the corresponding
relation obtained upstairs.
Let $e\in G$ denote the neutral element and set
\[ x := \phi^{-1} (e,\xi),~
   x_k := \phi^{-1} (e,\xi_k),~
   y_k := \phi^{-1} (e,\eta_k). \]
Since $\pi \circ \phi^{-1} = \pro_2$,
the points $x, x_k, y_k$ are indeed lifts with respect to $\pi$ of
the corresponding points $\xi, \xi_k, \eta_k$.
Since $\pi^{-1} (R/G) = R$ and $\pi^{-1} (S/G)=S$, this shows
that $x, x_k \in R$ and $y_k \in S$.

We shall next construct a smooth chart
$\widetilde{\chi}: \widetilde{U} \to \real^{n+j},$ $j=\dim G,$
$\widetilde{U} \subset M$ open,
wherein we will then verify that (R1) -- (R3) hold for $x, (x_k)$ and $(y_k)$.
Let $\gamma: U_e \to \real^j$ be a smooth chart in $G$ around $e\in G$,
$\gamma (e)=0$.
The set
$\widetilde{U} := \phi^{-1} (U_e \times U)$
is open in $M$ and contains $x$, $(x_k)$ and $(y_k)$. 
The diffeomorphism $\widetilde{\chi}$ is constructed as the composition
\[
\widetilde{U} \stackrel{\phi}{\longrightarrow} 
  U_e \times U \stackrel{\gamma \times \chi}{\longrightarrow}
  \real^j \times \real^n = \real^{j+n}.
\] 
It maps $x$ to the origin.
Let us verify (R1) -- (R3) for $x, (x_k)$ and $(y_k)$ in the chart
$\widetilde{\chi}$.

As for (R1), we need to show that the sequences $(x_k)$ and $(y_k)$
both converge to $x$.
This follows from the continuity of $\phi^{-1}$, observing that
$(e,\xi_k) \to (e,\xi)$ and $(e,\eta_k) \to (e,\xi)$ in
$U_e \times U$.

To establish (R2), we shall prove that
$G[\widetilde{\chi} (x_k), \widetilde{\chi} (y_k)]$ converges in 
$\real \pr^{n+j-1}$ to a point $\widetilde{\ell}$.
Indeed,
\begin{align*}
[\widetilde{\chi} (x_k), \widetilde{\chi} (y_k)]
&= [\widetilde{\chi} (\phi^{-1} (e,\xi_k)), \widetilde{\chi} (\phi^{-1} (e,\eta_k))] 
  = [(\gamma \times \chi) (e,\xi_k)), (\gamma \times \chi)(e,\eta_k))] \\
&= [(\gamma (e), \chi (\xi_k)), (\gamma (e), \chi (\eta_k))] 
  = \{ \gamma (e) \} \times [\chi (\xi_k), \chi (\eta_k)] \\
&= \{ 0_{\real^j} \} \times [\chi (\xi_k), \chi (\eta_k)]
\end{align*}
and thus
\[
G [\widetilde{\chi} (x_k), \widetilde{\chi} (y_k)]
= \{ 0_{\real^j} \} \times G[\chi (\xi_k), \chi (\eta_k)].
\]
(Parallel translating the secant $[\chi (\xi_k), \chi (\eta_k)]$
in $\real^n$ to $0_{\real^n}$ and then regarding it as a line in
$\real^{j+n}$ under the linear standard embedding
$\real^n = \{ 0_{\real^j} \} \times \real^n \subset \real^j \times \real^n$,
is the same as parallel translating
the secant $\{ 0_{\real^j} \} \times [\chi (\xi_k), \chi (\eta_k)]$
in $\real^{j+n}$ to $0_{\real^{j+n}}$.
See Example \ref{exple.gaussmapsofaffinelines}.)
As $G[\chi (\xi_k), \chi (\eta_k)]$ converges to $\ell$,
$\{ 0_{\real^j} \} \times G[\chi (\xi_k), \chi (\eta_k)]$ converges to
\begin{equation} \label{equ.welliszerotimesell}
\widetilde{\ell} := \{ 0_{\real^j} \} \times \ell.  
\end{equation}
Therefore, $G [\widetilde{\chi} (x_k), \widetilde{\chi} (y_k)]$
converges to $\widetilde{\ell}$.

We turn to (R3), where we need to explain that
$G(T_{\widetilde{\chi} (y_k)} \widetilde{\chi} (S\cap \widetilde{U}))$ converges in
the Grassmannian $G(d,n+j)$, $d=\dim S,$ to a point $\widetilde{\tau}.$
We shall first compute what the manifold $S\cap \widetilde{U}$
looks like in the coordinate system $\widetilde{\chi}$.
We claim that
\begin{equation} \label{equ.chitildesintutildeisrjtimeschismodgintu}
\widetilde{\chi}(S \cap \widetilde{U})
 = \real^j \times \chi ((S/G) \cap U).
\end{equation}
To prove this claim,
let $(g,\zeta)$ be a point in $U_e \times ((S/G)\cap U).$
Then $\phi^{-1} (g,\zeta) \in \phi^{-1} (U_e \times U) = \widetilde{U}$
and
$\pi \phi^{-1} (g,\zeta) = \pro_2 (g,\zeta) = \zeta \in S/G.$
Hence,
$\phi^{-1} (g,\zeta) \in \pi^{-1} (S/G) \cap \widetilde{U},$
from which we deduce that
\[ (g,\zeta) = \phi(\phi^{-1} (g,\zeta)) \in
  \phi (\pi^{-1} (S/G) \cap \widetilde{U}). \]
Therefore,
\[ U_e \times ((S/G)\cap U) \subset 
   \phi (\pi^{-1} (S/G) \cap \widetilde{U}). \]  
Conversely, suppose that
$(g,\zeta) \in \phi (\pi^{-1} (S/G) \cap \widetilde{U})$.
Then
$\phi^{-1} (g,\zeta) \in \pi^{-1} (S/G) \cap \widetilde{U}$,
so that
$\zeta = \pro_2 (g,\zeta) = \pi \phi^{-1} (g,\zeta) \in S/G.$
Furthermore, $\zeta \in U$, as $\im \phi = G\times U$.
Therefore, $\zeta \in (S/G) \cap U$.
Since
\[ \phi^{-1} (g,\zeta) \in
    \pi^{-1} (S/G) \cap \widetilde{U} \subset \widetilde{U}
    = \phi^{-1} (U_e \times U), \]
the pair $(g,\zeta)$ is a point of $U_e \times U$. In particular,
we find that $g\in U_e$. We conclude that
$(g,\zeta) \in U_e \times ((S/G)\cap U)$.
We have shown that
\[
   U_e \times ((S/G)\cap U) 
   = \phi (\pi^{-1} (S/G) \cap \widetilde{U}). 
\]
Using this equation, we compute
\begin{align*}
\widetilde{\chi} (S\cap \widetilde{U})
&= (\gamma \times \chi) \phi (S\cap \widetilde{U}) 
 = (\gamma \times \chi) \phi (\pi^{-1} (S/G) \cap \widetilde{U}) \\
&= (\gamma \times \chi) (U_e \times ((S/G)\cap U)) 
  = \gamma (U_e) \times \chi ((S/G) \cap U) \\
&= \real^j \times \chi ((S/G) \cap U),
\end{align*}
as claimed. This establishes (\ref{equ.chitildesintutildeisrjtimeschismodgintu}).
Using this description of $\widetilde{\chi} (S\cap \widetilde{U})$,
we compute the tangent space at $\widetilde{\chi} (y_k)$:
\begin{align*}
T_{\widetilde{\chi} (y_k)} \widetilde{\chi} (S\cap \widetilde{U})
&= T_{\widetilde{\chi} (\phi^{-1} (e,\eta_k))} (\real^j \times \chi ((S/G)\cap U))
 = T_{(\gamma (e), \chi (\eta_k))} (\real^j \times \chi ((S/G)\cap U))\\
&= T_{(0, \chi (\eta_k))} (\real^j \times \chi ((S/G)\cap U))
 = T_0 (\real^j) \times T_{\chi (\eta_k)} \chi ((S/G)\cap U).
\end{align*}
Applying the Gauss map $G$ in $\real^j \times \real^n,$ the planes
\begin{align*}
G_{j+n}(T_{\widetilde{\chi} (y_k)} \widetilde{\chi} (S\cap \widetilde{U}))
&= G_{j+n} (T_0 (\real^j) \times T_{\chi (\eta_k)} \chi ((S/G)\cap U)) \\
&= \real^j \times G_n (T_{\chi (\eta_k)} \chi ((S/G)\cap U))
\end{align*}
(see Example \ref{exple.gaussmapsofsplitaffineplanes})
converge to the plane
\begin{equation} \label{equ.wtrjtimest}
\widetilde{\tau} := \real^j \times \tau, 
\end{equation}
using (R3) for $(\eta_k)$ in the chart $\chi$.
This finishes the verification of (R1) -- (R3) for $x, (x_k), (y_k)$
in the coordinate system $\widetilde{\chi}$.

The relation $R/G \subset \overline{S/G}$ implies that
$R \subset \overline{S}$.
By assumption, $(R,S)$ is (B)-regular at $x\in R$.
This means that there exists \emph{some} chart around $x$ in which,
whenever (R1) -- (R3) are satisfied,
the limit of secants is contained in the limit of tangent planes
associated to $x,(x_k),(y_k)$. But then this is true in \emph{any}
chart around $x$, in particular in the chart $\widetilde{\chi}$.
Having already verified (R1) -- (R3) in the coordinate system 
$\widetilde{\chi}$, we therefore conclude that
$\widetilde{\ell} \subset \widetilde{\tau}.$
Thus by (\ref{equ.welliszerotimesell}) and (\ref{equ.wtrjtimest}),
$\{ 0_{\real^j} \} \times \ell \subset \real^j \times \tau.$
Applying the standard projection
$\pro_2: \real^j \times \real^n \to \real^n,$ we have
\[ \pro_2 (\{ 0_{\real^j} \} \times \ell) = \ell,~
   \pro_2 (\real^j \times \tau) = \tau  \]
and thus
\[ \ell = \pro_2 (\{ 0_{\real^j} \} \times \ell) \subset 
   \pro_2 (\real^j \times \tau) = \tau \]
in $\real^n$, as was to be shown.
\end{proof}

A Whitney stratified space $(X,\Sa)$ of dimension $n$
is called a \emph{pseudomanifold} if
$X^\reg$ is dense in $X$ and whenever a stratum in $\Sa$ is not
contained in $X^\reg$, its dimension is at most $n-2$.

\begin{lemma} \label{lem.quotientoffreeispseudomfd}
We adopt the assumptions of Theorem \ref{thm.orbitspaceoffreeisbregular};
in particular $G$ acts freely.
If $(X,\Sa)$ is a pseudomanifold, then $(X/G, \Sa/G)$ is a pseudomanifold.
\end{lemma}
\begin{proof}
Let $n$ denote the dimension of $X$ and $j$ the dimension of $G$.
Thus there is an $n$-dimensional stratum $S$ of $X$. By freeness
(and compactness of $G$), the quotient
$S/G$ is a smooth manifold of dimension $n-j$. If $R/G$ is any
stratum of $X/G$, then $\dim R \leq n$, so $\dim (R/G)\leq n-j.$
This shows that $\dim (X/G) = n-j.$

We prove that the regular part of $X/G$ is dense in $X/G$:
Let $y$ be a point in $X/G$ and let
$\pi: X \to X/G$ denote the canonical quotient map.
Choose a point $x\in X$ such that $\pi (x)=y$.
As $(X,\Sa)$ is a pseudomanifold, the regular part $X^\reg$ is
dense in $X$.
Thus there is a sequence $(x_i) \subset X^\reg$ of regular points
such that $(x_i)$ converges to $x$ as $i\to \infty$.
As $\pi$ is continuous, $y_i := \pi (x_i)$ converges to $\pi (x)=y$.
We claim that the points $y_i$ are contained in $(X/G)^\reg$.
Fix $i$.
As $x_i$ lies in $X^\reg$, there is a stratum $S\in \Sa$ such that
$\dim S=n$ and $x_i\in S$. 
Then $y_i = \pi (x_i) \in \pi (S) = S/G$ and $\dim (S/G)=n-j$ by freeness.
Since $\dim (X/G)=n-j$, this shows that $y_i$ is in $(X/G)^\reg$.
We have shown that $(X/G)^\reg$ is dense in $X/G$.

Let $S/G \in \Sa/G$ be a stratum of $X/G$ which is not contained in $(X/G)^\reg$.
Then $\dim (S/G) < \dim (X/G)=n-j$.
By freeness, $\dim (S/G) = (\dim S)-j$.
Therefore, $(\dim S)-j < n-j$, i.e. $\dim S<n$.
In particular, $S$ is not contained in $X^\reg$.
Since $(X,\Sa)$ is a pseudomanifold, we have $\dim S \leq n-2$
and thus $\dim (S/G) = (\dim S)-j \leq n-j-2$, as was to be shown.
\end{proof}

\section{Whitney Stratified Approximations to the Borel Construction}
\label{sec.whitneyapproxtoborel}

Suppose that $G$ is a compact Lie group which acts smoothly on
a manifold $M$ and that
$(X,\Sa),$ $X\subset M,$ is a Whitney stratified subset
such that $X$ is $G$-invariant and the induced action of $G$
on $X$ is compatible with the stratification $\Sa$.
Let $EG \to BG$ be a universal $G$-bundle.
Borel's homotopy quotient $X_G = EG \times_G X$, giving rise to
equivariant cohomology, is generally infinite dimensional.
Our construction of equivariant $L$-classes requires approximations
$X_G (k)$ of $X_G$, $k=1,2,\ldots,$ that are finite dimensional.
Beyond finite dimensionality, the demands on the approximations are the
following:
If $X$ is compact, then every $X_G (k)$ should be compact.
The corresponding approximations $M_G (k)$ of $M_G = EG \times_G M$
should be finite-dimensional smooth manifolds such that
$X_G (k) \subset M_G (k)$ is a Whitney stratified subset.
Moreover, if $X$ is a pseudomanifold that satisfies the Witt condition,
then $X_G (k)$ should inherit these properties.
Such approximations will be constructed in the present section.
They are based on compact smooth Stiefel manifolds $EG_k$, 
endowed with a free and smooth $G$-action, that serve
as finite dimensional $(k-1)$-connected approximations to the free contractible
$G$-space $EG$ as $k\to \infty$.
The desired approximations of $X_G$ will then have the form
$X_G (k) = EG_k \times_G X$. The key point is to show that they are Whitney stratified
in the smooth manifold $M_G (k) = EG_k \times_G M$
(Theorem \ref{thm.xgkwhitneystratpsdmfdwitt}).
We will also prove that $X_G (k)$ inherits pseudomanifold and Witt properties
from $X$.\\

We start with some general preliminaries.
Let $G$ be a Lie group and $K\subset H\subset G$ closed subgroups.
The multiplication $G\times G\to G$ restricts to a free right action
$G\times H\to G$. 
The notation $G/K$ will mean the set of \emph{left} cosets $gK$ of $K$ in $G$.
Suppose that $H$ is compact and $K$ is normal in $H$.
Then $H/K$ is a Lie group,
the free right action $G\times H\to G$ descends to a free right action
$(G/K) \times (H/K) \to G/K,$ and
the projection $G/K \to G/H$ is a smooth (locally trivial) principal
$(H/K)$-bundle. Note that $G/H$ is diffeomorphic to $(G/K)/(H/K)$.

Let $G$ be a compact Lie group of dimension $d$.
There exists an $n$ such that $G$ can be embedded in the orthogonal
group $\Or (n)$ as a closed subgroup $G\subset \Or (n)$.
We choose and fix such an embedding.
We use the convention $\Or (0) = \{ 1 \}$, the trivial group. This allows
us to take $n=0$ when $G$ is the trivial group.
The sequence
$\Or_k := \Or (n+k),$ $k=1,2,3,\ldots,$
of orthogonal groups is a nested sequence
\[ \Or_1 \subset \Or_2 \subset \Or_3 \subset \cdots \]
of closed subgroups by mapping $A\in \Or_k$ to
\[   
\begin{pmatrix} A & 0 \\ 0 & 1
\end{pmatrix} \in \Or_{k+1}.
\]
Using the standard closed embedding $\Or (n) \times \Or (k) \subset \Or (n+k)$
given by
\[
(A,B) \mapsto \begin{pmatrix} A & 0 \\ 0 & B \end{pmatrix} =: A\oplus B,
\]
the diagram of closed subgroup inclusions
\[ \xymatrix@R=15pt{
\Or (n+k) \ar@{^{(}->}[r] & \Or (n+k+1) \\
\Or (n) \times \Or (k) \ar@{^{(}->}[u] \ar@{^{(}->}[r]
& \Or (n) \times \Or (k+1) \ar@{^{(}->}[u] \\
G \times \Or (k) \ar@{^{(}->}[u] \ar@{^{(}->}[r]
& G \times \Or (k+1) \ar@{^{(}->}[u]
} \]
commutes.
Let $K_k$ be the closed subgroup of $\Or_k$ given by
$K_k := 1_n \times \Or (k) \subset \Or_k$ and
set 
\[ EG_k := \Or_k / K_k. \] 
Homogeneous spaces obtained as quotients of a Lie group by a closed subgroup
are smooth manifolds such that
the quotient map is smooth and the action of the Lie group on the quotient
is smooth. Thus $EG_k$ is a smooth $\Or_k$-manifold.
It is compact, as $\Or_k$ is compact.
In fact, $EG_k = V_n (\real^{n+k})$
is the Stiefel manifold of orthonormal $n$-frames in $\real^{n+k}$ 
and its dimension is
\[ \dim EG_k = nk + \smlhf n(n-1). \]
For example, when $n=1$,  
$EG_k = V_1 (\real^{1+k}) = S^k$ and
when $n=2$,
$EG_k = V_2 (\real^{2+k}) = S (TS^{k+1})$ is
the unit sphere bundle of the tangent bundle of the $(k+1)$-sphere,
$\dim EG_k =2k+1.$
Let $H_k \subset \Or_k$ be the closed subgroup given by
\[  H_k := G \times \Or (k) \subset 
      \Or (n) \times \Or (k) \subset \Or (n+k) = \Or_k. \]
The subgroup $H_k$ contains $K_k$ as a normal subgroup.     
Applying the above preliminary considerations to
the closed subgroups
$K_k \lhd H_k < \Or_k,$
we deduce that
$H_k /K_k$ is a Lie group, the free action
$\Or_k \times H_k \to \Or_k$ descends to a free action
$(\Or_k/K_k) \times (H_k/K_k) \to \Or_k/K_k,$ and
the projection 
\[ p_k: EG_k = \Or_k/K_k \to \Or_k/H_k \] 
is a smooth locally trivial principal
$(H_k/K_k)$-bundle. We note that $\Or_k/H_k$ is diffeomorphic to 
$(\Or_k/K_k)/(H_k/K_k)$. The structure group $H_k /K_k$ is
\[ H_k / K_k = (G\times \Or (k))/(1_n \times \Or (k)) = G. \]
In particular, $EG_k$ is a free $G$-space. The notation $EG_k$
indicates that $EG_k$ consists of an underlying space together
with a particular $G$-action.
If two groups $G,G'$ are both embedded in $\Or (n)$,
then the underlying spaces of $EG_k$ and $EG'_k$ are equal,
but the notation reminds us that $EG_k$ is to be considered as a free $G$-space,
while $EG'_k$ is to be considered as a free $G'$-space.
The base space
\[ BG_k := \Or_k / H_k  \]
is a compact smooth manifold.
Thus we obtain a smooth, locally trivial principal $G$-bundle
\[ p_k: EG_k \longrightarrow BG_k, \]
whose total and base space are both compact smooth manifolds.
This principal bundle structure also shows that
\[ 
\dim BG_k = \dim EG_k - \dim G = nk + \smlhf n(n-1) - d,~ d=\dim G.
\]
We shall henceforth briefly write
\[ a := \smlhf n(n-1) - d. \]

\begin{remark} \label{rem.orthoandspecialortho}
The canonical embedding
$\SO (n+k) \subset \Or (n+k)$ induces for $k\geq 1$ a 
smooth identification
\[ \frac{\SO (n+k)}{1_n \times \SO (k)}
   \cong
   \frac{\Or (n+k)}{1_n \times \Or (k)} = EG_k. \]
If $G$ is a closed subgroup of $\SO (n)$, 
then this identification is $G$-equivariant, 
and one obtains a smooth identification
\[ \frac{\SO (n+k)}{G \times \SO (k)}
   \cong
   \frac{\Or (n+k)}{G \times \Or (k)} = BG_k. \]
Since $\Or (n)$ can be embedded into $\SO (n+1)$
(via $A\mapsto A\oplus (\det A)$), we could therefore also work with embeddings 
of $G$ into special orthogonal groups.
\end{remark}

The connectivity of Stiefel manifolds is well-known
(see e.g. Husemoller \cite[Ch. 8.11, p. 103]{husemoller}):
\begin{lemma} \label{lem.egkiskm1conn}
The manifold $EG_k$ is $(k-1)$-connected.
\end{lemma}
Since the relation $K_k = \Or_k \cap K_{k+1}$ holds in $\Or_{k+1}$
(Mimura, Toda \cite[p. 88]{mimuratoda}),
there is an induced submanifold inclusion
\[ EG_k = \Or_k / K_k 
    = \Or_k / (\Or_k \cap K_{k+1}) \subset \Or_{k+1} / K_{k+1} = EG_{k+1},  \]
which is $G$-equivariant, since
\[ (A_{n+k} \cdot (g_n \oplus 1_k)) \oplus 1_1
   = (A_{n+k} \oplus 1_1)\cdot (g_n \oplus 1_{k+1}) \]
for any $A_{n+k} \in \Or (n+k),$ $g_n \in G \subset \Or (n)$.   
The codimension of this inclusion is independent of $k$, in fact
\[
 \dim EG_{k+1} - \dim EG_k = n.
\]
Similarly, as the relation $H_k = \Or_k \cap H_{k+1}$ holds in $\Or_{k+1}$,
there is an induced submanifold inclusion
\[ \beta_k: BG_k = \Or_k / H_k 
    = \Or_k / (\Or_k \cap H_{k+1}) \subset \Or_{k+1} / H_{k+1} = BG_{k+1}  \]
such that there is a commutative diagram
\begin{equation}  \label{dia.egktoegkplusone}
\xymatrix{
EG_k \ar[d]_{p_k} \ar@{^{(}->}[r] & EG_{k+1} \ar[d]^{p_{k+1}} \\
BG_k \ar@{^{(}->}[r]^{\beta_k} & BG_{k+1}
} 
\end{equation}
Since the horizontal inclusion of total spaces is $G$-equivariant, this
constitutes a morphism of principal $G$-bundles and the diagram is cartesian.
The codimension of the inclusion $BG_k \subset BG_{k+1}$ equals
the codimension of $EG_k$ in $EG_{k+1},$
\[
 \dim BG_{k+1} - \dim BG_k =n.
\]
As $k\to \infty$, Lemma \ref{lem.egkiskm1conn} shows that the bundles 
$p_k$ may serve as a
finite-dimensional smooth approximation to the 
universal principal $G$-bundle $EG \to BG$ by taking
\[ EG := \bigcup_{k=1}^\infty EG_k,~
   BG := \bigcup_{k=1}^\infty BG_k. \]

If $E$ and $X$ are left $G$-spaces, then the \emph{diagonal action}
of $G$ on the product $E\times X$ is the left action given by
$g \cdot (e,x) = (ge,gx),$ $g\in G,$ $e\in E,$ $x\in X$.
If $E$ is a \emph{right} $G$-space and $X$ a left $G$-space, then
$E$ becomes a left $G$-space by $ge := eg^{-1}$ so that the
above formula applies, in which case we have
$g \cdot (e,x) = (ge,gx) = (eg^{-1}, gx)$.
The twisted product of $E$ and $X$ is the orbit space
$E \times_G X := (E\times X)/G$ under the diagonal action.
If $G$ acts freely on $E$, then the diagonal action on $E\times X$
is free as well.

Let $G$ be a compact Lie group which acts smoothly from the left on
a manifold $M$. 
As $G$ acts freely and smoothly on $EG_k$, its diagonal action on the 
product manifold $M' := EG_k \times M$ is also free and smooth.
Hence, as $G$ is compact, the orbit space
\[ M_G (k) := EG_k \times_G M = M'/G  \]
is a smooth manifold.
Suppose that
$(X,\Sa),$ $X\subset M,$ is a Whitney stratified subset
such that $X$ is $G$-invariant and the induced action of $G$
on $X$ is compatible with the stratification $\Sa$.
The diagonal action of $G$ on $M'$ leaves the subspace 
$X' := EG_k \times X \subset M'$
invariant. The restriction of this action to $X'$ is the
diagonal action of $G$ on $X'$.
Set
\[ X_G (k) := EG_k \times_G X = X'/G.   \]
Let $q_k: X_G (k) \to BG_k$ be the factorization of the 
composition $EG_k \times X \to EG_k \stackrel{p_k}{\longrightarrow} BG_k$
by the quotient map $EG_k \times X \to X_G (k)$.
Then $(X_G (k), q_k, BG_k, X, G)$ is the fiber bundle with
fiber $X$ associated to the principal bundle
$(EG_k, p_k, BG_k, G)$.
In particular, the projection map $q_k$ is locally trivial.

The equivariant closed inclusion $X\subset M$ induces 
an equivariant closed inclusion $X' \subset M'$ and
a map
$EG_k \times_G X \to EG_k \times_G M,$
which is a closed inclusion
\[ X_G (k) \subset M_G (k). \]
The space $X_G (k)$ is endowed with a stratification as follows:
Let $\Sa' := EG_k \times \Sa$ be the product stratification
of $X'$ (Lemma \ref{lem.productstrat}).
According to Lemma \ref{lem.productwhitneystrat}, 
$(X', \Sa'),$
$X' \subset M',$ is a Whitney stratified subset of $M'$.
We note that the diagonal action on $X'$ is compatible with
the stratification $\Sa'$, since
every stratum $EG_k \times S,$ $S\in \Sa,$ is $G$-invariant
due to the $G$-invariance of $S$, and the induced action
$G\times (EG_k \times S) \to EG_k \times S$ is smooth due
to the smoothness of $G\times S\to S$ (and the smoothness
of $G\times EG_k \to EG_k$).
As $G$ acts freely and smoothly on $M'$,
Lemma \ref{lem.smodgisstratn} shows that
\[ \Sa_G (k) := \Sa'/G \]
is a stratification of $X'/G= X_G (k)$.

Let $\bar{m}$ and $\bar{n}$ denote lower and upper middle perversity
functions in the sense of intersection homology theory.
These are dual to each other and if $(X,\Sa)$ is an oriented Whitney stratified
pseudomanifold of dimension $n$, then the shifted Verdier dual 
$\dual IC^{\bar{m}}_X [n]$ 
of the rational intersection chain complex $IC^{\bar{m}}_X$ of sheaves is
isomorphic to $IC^{\bar{n}}_X$ in the derived category of constructible
complexes of sheaves on $X$.
A Whitney stratified pseudomanifold $(X,\Sa)$ is said
to satisfy the \emph{Witt condition} if
the intersection homology groups $IH^{\bar{m}}_k (L^{2k};\rat)$
of all even dimensional links $L^{2k}$ of (connected components of)
strata vanishes. This holds if and only if the canonical morphism
$IC^{\bar{m}}_X \to IC^{\bar{n}}_X$ is a quasi-isomorphism, in which
case $IC^{\bar{m}}_X$ is self-dual when $X$ is oriented
(Siegel \cite{siegel}, Goresky and MacPherson \cite[5.6.1]{gmih2}).
If $(X,\Sa)$ satisfies the Witt condition, we will call it a
\emph{Witt space}. 
For example, if $\Sa$ contains only strata of even codimension, then $(X,\Sa)$
is evidently a Witt space. Pure-dimensional complex algebraic varieties $X$
can be Whitney stratified in this way and thus are Witt spaces,
since they are also pseudomanifolds.
If $(X,\Sa)$ is a Witt space and $N$ a smooth manifold
without boundary, then $(X\times N, \Sa \times N)$ is a Witt space,
since both stratifications have the same links $L^{2k}$.

\begin{thm} \label{thm.xgkwhitneystratpsdmfdwitt}
Let $G$ be a compact Lie group which acts smoothly on
a smooth manifold $M$. Suppose that
$(X,\Sa),$ $X\subset M,$ is a Whitney stratified subset
such that $X$ is $G$-invariant and the induced action of $G$
on $X$ is compatible with the stratification $\Sa$.
Then
$(X_G (k),\Sa_G (k))$ is a Whitney stratified subspace of
the smooth manifold $M_G (k)$.
If $X$ is a pseudomanifold, then so is $X_G (k)$.
If $X$ is compact, then so is $X_G (k)$.
If $X$ satisfies the Witt condition, then so does $X_G (k)$.
\end{thm}
\begin{proof}
The compact Lie group $G$ acts smoothly and freely on the smooth
manifold $M'$. 
We have already explained above that the pair $(X',\Sa'),$ $X'\subset M',$ 
is a Whitney stratified subset
such that $X'$ is $G$-invariant and the induced action of $G$
on $X'$ is compatible with the stratification $\Sa'$.
By Theorem \ref{thm.orbitspaceoffreeisbregular} on free actions,
$(X'/G, \Sa'/G)$, $X'/G \subset M'/G$, is a Whitney stratified subset.
This subset is nothing but
$(EG_k \times_G X, \Sa'/G)$, $EG_k \times_G X \subset EG_k \times_G M$,
i.e. $(X_G (k), \Sa_G (k))$, $X_G (k) \subset M_G (k)$.

The product of a pseudomanifold with a smooth manifold, equipped with
the product stratification as in Lemma \ref{lem.productstrat}, is a
pseudomanifold.
Thus if $(X,\Sa)$ is a pseudomanifold, then $(X',\Sa')$ is a pseudomanifold,
and $(X'/G, \Sa'/G)$ is a pseudomanifold according to 
Lemma \ref{lem.quotientoffreeispseudomfd}, as $G$ acts freely on $M'$.
If $X$ is compact, then $X_G (K)$ is compact, since it is a quotient of the product
space $X' = EG_k \times X$, both of whose factors are compact.

Suppose $(X,\Sa)$ satisfies the Witt condition. We use the fiber bundle
$q_k: X_G (k) \to BG_k$ with fiber $X$.
As $BG_k$ is a manifold, it can be covered by open sets $U_\alpha$
which are homeomorphic to $\real^{nk+a}$.
Since $q_k$ is locally trivial, there are homeomorphisms 
$q^{-1}_k (U_\alpha) \cong U_\alpha \times X$.
As $(X,\Sa)$ satisfies the Witt condition, 
$(\real^{nk+a} \times X, \real^{nk+a} \times \Sa)$ satisfies the Witt condition.
Since intersection homology is topologically invariant,
$U_\alpha \times X$ and thus $q^{-1}_k (U_\alpha)$ satisfy the Witt condition.
So $X_G (k)$ is covered by open sets $q^{-1}_k (U_\alpha)$, each of which
are Witt. This implies that $X_G (k)$ is Witt.
\end{proof}

The equivariant inclusion $EG_k \hookrightarrow EG_{k+1}$
induces inclusions
\[ M_G (k) = EG_k \times_G M \hookrightarrow EG_{k+1} \times_G M = M_G (k+1) \]
and
\[ \xi_k: X_G (k) = EG_k \times_G X \hookrightarrow EG_{k+1} \times_G X = X_G (k+1), \]
which fit into a commutative square
\[ \xymatrix{
M_G (k) \ar@{^{(}->}[r] & M_G (k+1) \\
X_G (k) \ar@{^{(}->}[u] \ar@{^{(}->}[r] & X_G (k+1). \ar@{^{(}->}[u]
} \]
Thus a model of the Borel homotopy quotient $X_G$ is given by
\[ X_G = EG \times_G X
       = (\bigcup_{k=1}^\infty EG_k) \times_G X
       = \bigcup_{k=1}^\infty (EG_k \times_G X)
       = \bigcup_{k=1}^\infty X_G (k)
\]
and similarly $M_G = \bigcup_k M_G (k)$.

\section{Orientability}
\label{sec.orientability}

Let $X$ be a compact pseudomanifold, Whitney stratified in some ambient manifold.
The Goresky-MacPherson $L$-class $L_* (X)$ is defined only when the
pseudomanifold $X$ is orientable. Thus in order for $L_* (X_G (k))$
to be available, we must ensure that the approximating Borel pseudomanifolds
$X_G (k)$ can be oriented. When $X$ is a point, $X_G (k) = BG_k$,
so $BG_k$ should be an orientable manifold, at least an infinite
subsequence thereof. 
\begin{example}
We consider the group $G=\intg/_2$.
Choosing $n=1$ and the orthogonal embedding 
$\intg/_2 = \{ \pm 1 \} = \Or (1)$, we have
\[ (B\intg/_2)_k = \frac{\Or (1+k)}{\Or (1) \times \Or (k)}
    = \real \mathbb{P}^k, \]
which is orientable for $k$ odd, but not orientable for $k$ even.
As we will see, this is not due to an unsuitability of the group $\intg/_2$,
but rather to a choice of unsuitable embedding of it in an orthogonal group.
\end{example}

Any compact Lie group $G$ can be embedded in some orthogonal group
$\Or (n)$ in such a way that $G$ acts orientation preservingly on the
corresponding $EG_k$, $k\geq 2$. 
(See Proposition \ref{prop.gactsonegkorientpres}.)
We will show (Proposition \ref{prop.biinvorientgrpisorupondeloop}) that if $G$ is 
bi-invariantly orientable, then the corresponding $BG_k$ is
orientable and the fiber bundle $p_k: EG_k \to BG_k$ is orientable.
Note that the underlying fiber bundle of a principal $G$-bundle
is not generally oriented, even though the fiber $G$ is orientable
and, after having chosen a right-invariant orientation of $G$
(which always exists), the action of the structure group $G$ on the fiber $G$
by right translation is orientation preserving. The problem is that
while $G$ acts from the right on the total space, the transition functions
of the bundle act by \emph{left} translation.
(See Formula \ref{equ.transitionfnprincbundlelefttransl} below.)
This explains the
need for bi-invariant orientability of $G$; not every group is
bi-invariantly orientable.
Let $H^\BM_* (-)$ denote Borel-Moore homology.

\begin{lemma} \label{lem.connactionisorientpres}
Let $\overline{N}$ be an oriented manifold without boundary and $H$ a topological
group acting continuously on $\overline{N}$. If $H$ is path connected, then $H$
acts in an orientation preserving manner.
\end{lemma}
\begin{proof}
The oriented manifold $\overline{N}$ has a fundamental class
$[\overline{N}] \in H^\BM_n (\overline{N};\intg),$ $n=\dim \overline{N}$.
An element $h\in H$ induces an automorphism
$h_*: H^\BM_n (\overline{N};\intg) \to H^\BM_n (\overline{N};\intg)$.
As $H$ is path connected, there exists a path 
$\gamma: I=[0,1] \to H$ from the identity $\gamma (0)=e$ to
$\gamma (1)=h$. The map $H: \overline{N} \times I \to \overline{N}$
given by $H(p,t) = \gamma(t)\cdot p$ is an isotopy from
$H_0 = \id_{\overline{N}}$ to $H_1 = h$. Since Borel-Moore homology
is isotopy invariant, we conclude that $h_* = \id$. In particular, 
$h_* [\overline{N}] = [\overline{N}]$, which implies that $h$ preserves
the orientation.
\end{proof}

\begin{prop} \label{prop.gactsonegkorientpres}
For any compact Lie group $G$, there exists an embedding
$G \subset \Or (n)$ such that $G$ acts orientation preservingly on
$EG_k$, $k\geq 2$.
\end{prop}
\begin{proof}
Let $G$ be a compact Lie group.
Choose an embedding $G\subset \Or (n-1)$ for some $n\geq 1$.
Composing with the embedding $\Or (n-1) \subset \SO (n)$
given by
\[ 
A \mapsto \begin{pmatrix} A & 0 \\ 0 & \det A \end{pmatrix},
\]
we obtain an embedding $G\subset \SO (n)$.
Consider the manifold 
$\overline{N} := \SO (n+k) / (1_n \times \SO (k)).$
The group $H := \SO (n)$, and hence $G$, acts freely on the connected manifold
$\overline{N}$.
For $k\geq 2$, the Stiefel manifold $\overline{N}$
is simply connected, so orientable. Choose an orientation of $\overline{N}$.
The group $H$ is path connected. 
Thus, by Lemma \ref{lem.connactionisorientpres}, it
acts in an orientation preserving manner on $\overline{N}$.
In particular, its subgroup $G$ acts orientation preservingly on $\overline{N}$.
By Remark \ref{rem.orthoandspecialortho}, 
$\overline{N}$ is $G$-equivariantly homeomorphic
to $EG_k = \Or (n+k)/(1_n \times \Or (k))$.
\end{proof}

Recall that any Lie group has precisely two right-invariant orientations.
We emphasize that the group $G$ appearing in the following lemma
need not be connected. The case of noncompact $M$ will be relevant for
the ensuing application of the lemma.
\begin{lemma} \label{lem.orientbaseofproduct}
Let $M$ be a connected orientable smooth manifold without boundary
and let $G$ be a compact Lie group.
Let $\omega_G$ be one of the two right-invariant orientations of $G$.
Given any orientation $\omega_{M\times G}$ of $M\times G$, the
following statements are equivalent:
\begin{enumerate}
\item The smooth right action of $G$ on $M\times G$ given by
 $(p,h)\cdot g = (p,hg)$ preserves the orientation $\omega_{M\times G}$.
\item There exists a unique orientation $\omega_M$ of $M$ such that
 the product orientation $\omega_M \times \omega_G$ agrees with $\omega_{M\times G}$.
\end{enumerate}
\end{lemma}
The proof is straightforward and may for example be cast in terms
of fundamental classes in Borel-Moore homology. It is useful to write
these classes as sums over connected components. If $G_0$ is the identity
component of $G$, then one can use the
Künneth theorem for Borel-Moore homology
(Bredon \cite[p. 366, Thm. V.14.4]{bredonsheaftheory}) to obtain
a cross product isomorphism
\[  H^\BM_n (M;\intg) \otimes H_d (G_0;\intg) \stackrel{\times}{\longrightarrow}  
     H^\BM_{n+d} (M\times G_0;\intg) \cong \intg, \]
$n=\dim M,$ $d=\dim G,$ which can be used to construct $\omega_M$.

\begin{defn}
A Lie group is called \emph{bi-invariantly orientable} if it has
an orientation which is simultaneously left- and right-invariant.
\end{defn}
\begin{example}
Discrete groups are bi-invariantly orientable by assigning to each
point the orientation $+1$.
\end{example}
\begin{example}
Abelian Lie groups are bi-invariantly orientable, since every Lie group
has a left-invariant orientation and in the abelian case, left and
right translation by an element coincide.
\end{example}
\begin{example}
The orthogonal group $\Or (2)$ is not bi-invariantly orientable.
The conjugation action $(g,h) \mapsto ghg^{-1}$ of $\Or (2)$ on
itself restricts to an action of $\Or (2)$ on the identity component
$\SO (2) \subset \Or (2)$, which does not preserve any orientation
of $\SO (2)$, since conjugation by the element 
\[ g= \begin{pmatrix} 0 & 1 \\ 1 & 0\end{pmatrix} \in \Or (2) \]
is a reflection on $\SO (2)=S^1$.
\end{example}

The property of bi-invariant orientability can be rephrased as
a property of the adjoint representation:
\begin{prop} \label{prop.biinvorientadjrep}
A Lie group is bi-invariantly orientable if and only if
its adjoint representation is orientation preserving.
\end{prop}
\begin{proof}
Let $G$ be a Lie group with Lie algebra $\mathfrak{g}$.
Choose and fix a left-invariant orientation of $G$.
In particular, this fixes an orientation of $\mathfrak{g}$.
On an element $g\in G$, the adjoint representation 
$\Ad: G \longrightarrow GL(\mathfrak{g})$ is a linear automorphism
$\Ad (g)= \phi_{g*,e}: \mathfrak{g} \to \mathfrak{g}$
given by the differential at $e\in G$ of the conjugation diffeomorphism
$\phi_g:G\to G,$ $\phi_g (h)=ghg^{-1}$.
With $L_g, R_g: G\to G$ denoting left and right translation by $g$,
the conjugation can be written as $\phi_g = L_g \circ R_{g^{-1}}$.
Its differential at $e$ is thus
$\phi_{g*,e} = L_{g*, g^{-1}} \circ R_{g^{-1}*, e}$.

Now suppose that a left-invariant orientation of $G$ can be
chosen in addition to be right-invariant. Then both
$L_{g*, g^{-1}}$ and $R_{g^{-1}*, e}$ are orientation preserving.
Thus their composition $\Ad (g) = \phi_{g*,e}$ is orientation preserving.
Conversely, if $\phi_{g*,e}$ is orientation preserving, then,
for a choice of left-invariant orientation of $G$,
$R_{g^{-1}*, e} = L_{g^{-1}*, e} \circ \phi_{g*,e}$ preserves the
orientation for every $g$. Thus the left-invariant orientation is
also right-invariant.
\end{proof}

\begin{prop} \label{prop.conngrpsarebiinvorient}
Connected Lie groups are bi-invariantly orientable.
\end{prop}
\begin{proof}
Let $G$ be a connected Lie group. An element $g\in G$ defines
the conjugation $\phi_g: G\to G,$ $\phi_g (h)=ghg^{-1}$.
As $G$ is connected, there exists a smooth path
$\gamma: I\to G$ from $\gamma (0)=e$ to $\gamma (1)=g$.
An isotopy $\Phi: G\times I \to G$ from 
$\Phi_0 = \id_G$ to $\Phi_1 = \phi_g$ is given by
$\Phi (h,t) = \gamma (t)h \gamma (t)^{-1}$.
Note that $\Phi (e,t) = e$ for all $t\in I$.
Thus the differential $\Gamma (t) := (\Phi_t)_{*,e}$ of $\Phi_t$ at $e\in G$ 
defines a path $\Gamma: I\to \operatorname{GL} (\mathfrak{g})$
from $\Gamma (0) = (\Phi_0)_{*,e} = (\id_G)_{*,e} = \id_{\mathfrak{g}}$
to $\Gamma (1) = (\Phi_1)_{*,e} = (\phi_g)_{*,e} = \Ad (g)$.
Therefore, $\Ad (g)$ lies in the identity component of 
$\operatorname{GL} (\mathfrak{g})$. Thus $\Ad (g)$ has positive determinant,
which shows that the adjoint representation preserves the orientation.
By Proposition \ref{prop.biinvorientadjrep}, $G$ is bi-invariantly orientable.
\end{proof}

In light of the above examples and results, the next proposition
applies in particular to finite groups, abelian groups, and connected groups.
\begin{lemma} \label{lem.totalspaceorientbaseorient}
Let $H$ be a bi-invariantly orientable compact Lie group, 
let $N,\overline{N}$ be smooth manifolds without boundary, and
let $p: \overline{N} \to N$ be a smooth principal $H$-bundle.
If the manifold $\overline{N}$ is oriented and $H$
acts in an orientation preserving manner on $\overline{N}$, 
then $p,$ together with a bi-invariant orientation of $H,$
induces an orientation on $N$.
\end{lemma}
\begin{proof}
Our convention is that $H$ acts on the right on $\overline{N}$.
We fix a bi-invariant orientation $\omega_H$ of $H$.
Let $\{ U_\alpha \}$ be a good open cover of $N$ by charts $U_\alpha \cong \real^n$
so that the principal $H$-bundle trivializes over $U_\alpha$. 
(Recall that an open cover is \emph{good} if all nonempty finite intersections
of sets in the cover are diffeomorphic to $\real^n$; such a cover always exists.)
Let
$\phi_\alpha: p^{-1} (U_\alpha) \stackrel{\simeq}{\longrightarrow}
U_\alpha \times H$ be an $H$-equivariant diffeomorphism such that
\[ \xymatrix{
p^{-1} (U_\alpha) \ar[rd]_p \ar[rr]^{\phi_\alpha} & & U_\alpha \times H \ar[ld] \\
& U_\alpha &
} \]
commutes. Here $H$ acts on the product $U_\alpha \times H$ by
$(x,h)\cdot g = (x, hg),$ $x\in U_\alpha,$ $g,h \in H$
(as in Lemma \ref{lem.orientbaseofproduct}).
The preimage $p^{-1} (U_\alpha)$ receives an orientation as an open subset
of the oriented manifold $\overline{N}$.
Let $\omega_{U_\alpha \times H}$ be the unique orientation on
$U_\alpha \times H$ that renders $\phi_\alpha$ orientation preserving.
Since $p^{-1} (U_\alpha)$ is $H$-invariant, it is an $H$-space and, 
by assumption, $H$ acts on it in an orientation preserving way.
Since $\phi_\alpha$ is $H$-equivariant and orientation preserving,
it follows that $H$ acts on $U_\alpha \times H$ orientation preservingly
and statement (1) of Lemma \ref{lem.orientbaseofproduct} holds.
By this Lemma, there exists
a unique orientation $\omega_{U_\alpha}$ of $U_\alpha$ such that
$\omega_{U_\alpha} \times \omega_H = \omega_{U_\alpha \times H}$.
(This uses only the right-invariance of $\omega_H$.)
  
Suppose that $V := U_\alpha \cap U_\beta$ is not empty (and thus $V\cong \real^n$).
We claim that $\omega_{U_\alpha}$ agrees with $\omega_{U_\beta}$ on $V$.
Once this is known, the various $\omega_{U_\alpha}$ glue to yield a global
orientation of $N$. Let us then prove the claim.
The $H$-equivariant diffeomorphism
\[ \gamma := \phi_\beta \phi^{-1}_\alpha: V \times H \to V\times H \] 
is of the form
$\gamma  (x,h) = (x, \tau (x,h))$  
with smooth transition function $\tau: V\times H \to H$.
Since $\gamma$ is $H$-equivariant, the transition function satisfies
$\tau (x, hg) = \tau (x,h)\cdot g.$
In particular, for $h=e$,
$\tau (x, g) = \tau (x,e)\cdot g.$   
Let $\sigma: V \to H$ be the function $\sigma (x) := \tau (x,e)$.
Then $\tau$ has the form
\begin{equation} \label{equ.transitionfnprincbundlelefttransl}  
\tau (x, h) = \sigma (x)\cdot h 
\end{equation}
for all $x\in V,$ $h\in H$.
In the commutative diagram
\[ \xymatrix{
  &  (V \times H,\omega_{U_\alpha \times H}|) \ar[dd]^\gamma \\
(p^{-1} (V),\omega_{\overline{N}}|) \ar[ru]^{\phi_\alpha} \ar[rd]_{\phi_\beta} & \\
  & (V \times H,\omega_{U_\beta \times H}|),
} \]
the two slanted arrows are orientation preserving. Hence the vertical
arrow is orientation preserving. If we knew that
$\gamma: (V\times H,\omega_{U_\alpha \times H}|) \to 
  (V\times H,\omega_{U_\alpha \times H}|)$
were also orientation preserving, then it would follow that
\[ \omega_{U_\beta} \times \omega_H = \omega_{U_\beta \times H}| 
  = \omega_{U_\alpha \times H}| = \omega_{U_\alpha} \times \omega_H, \]
which implies $\omega_{U_\beta} = \omega_{U_\alpha}$ on $V$
by the uniqueness in statement (2) of Lemma \ref{lem.orientbaseofproduct}.
Thus it remains to prove that 
$\gamma: (V\times H,\omega_{U_\alpha \times H}|) \to 
  (V\times H,\omega_{U_\alpha \times H}|)$
is orientation preserving.

Using the identification $V\cong \real^n$, the space $V$ inherits from
$\real^n$ a smooth scalar multiplication $\real \times V \to V$, 
$(t,x) \mapsto tx$, $x\in V$, $t\in \real$. The origin in $\real^n$
corresponds to a point in $V$, which we will also call $0\in V$.
We use this multiplication to define a smooth homotopy
\[ \Gamma: V\times H \times I \longrightarrow
       V\times H,~
    \Gamma (x,h,t) := (x, \sigma (tx)\cdot h).  \]
For every $t\in I$, the map
$\Gamma_t: V\times H \to V\times H$ is a diffeomorphism with
inverse $(x,h) \mapsto (x, \sigma (tx)^{-1} h)$.
Thus $\Gamma$ is in fact an isotopy between
$\Gamma_1 = \gamma$ and $\Gamma_0$ given by
$\Gamma_0 (x,h) = (x,\sigma (0)\cdot h)$.
The latter map can be written as a product
$\Gamma_0 = \id_V \times L_{\sigma (0)}$, where
$L_{\sigma (0)}: H \to H$ is left translation by $\sigma (0)$.
The diffeomorphism $\gamma$ induces an isomorphism
\[ \gamma_*: H^\BM_{n+d} (V\times H;\intg) \stackrel{\simeq}{\longrightarrow} 
   H^\BM_{n+d} (V\times H;\intg)  \]
on Borel-Moore homology, $d=\dim H$. Let
$[\omega_{U_\alpha \times H}] \in H^\BM_{n+d} (V\times H)$ 
denote the fundamental class defined by the orientation 
$\omega_{U_\alpha \times H}$ (restricted to $V\times H$).
Under the cross product
\[ H^\BM_n (V;\intg) \otimes H_d (H;\intg) \stackrel{\times}{\longrightarrow}
   H^\BM_{n+d} (V\times H;\intg),   \]
the equation
$[\omega_{U_\alpha \times H}] = [\omega_{U_\alpha}] \times [\omega_{H}]$
holds by construction of $\omega_{U_\alpha}$.
Although Borel-Moore homology is not invariant under general homotopies,
it \emph{is} an isotopy-invariant. Therefore,
\[ \gamma_* = \Gamma_{1*} = \Gamma_{0*} = (\id_V \times L_{\sigma (0)})_*  \]
and thus
\begin{align*}
\gamma_* [\omega_{U_\alpha \times H}]
&= (\id_V \times L_{\sigma (0)})_* [\omega_{U_\alpha \times H}] \\
&= (\id \times L_{\sigma (0) *}) ([\omega_{U_\alpha}] \times [\omega_H]) 
  = [\omega_{U_\alpha}] \times L_{\sigma (0) *} [\omega_H].
\end{align*}
Since $\omega_H$ is also \emph{left}-invariant, 
 $L_{\sigma (0) *} [\omega_H] = [\omega_H]$.
We conclude that
\[
\gamma_* [\omega_{U_\alpha \times H}] = 
  [\omega_{U_\alpha}] \times [\omega_H] = [\omega_{U_\alpha \times H}],
\]
i.e. $\gamma: (V\times H,\omega_{U_\alpha \times H}|) \to 
  (V\times H,\omega_{U_\alpha \times H}|)$ is indeed orientation preserving,
as was to be shown.

We remark on the side that we may alternatively
analyze the differential $\gamma_*$ of $\gamma$ at a point $(x,h)\in V\times H$
in order to establish that $\gamma$ preserves $\omega_{U_\alpha \times H}$.
To compute this differential,
let $\Delta: V \to V\times V$ denote the diagonal map,
let $\mu: H \times H \to H$ denote the multiplication map,
let $L_h: H\to H$ denote left multiplication by $h\in H$, and
let $R_h: H\to H$ denote right multiplication by $h$.
The differential of $\mu$ at $(g,h)\in H\times H$ is given by
\[ \mu_{*, (g,h)} (v, w) = R_{h*} (v) + L_{g*} (w),~
    v \in T_g H,~ w \in T_h H. \]
We observe that $\gamma$ can be factored as
\[ \xymatrix@C=40pt{
V \times H \ar[drr]_\gamma \ar[r]^{\Delta \times 1_H} & 
 V\times V\times H \ar[r]^{1_V \times \sigma \times 1_H} 
  & V \times H \times H \ar[d]^{1_V \times \mu} \\
& & V \times H.   
} \]
Using this factorization, one finds on a tangent vector
$(u,w) \in T_x V \oplus T_h H,$
\begin{equation} \label{equ.differentialofgamma}
\gamma_{*,(x,h)} (u,w) = (u,~ R_{h*} (\sigma_{*,x} (u)) + L_{\sigma (x) *} (w)). 
\end{equation}

Choose a positively oriented ordered basis $B$ of the oriented vector space
$(T_x V \oplus T_h H, \omega_{U_\alpha} \times \omega_H)$ as follows:
Let $\{ e_1,\ldots, e_n \}$ be an ordered basis of $T_x V$ which is
positively oriented with respect to $\omega_{U_\alpha}$, and
let $\{ f_1,\ldots, f_d \}$ be an ordered basis of $T_h H$ which is
positively oriented with respect to $\omega_H$.
Then 
\[
B := \{ (e_1,0),\ldots, (e_n,0),  (0,f_1),\ldots, (0,f_d) \}
\]
is an ordered basis of $T_x V \oplus T_h H$ which is positively oriented
with respect to $\omega_{U_\alpha} \times \omega_H$.
Since the left translation $L_{\sigma (x)}: H\to H$ is
orientation preserving, the ordered basis
\[ \{ L_{\sigma (x)*,h} (f_1),\ldots, L_{\sigma (x)*,h} (f_d) \} \]
of $T_{\sigma (x)h} H$ is positively oriented with respect to $\omega_H$.
Thus 
\[
B' := \{ (e_1,0),\ldots, (e_n,0),  
  (0,L_{\sigma (x)*,h} (f_1)),\ldots, (0,L_{\sigma (x)*,h} (f_d)) \}
\]
is an ordered basis of $T_x V \oplus T_{\sigma (x)h} H$ which is positively oriented
with respect to $\omega_{U_\alpha} \times \omega_H$.
The differential $\gamma_{*,(x,h)}$ is orientation preserving if and only if
the determinant of the matrix representation $M$ of $\gamma_{*,(x,h)}$
with respect to the bases $B,B'$ is positive.
Using (\ref{equ.differentialofgamma}),
the $(n+d)\times (n+d)$-matrix $M$ has the block form
\[ M = \begin{pmatrix}
  1_n & 0_{n\times d} \\ * & 1_d
  \end{pmatrix}, \] 
which indeed has determinant $1>0$, as was to be shown.
\end{proof}

\begin{prop} \label{prop.biinvorientgrpisorupondeloop}
Let $G$ be a compact Lie group.
If $G$ is bi-invariantly orientable, then there exists an embedding
$G \subset \Or (n)$ such that the bundle $p_k: EG_k \to BG_k$
is orientable as a fiber bundle and
$BG_k$ is orientable as a manifold for every $k\geq 2$.
\end{prop}
\begin{proof}
According to Proposition \ref{prop.gactsonegkorientpres},
there exists an embedding
$G \subset \Or (n)$ such that $G$ acts orientation preservingly on
the oriented manifold $EG_k$, $k\geq 2$.
Using the bi-invariant orientability of $G$,
an application of Lemma \ref{lem.totalspaceorientbaseorient} 
to the principal $G$-bundle $p_k:\overline{N}=EG_k \to BG_k=N$
yields an orientation on $N=BG_k$.
Now, a fiber bundle of manifolds whose fiber, base and total space are all orientable
is orientable as a bundle. Thus $p_k$ is orientable as a bundle.
\end{proof}

\begin{example}
Let $G$ be the orthogonal group $\Or (2)$, which is not
bi-invariantly orientable.
We embed $\Or (2) \subset \SO (3)$ as in the proof
of Proposition \ref{prop.gactsonegkorientpres}, i.e. by
$A\mapsto A\oplus \det A$.
None of the spaces
\[ \frac{\SO (3+k)}{\Or (2) \times \SO (k)}
   \cong
   \frac{\Or (3+k)}{\Or (2) \times \Or (k)} = BG_k \]
is orientable.
To see this, one may consult the fiber bundle
\[
\real \mathbb{P}^2 = \frac{\SO (3)}{\Or (2)} \longrightarrow
 \frac{SO(3+k)}{O(2) \times \SO (k)} \longrightarrow 
   \frac{\SO (3+k)}{\SO (3) \times \SO (k)}.
\]
The base is the Grassmannian $\widetilde{\Gr}_3 (\real^{3+k})$
of oriented $3$-planes in $\real^{3+k}$, which is an orientable
manifold.
If the total space were orientable, then the fiber $\real \mathbb{P}^2$
would be orientable, too, since its normal bundle is trivial.
\end{example}

\begin{remark}
Proposition \ref{prop.conngrpsarebiinvorient} together with
Proposition \ref{prop.biinvorientgrpisorupondeloop} implies  
that connected (compact) groups $G$ have orientable $BG_k$.
This may of course be established more directly by the following simple observation:
The principal $G$-bundle $p_k: EG_k \to BG_k$ induces an exact sequence
\[ \pi_1 (EG_k) \to \pi_1 (BG_k) \to \pi_0 (G) \to \pi_0 (EG_k) \]
on homotopy groups. For $k\geq 2$, $EG_k$ is simply connected
and the two outer groups vanish, so that the middle map is bijective. 
If $G$ is connected, then $\pi_0 (G)$ is trivial, and thus also
$\pi_1 (BG_k)$. So $BG_k$ is simply connected, thus orientable.
\end{remark}

\begin{prop} \label{prop.xgkwhitneystratpsdmfdoriented}
Let $G$ be a compact bi-invariantly orientable Lie group which acts smoothly on
a smooth manifold $M$. Suppose that
$(X,\Sa),$ $X\subset M,$ is a Whitney stratified oriented pseudomanifold
such that $X$ is $G$-invariant and the induced action of $G$
on $X$ is compatible with the stratification $\Sa$ and preserves the
orientation of $X$.
Then there exists an embedding $G\subset \Or (n)$ such that
the pseudomanifolds 
$(X_G (k),\Sa_G (k))$ (Theorem \ref{thm.xgkwhitneystratpsdmfdwitt}) 
are orientable for $k\geq 2$.
\end{prop}
\begin{proof}
Since $G$ is bi-invariantly orientable, there exists 
by Proposition \ref{prop.biinvorientgrpisorupondeloop} an 
embedding $G\subset \Or (n)$ such that every $BG_k,$ $k\geq 2,$ is oriented.
By assumption, the structure group $G$ of the fiber bundle
\[  X \hookrightarrow X_G (k) \stackrel{q_k}{\longrightarrow} BG_k \]
acts orientation preservingly on the fiber $X$, and the base space
$BG_k$ is an oriented manifold. Thus the total space $X_G (k)$
is canonically oriented as well. In fact if $U\subset BG_k$
is a chart and $q^{-1}_k (U) \cong U \times X$ a local trivialization over
$U$, then $U$ is oriented as an open subset of $BG_k$ and
$U\times X$ receives the product orientation. Then 
$q^{-1}_k (U)$ receives its orientation from the trivialization.
If $V\subset BG_k$ is another chart, with $U\cap V \not= \varnothing$,
then the orientations on $q^{-1}_k (V)$ and $q^{-1}_k (U)$
agree over $q^{-1}_k (U\cap V)$, since the transition functions
with values in $G$ act in an orientation preserving way.
(Another way of proving this is by looking at the Serre spectral
sequence of the above fiber bundle.)
\end{proof}
In the remainder of the paper, we will usually assume that
the embedding $G\subset \Or (n)$ of a bi-invariantly orientable group $G$ 
has been chosen so that the $BG_k,$ $k\geq 2,$ are orientable.

\section{Equivariant Homology}
\label{sec.equivarianthomology}

Equivariant homology, denoted $H^G_* (X),$ is an inverse limit
of the ordinary singular homology groups $H_* (X_G (k))$ as $k\to \infty$.
The inverse system is given by Gysin restrictions
$H_{i+n} (X_G (k+1)) \to H_i (X_G (k))$.
We review the construction and basic properties of equivariant homology groups.
They should be such that
compact oriented $G$-pseudomanifolds $X$ possess an equivariant fundamental class
$[X]_G \in H^G_m (X),$ $m=\dim X$ (see Section \ref{ssec.equivfundclass}),
and when $X$ is a manifold, $H^G_* (X)$ should be equivariantly
Poincar\'e dual to the equivariant cohomology
$H^*_G (X) = H^* (EG \times_G X)$ of Borel (see Section \ref{ssec.equivpd}).
The viewpoint on equivariant homology as adopted in this paper is 
compatible with the approach used by
Brylinski and Zhang in \cite{brylinskizhang},
Ohmoto in \cite{ohmoto1} and Weber in \cite{weber}. Here, we need to use compact topological
approximations $X_G (k)$ instead of the complex algebraic approximations
(which are generally noncompact) used in the algebraic setting. 
It turns out that the inverse limit stabilizes at finite $k$, as
we show in Lemma \ref{lem.gysinstabilization}.

We summarize the setup:
Let $G$ be a compact bi-invariantly orientable Lie group of dimension $d$.
Choose an embedding $G\subset \Or (n)$ as a closed subgroup such that
$BG_k$ is orientable for $k\geq 2$ (Proposition \ref{prop.biinvorientgrpisorupondeloop}).
Choose and fix orientations for every $BG_k$.
Suppose that $G$ acts smoothly on
a smooth manifold $M$ which contains a compact stratified oriented pseudomanifold
$(X,\Sa)$ of dimension $m$ as a $G$-invariant Whitney stratified subset
such that the induced action of $G$
on $X$ is compatible with the stratification $\Sa$ and preserves the
orientation of $X$.
The spaces $X_G (k) = EG_k \times_G X$ are compact oriented pseudomanifold
approximations of the Borel construction $X_G = EG \times_G X$
and they fiber over $BG_k$ with fiber $X$. The approximations
$BG_k$ of $BG$ are closed smooth manifolds of dimension
\[ b_k := nk + a,~  a= \smlhf n(n-1) - d. \]
The morphism (\ref{dia.egktoegkplusone}) of principal $G$-bundles
induces a morphism 
\begin{equation} \label{dia.xgktoxgkplusone}
\xymatrix{
X_G (k) \ar[d]_{q_k} \ar@{^{(}->}[r]^{\xi_k} & X_G (k+1) \ar[d]^{q_{k+1}} \\
BG_k \ar@{^{(}->}[r]^{\beta_k} & BG_{k+1}
} \end{equation}
of associated $X$-fiber bundles.
The closed inclusion $\beta_k$ is a smooth codimension $n$
embedding of $BG_k$ as a submanifold of $BG_{k+1}$.
The normal bundle $\nu_k$ of this embedding is a smooth
oriented vector bundle of rank $n$ over $BG_k$. (The orientation is
induced by the orientations on $BG_k$ and $BG_{k+1}$.) 
Since Diagram (\ref{dia.egktoegkplusone}) is cartesian, 
Diagram (\ref{dia.xgktoxgkplusone}) is cartesian as well,
for
\begin{align*} 
 \beta^*_k X_G (k+1)
 &= \beta^*_k (EG_{k+1} \times_G X)
   = (\beta^*_k EG_{k+1}) \times_G X \\
 &= EG_k \times_G X = X_G (k).  
\end{align*}
By Lemma \ref{lem.normalbundlesofxgk} of the Appendix, the inclusion 
$\xi_k$ is (even PL) normally nonsingular and its
normal bundle $\mu_k$ is given by the oriented vector bundle
$\mu_k = q^*_k \nu_k.$
Oriented normally nonsingular inclusions such as $\xi_k$ have associated Gysin restrictions.
The Gysin restriction 
\[ \xi^!_k: H_{i+n} (X_G (k+1)) \longrightarrow H_i (X_G (k)) \]
can be described as the composition
\[ H_{i+n} (X_G (k+1)) \longrightarrow
   H_{i+n} (X_G (k+1), X_G (k+1) - X_G (k))
  \stackrel{\simeq}{\longrightarrow}  
   H_i (X_G (k)), \]
where the second map is the excision isomorphism to a tubular neighborhood
composed with the Thom isomorphism of the oriented normal bundle $\mu_k$.
Keeping $i$ fixed, these Gysin restrictions $\xi^!_k$
make $\{ H_{i+kn} (X_G (k)) \}_{k=1,2,\ldots}$
into an inverse system with respect to the directed set
$(\{ k=1,2,\ldots \}, \leq)$.

\begin{defn}
Let $i$ be an integer.
The \emph{equivariant homology group} $H^G_i (X)$ in degree $i$
of $X$ is the inverse limit
\[ H^G_i (X) := \underset{\longleftarrow_k}{\lim}
  H_{i+ nk + a} (X_G (k)).
 \]
\end{defn}
We shall often refer to $i$ as the \emph{equivariant} degree of
a class in $H^G_i (X)$, as opposed to the degrees of representative elements
in the inverse system.

\begin{remark} \label{rem.equivabovedimvanishing}
The dimension of $X_G (k)$ is $m + nk + a$,
where $m=\dim X$. Therefore,
\[ H^G_i (X) =0 \text{ for } i > m = \dim X. \]
\end{remark}

\subsection{Stabilization}
\label{ssec.stabilization}

The starting point of the stabilization analysis is to determine
the connectivity of the inclusion $\xi_k$.
\begin{lemma} \label{lem.xikiskconnected}
For all $k\geq 1,$
the map $\xi_k: X_G (k) \hookrightarrow X_G (k+1)$ is
$k$-connected and induces an isomorphism
$(\xi_k)_*: H_i (X_G (k);\intg) \longrightarrow H_i (X_G (k+1);\intg)$
for $i<k$ and an epimorphism for $i=k$.
\end{lemma}
\begin{proof}
For the special case $X=\pt$, the map $\xi_k$ is the map
$\beta_k: BG_k \to BG_{k+1}$. 
Consider the morphism of long exact sequences induced 
on homotopy groups by the morphism (\ref{dia.egktoegkplusone}) 
of principal $G$-bundles.
It is then a standard consequence of Lemma \ref{lem.egkiskm1conn}
that $\beta_k$ is $k$-connected for all $k=1,2,\ldots$.
We return to general $X$.
The morphism (\ref{dia.xgktoxgkplusone}) of $X$-fiber bundles
induces a morphism of exact sequences
\[ \xymatrix{
\pi_0 (X) \ar[d]^\simeq \ar[r] &
\pi_0 (X_G (k)) \ar[d]^{(\xi_k)_*} \ar[r] &
\pi_0 (BG_k)=0 \ar[d]^{(\beta_k)_*} \\
\pi_0 (X) \ar[r] &
\pi_0 (X_G (k+1)) \ar[r] &
\pi_0 (BG_{k+1})=0,
} \]
which shows 
that the middle vertical map $(\xi_k)_*$ is a surjection.
It is in fact a bijection, as one verifies by using the
homotopy lifting property of the (locally trivial) principal $G$-bundle
$EG_{k+1} \times X \to EG_{k+1} \times_G X =X_G (k+1)$
obtained by restricting the smooth principal $G$-bundle
$EG_{k+1} \times M \to M_G (k+1)$, together with the fact
that $EG_k$ is path connected.
For $i\geq 1,$ 
the connectivity statement for $\xi_k$ in degree $i$ follows from
the morphism of long exact sequences
\[ \xymatrix{
\pi_{i+1} (BG_k) \ar[d]^{(\beta_k)_*} \ar[r] &
\pi_i (X) \ar[d]^\simeq \ar[r] &
\pi_i (X_G (k)) \ar[d]^{(\xi_k)_*} \ar[r] &
\pi_i (BG_k) \ar[d]^{(\beta_k)_*} \ar[r] &
\pi_{i-1} (X) \ar[d]^\simeq \\
\pi_{i+1} (BG_{k+1}) \ar[r] &
\pi_i (X) \ar[r] &
\pi_i (X_G (k+1)) \ar[r] &
\pi_i (BG_{k+1}) \ar[r] &
\pi_{i-1} (X)
} \]
induced by the morphism (\ref{dia.xgktoxgkplusone}) 
of $X$-fiber bundles, together with the fact that $\beta_k$
is $k$-connected as noted above.
The homological consequence is then standard, since we have
observed above that
$(\xi_k)_*: \pi_0 (X_G (k)) \to \pi_0 (X_G (k+1))$
is a bijection for all $k\geq 1$.  
\end{proof}

\begin{lemma} \label{lem.cohomstabilization}
(Cohomological Stabilization.)
The map $\xi^*_k: H^i (X_G (k+1);\intg) \to H^i  (X_G (k);\intg)$ 
is an isomorphism for $i<k$ and a monomorphism for $i=k$. 
\end{lemma}
\begin{proof}
This follows readily from Lemma \ref{lem.xikiskconnected} by
examining the morphism of universal coefficient sequences
associated to $\xi^*_k$ and using the four/five-lemma.
\end{proof}

\begin{remark} \label{rem.equivcohominvlim}
We note that the \emph{equivariant cohomology} in degree $i$ is given by
\[ H^i_G (X) = H^i (EG \times_G X)
    = \underset{\longleftarrow_k}{\lim}~
      \big( H^i  (X_G (k)) \stackrel{\xi^*_k}{\longleftarrow} H^i (X_G (k+1)) \big), \]
where the cohomological restriction $\xi^*_k$ is contravariantly induced
by the inclusion $\xi_k$.
The reason is that as the CW complex $X_G = EG \times_G X$ is the union of the
increasing sequence of subcomplexes $X_G (k),$ there is an exact sequence
\[
0 \to \underset{\longleftarrow_k}{{\lim}^1} H^{i-1} (X_G (k))
 \longrightarrow H^i (X_G) \longrightarrow
 \underset{\longleftarrow_k}{\lim} H^i (X_G (k)) \to 0.
\]
By the cohomological stabilization Lemma 
\ref{lem.cohomstabilization}, the $\xi^*_k$ are isomorphisms for $k>i$.
Thus the inverse system satisfies the Mittag-Leffler condition and
$\underset{\longleftarrow}{{\lim}^1} H^{i-1} (X_G (k))=0$.
In particular,
\begin{equation} \label{equ.cohomgxiscohomxgkforkgri} 
H^i_G (X;\intg) \cong
  H^i (X_G (k);\intg) \text{ for } k>i.
\end{equation}
\end{remark}

For nonsingular $X$, the following stabilization result can alternatively
be deduced from equivariant Poincar\'e duality, 
see Lemma \ref{lem.nonsingstabilization}.
Recall that $G\subset \Or (n)$ and $a := \smlhf n(n-1) - d,$ $d=\dim G$.
The condition $n\geq 3$ in the next lemma can of course always be satisfied by
re-embedding $G$ into a higher dimensional orthogonal group.
\begin{lemma} \label{lem.gysinstabilization}
(Homological Gysin Stabilization.)
Suppose that $n\geq 3,$ $k\geq 2,$ and let $m$ be the dimension of $X$.
Then the Gysin restriction $\xi^!_k: H_{i+n} (X_G (k+1);\rat) \to H_i (X_G (k);\rat)$
is an isomorphism for $i > m + (n-1)k + a$. Hence,
\[ H^G_j (X; \rat) \cong
   H_{j+nk+a} (X_G (k); \rat)  \]
for $k > m-j$.
\end{lemma}
\begin{proof}
Let $\pi = \pi_1 (BG_2)$ be the fundamental group of $BG_k$, $k=2$.
By Lemma \ref{lem.xikiskconnected} (for $X$ a point), 
the inclusion $\beta_k: BG_k \hookrightarrow BG_{k+1}$
is $k$-connected. Thus
$(\beta_k)_*: \pi_1 (BG_k) \to \pi_1 (BG_{k+1})$ is an isomorphism, $k\geq 2$,
which shows that $\pi = \colim_k \pi_1 (BG_k)$.
A compact subspace of a CW complex is contained in a finite subcomplex.
Thus for every compact subset 
$C \subset BG = \bigcup_k BG_k$, there exists a $k$ such that
$C \subset BG_k$.
Therefore,
$\pi_1 (BG) = \colim_k \pi_1 (BG_k) = \pi.$
Now $\pi_1 (BG)$ is the component group
$\pi_0 (G)$, which is finite as $G$ is compact.
This shows that $\pi$ is a finite group.

Let $U = BG_{k+1} - BG_k$ be the complement of $BG_k$ in
$BG_{k+1}$. Since the codimension of $BG_k$ in $BG_{k+1}$
is $n\geq 3,$ general position 
(or a Seifert-Van Kampen argument using the sphere bundle of the normal
bundle $\nu_k$ of $BG_k$)
implies that $U$ is path connected and
that the open inclusion 
$i: U \hookrightarrow BG_{k+1}$ induces an isomorphism
\[ i_*: \pi_1 (U) \stackrel{\simeq}{\longrightarrow} \pi_1 (BG_{k+1}) 
   = \pi.  \]
In particular, $\pi_1 (U) \cong \pi$ is finite.

Next, we will prove the following statement:
If $L$ is any local coefficient system of finite dimensional rational 
vector spaces on $U$, then
\begin{equation} \label{equ.hpukp1liszero}
H_p (U;L)=0~ 
   \text{ for } p\geq (n-1)k + n + a.
\end{equation}   
Since $\pi$ is a finite group, Maschke's theorem shows that
the group ring $\rat [\pi]$ is semi-simple.
We view $L$ as a left $\rat [\pi]$-module.
Let $\widetilde{U}$ be the universal cover of $U$ (which is path connected)
and consider the universal coefficient spectral sequence
\[ E^2_{p,q} = \Tor^{\rat [\pi]}_{p,q} (H_* (\widetilde{U};\rat), L)  
  \Rightarrow H_{p+q} (U; L). \]
Since $\rat [\pi]$ is semi-simple, $E^2_{p,q}=0$ for $p>0$.
Consequently, the spectral sequence collapses to an isomorphism
\[ H_* (U;L) \cong H_* (\widetilde{U};\rat) \otimes_{\rat [\pi]} L.  \]
Thus it suffices to prove that $H_p (\widetilde{U};\rat)$ vanishes
for $p$ in the indicated range.
Let $\widetilde{BG_{k+1}}$ be the universal cover of $BG_{k+1}$.
Since the closed inclusion
$\beta_k: BG_k \hookrightarrow BG_{k+1}$ induces 
a $\pi_1$-isomorphism, a homotopy lifting argument with respect to the
covering projection shows that
the pullback cover $\beta^*_k (\widetilde{BG_{k+1}})$ of $BG_k$ is path connected
and in fact simply connected.
Thus $\beta^*_k (\widetilde{BG_{k+1}})$ is the universal cover
$\widetilde{BG_k}$ of $BG_k$.
Similar observations apply to the
pullback cover $i^* (\widetilde{BG_{k+1}})$, given by
\[
i^* (\widetilde{BG_{k+1}}) 
   = \widetilde{BG_{k+1}} - \widetilde{BG_k}.  
\]
Since $\widetilde{BG_k}$ is a submanifold of $\widetilde{BG_{k+1}}$ of
codimension $n\geq 3$, and $\widetilde{BG_{k+1}}$ is path connected, 
general position implies that
$i^* (\widetilde{BG_{k+1}})$ is path connected.
Since the open inclusion $i: U \hookrightarrow BG_{k+1}$
induces a $\pi_1$-isomorphism, the
pullback cover $i^* (\widetilde{BG_{k+1}})$ 
is the universal cover $\widetilde{U}$ of $U$.
This shows that
\[ \widetilde{U} = \widetilde{BG_{k+1}} - \widetilde{BG_k},    \]
where the complement is taken with respect to the inclusion
$\widetilde{\beta_k}: \widetilde{BG_k} \hookrightarrow \widetilde{BG_{k+1}}$
which lifts $\beta_k$. Since the latter is $k$-connected,
$\widetilde{\beta_k}$ is $k$-connected as well. Thus
$(\widetilde{BG_{k+1}}, \widetilde{BG_k})$ is a $k$-connected
pair of path connected spaces, which implies that
\[ \widetilde{\beta_k}_*: H_i (\widetilde{BG_k}) \to H_i (\widetilde{BG_{k+1}}) \]
is an isomorphism for $i<k$ and an epimorphism for $i=k$.
Hence, in view of the long exact sequence
\[ \cdots \to H_i (\widetilde{BG_k}) 
   \to H_i (\widetilde{BG_{k+1}}) 
   \to H^\BM_i (\widetilde{U})
  \to H^\BM_{i-1} (\widetilde{BG_k}) \to \cdots
\]
associated to the closed inclusion $\widetilde{\beta_k}$,
the group $H^\BM_i (\widetilde{U})$ vanishes for $i\leq k$.
Thus also
$H^\BM_i (\widetilde{U};\rat)=0$ for $i\leq k$.
We write $b_k = nk + a$
for the dimension of $BG_k$. The set $\widetilde{U}$ is open in
the smooth manifold $\widetilde{BG_{k+1}}$ and hence itself a smooth manifold
of dimension $b_{k+1}$.
By Poincar\'e duality,
$H^\BM_i (\widetilde{U};\rat) \cong
   H^{b_{k+1} -i} (\widetilde{U};\rat)$
so that by the universal coefficient theorem,
$H_{b_{k+1} -i} (\widetilde{U};\rat)=0$ for $i\leq k$.
Setting $p = b_{k+1} -i,$ the condition $i\leq k$ is equivalent
to $p\geq (n-1)k + n + a$.
This finishes the proof of statement (\ref{equ.hpukp1liszero}).

We consider the open subset
\[ U_X := X_G (k+1) - X_G (k)  \]
of $X_G (k+1)$.
The fiber bundle $q_{k+1}: X_G (k+1) \to BG_{k+1}$ restricts to a fiber bundle
$U_X \to U$ with fiber $X$.
The terms $E^2_{p,q}$ of the Serre spectral sequence 
\[ E^2_{p,q} = H_p (U; \underline{H}_q (X;\rat))
   \Rightarrow H_{p+q} (U_X; \rat)  \]
of this bundle vanish if
$p\geq (n-1)k + n + a$ (by statement (\ref{equ.hpukp1liszero})),
or $q>m$.
Thus if $p+q \geq m + (n-1)k + n + a$, then
$H_p (U; \underline{H}_q (X;\rat))=0$, so by the spectral sequence,
$H_{p+q} (U_X; \rat)=0$.
Suppose that 
$i>m + (n-1)k + n + a$.
Then the two outer terms of the exact sequence
\[ H_i (U_X;\rat)
   \to H_i (X_G (k+1);\rat)
   \to H_i (X_G (k+1), U_X;\rat)
   \to H_{i-1} (U_X;\rat)  \]
vanish, so that the middle map is an isomorphism.
The Gysin restriction
\[ \xi_k^!: H_i (X_G (k+1);\rat) 
 \longrightarrow H_{i-n} (X_G (k);\rat) \]
can be described as the composition
\[ H_i (X_G (k+1);\rat) \longrightarrow 
   H_i (X_G (k+1), U_X;\rat) \stackrel{\simeq}{\longrightarrow}
H_{i-n} (X_G (k);\rat),
\]
and is thus an isomorphism if $i>m + (n-1)k + n + a$.
The statement to be proven follows by shifting the degree.
\end{proof}
The above Gysin stabilization is illustrated by Example \ref{exple.circlegroup},
where we consider the circle group.
Proving that equivariant homology is independent of choices
(Theorem \ref{thm.equivhomolindepofgrpemb}, 
Proposition \ref{prop.transferidentcompatblktokp1}) 
requires certain bundle transfer maps to be isomorphisms in
appropriate ranges. This bundle transfer stabilization will be 
developed next.
\begin{lemma} \label{lem.transferisisom}
(Homological Bundle Transfer Stabilization.)
Let $p:Y \to B$ be an oriented fiber bundle whose fiber $F$ is a compact 
oriented smooth
$d$-dimensional manifold, whose structure group $G$ is a compact Lie
group acting smoothly on $F$ and whose base $B$ is a finite CW complex
of dimension $b$. 
If $F$ is $(k-1)$-connected, then the transfer
$p^!: H_i (B;\intg) \to H_{i+d} (Y;\intg)$ is an isomorphism for $i>b-k$.
\end{lemma}
\begin{proof}
We shall first construct a fiberwise embedding
$\theta: Y \hookrightarrow B \times \real^s,$ $s$ large,
over $B$ which is normally nonsingular and possesses a normal bundle $\nu$
admitting the interpretation of a vertical normal bundle for $p$.
We use techniques of Becker and Gottlieb \cite{beckergottlieb}.
Let $\widetilde{Y} \to B$
be the underlying principal $G$-bundle of $p:Y\to B$, so that
$Y = \widetilde{Y} \times_G F.$
By the Mostow-Palais equivariant embedding theorem,
there exists, for sufficiently large $n$, 
an orthogonal $G$-module structure on the real vector space $\real^n$ 
and a smooth equivariant embedding $F \subset \real^n$.
Restricting the $G$-action to the complement $\real^n -F$ of the
invariant subspace $F\subset \real^n$ defines a smooth $G$-action on
$\real^n - F$.
Let $\eta$ be the vector bundle over $B$ with projection
$E(\eta)= \widetilde{Y} \times_G \real^n \to B.$
The equivariant embedding $F\subset \real^n$ induces a fiberwise embedding
\[ Y = \widetilde{Y} \times_G F \subset \widetilde{Y} \times_G \real^n = E(\eta) \]
over $B$.
Since $B$ is a finite CW complex, we can choose a complementary vector
bundle $\zeta$ over $B$ with
$\eta \oplus \zeta = B \times \real^s,$
the trivial rank $s$ vector bundle over $B$.
We obtain an embedding $\theta: Y \hookrightarrow B \times \real^s$
as the composition
$Y \subset E(\eta) \subset E(\eta \oplus \zeta) = B \times \real^s.$
By construction, the composition of $\theta$ with the
projection $B\times \real^s \to B$ to the first factor is $p$.
The embedding $\theta$ is normally nonsingular. In fact, its normal bundle
$\nu$ is given by the Whitney sum of $p^* (\zeta)$ and the
normal bundle of $Y$ in $E(\eta)$.
The vertical tangent bundle $T^\vee Y$ of $p$ is the vector bundle over $Y$
given by $\widetilde{Y} \times_G TF \to \widetilde{Y} \times_G F=Y$.
It is oriented, since the fiber bundle $p:Y\to B$ is oriented.
Since $T^\vee Y \oplus \nu = Y\times \real^s$ is trivial, $\nu$ is oriented.

The complement $(B\times \real^s) - \theta (Y)$ fibers over $B$ with
fiber $\real^s- (F \times \{ 0 \})$.
Indeed, let $P\to B$ be the underlying principal $\Or (s-n)$-bundle of
$\zeta$. The total space of $\eta \oplus \zeta$ is
\begin{align*}    
E(\eta \oplus \zeta) 
&= E(\eta) \times_B E(\zeta)
  = (\widetilde{Y} \times_G \real^n) \times_B (P\times_{\Or (s-n)} \real^{s-n}) \\
&= (\widetilde{Y} \times_B P) 
 \times_{G \times \Or (s-n)} (\real^n \oplus \real^{s-n}),
\end{align*}
and the Whitney sum $\eta \oplus \zeta$ has the underlying principal
bundle
$\widetilde{Y} \times_B P \to B,$
with structure group $G \times \Or (s-n)$.
In terms of this description, the image $\theta (Y)$ is given by
\[    
 \theta (Y) =
 (\widetilde{Y} \times_B P) 
 \times_{G \times \Or (s-n)} (F \oplus \{ 0_{s-n} \}),
\]
and the complement is
\[    
(B\times \real^s) - \theta (Y) =
 E(\eta \oplus \zeta) - \theta (Y) =
 (\widetilde{Y} \times_B P) 
 \times_{G \times \Or (s-n)} ((\real^n \times \real^{s-n}) - 
      (F \times \{ 0_{s-n} \})).
\]
This endows the complement with a fiber bundle projection
$(B\times \real^s) - \theta (Y) \to B$
whose fiber is $\real^s - (F \times \{ 0_{s-n} \})$.
The action of $G \times \Or (s-n)$ on $\real^s = \real^n \times \real^{s-n}$
extends to an action on
the one point compactification $S^s = \real^s \cup \infty$
with $\infty$ as a fixed point.
In the fiberwise one point compactification $B\times S^s$
of $B\times \real^s$, the complement of $\theta (Y)$ is given by
\[    
(B\times S^s) - \theta (Y) =
 (\widetilde{Y} \times_B P) 
 \times_{G \times \Or (s-n)} (S^s - (F \times \{ 0_{s-n} \})).
\]
This endows the complement in the fiberwise compactification 
with a fiber bundle projection
\begin{equation} \label{equ.fiberbundlebtssminusthetay}  
(B\times S^s) - \theta (Y) \longrightarrow B  
\end{equation}
whose fiber is $S^s - (F \times \{ 0_{s-n} \})$.
This bundle has a section $\sigma$ given by $\sigma (x) = (x,\infty),$
$x\in B$. The suspension $S^s B^+$ is the quotient
$S^s B^+ = (B \times S^s) / \sigma (B)$
and may be viewed as the Thom space of the trivial vector bundle 
$B\times \real^s$.
The relative Serre spectral sequence of the fiber bundle
(\ref{equ.fiberbundlebtssminusthetay}) has $E^2$ term
$E^2_{p,q} = H_p (B; \underline{H}_q (S^s - (F\times 0), \infty))$
and converges to
\[ H_{p+q} ((B\times S^s) - \theta (Y), \sigma (B)) 
   = \widetilde{H}_{p+q} (S^s B^+ - \theta (Y)). \]
By Alexander duality,
$\widetilde{H}_q (S^s - (F\times 0); \intg)
   \cong \widetilde{H}^{s-1-q} (F; \intg) =0$
for $s-1-q < k$, since $F$ is $(k-1)$-connected.
So $\widetilde{H}_q (S^s - (F\times 0); \intg)=0$ for $q > s-1-k$.
Suppose that $i = p+q > b + s-1-k$.
Then $p>b$ or $q> s-1-k$.
In the first case, $H_p (B; \underline{H}_q (S^s - (F\times 0), \infty))=0,$
as the cellular chain complex of the universal cover of $B$ has no cells
of dimension higher than $b$.
In the second case, i.e. when $q > s-1-k,$
$E^2_{p,q}$ also vanishes since the stalk of the local system
$\underline{H}_q (S^s - (F\times 0), \infty)$ vanishes
according to the Alexander duality argument.
This shows that
\[ \widetilde{H}_i (S^s B^+ - \theta (Y))=0 \text{ for } i > b + s-1-k.  \]

The bundle transfer
$p^!: H_i (B) \to H_{i+d} (Y)$ is given by the composition
\[ H_i (B) \cong
  \widetilde{H}_{i+s} (S^s B^+)
   \stackrel{T(p)_*}{\longrightarrow} 
     \widetilde{H}_{i+s} (\Th (\nu))
      \stackrel{\Phi~ \simeq}{\longrightarrow}
 H_{i+s-(s-d)} (Y), \]
see \cite{banaglbundletransfer}, \cite{banaglko}. 
Here, the first map is the suspension isomorphism and
$T(p): S^s B^+ \to \Th (\nu)$ is the Umkehr map
(i.e. Pontrjagin-Thom collapse)
associated to the fiber preserving normally nonsingular embedding
$\theta: Y \hookrightarrow B \times \real^s$.
Cap product with the Thom class of the oriented vector bundle $\nu$ 
defines the Thom isomorphism $\Phi$.
The above description of $p^!$ shows that the transfer is an isomorphism
if and only if $T(p)_*$ is.
Excision identifies
\[ \widetilde{H}_{i+s} (\Th (\nu))
   \cong H_{i+s} (S^s B^+,~ S^s B^+ - \theta (Y)). \]
Under this identification, $T(p)_*$ fits into the exact sequence
\[
\widetilde{H}_{i+s} (S^s B^+ - \theta (Y)) \to
\widetilde{H}_{i+s} (S^s B^+) \stackrel{T(p)_*}{\longrightarrow}
H_{i+s} (S^s B^+, S^s B^+ - \theta (Y)) \to
\widetilde{H}_{i+s-1} (S^s B^+ - \theta (Y)).
\]
(Since the Pontrjagin-Thom collapse is the identity near $\theta (Y)$,
the induced homomorphism $T(p)_*$ agrees with the homomorphism
induced on homology by the inclusion
$(S^s B^+,\varnothing) \subset (S^s B^+, S^s B^+ - \theta (Y)$.)
Suppose that $i > b-k$.
Then $i+s-1 > b+s-1-k$ and thus the two outer groups in the exact
sequence vanish. Consequently, $T(p)_*$, and hence $p^!$, is
an isomorphism in the range $i > b-k$.
\end{proof}

\subsection{Equivariant Fundamental Classes}
\label{ssec.equivfundclass}

Suppose that the $G$-pseudomanifold $X$ is compact, $\dim X=m$, and
that $X$ oriented as a pseudomanifold.
Then $X$ has a fundamental homology class
$[X] \in H_m (X)$
in ordinary top dimensional homology.
In equivariant homology, it has an \emph{equivariant fundamental class}
\[ [X]_G \in H^G_m (X),  \]
as we shall now explain.
The approximations $X_G (k)$ of $X_G$ are compact pseudomanifolds of dimension
$m + nk + a$.
As $G$ is bi-invariantly orientable, every $BG_k$ is orientable by
Proposition \ref{prop.biinvorientgrpisorupondeloop}.
We choose and fix an orientation of $BG_k$ for every $k$.
Since the $G$-action on $X$ preserves the orientation of $X$,
the $X_G (k)$ are orientable by Proposition \ref{prop.xgkwhitneystratpsdmfdoriented}.
We choose and fix an orientation for every $X_G (k)$ compatibly with the
orientations of the $BG_k$. The $X_G (k)$ then possess fundamental classes
$[X_G (k)] \in H_{m + nk + a} (X_G (k)).$ The normal bundles $\nu_k$ of
$BG_k$ in $BG_{k+1},$ as well as the normal bundles of
$X_G (k)$ in $X_G (k+1),$ receive induced orientations.
Then the Gysin restriction
\[  \xi^!_k: H_{m + n(k+1) + a} (X_G (k+1)) \longrightarrow
   H_{m + nk + a} (X_G (k))  \]
sends
$[X_G (k+1)] \mapsto [X_G (k)].$
(This may be viewed as a special case of the author's Verdier-Riemann-Roch formula
for $L$-classes of singular spaces, recalled below as Theorem 
\ref{thm.myVRRforL}.)
Hence the sequence $\{ [X_G (k)] \}_{k=1,2,\ldots}$ is an element of the
inverse system and defines an element
$[X]_G \in H^G_m (X)$.

\subsection{Equivariant Poincar\'e Duality}
\label{ssec.equivpd}

Suppose that $X=M$ is an $m$-dimensional closed oriented manifold.
Then by classical Poincar\'e duality, cap product with the 
fundamental class $[X_G (k)] = [M_G (k)]$
is an isomorphism
\[ \xymatrix@C=40pt{ 
H^{m-i} (M_G (k)) \ar[r]^<<<<<<<{-\cap [M_G (k)]} &
    H_{i + nk + a} (M_G (k)).   
 } \]
The diagram
\[ \xymatrix@C=50pt{ 
H^{m-i} (M_G (k)) \ar[r]^<<<<<<<<<{-\cap [M_G (k)]}_\simeq &
    H_{i + nk + a} (M_G (k)) \\
 H^{m-i} (M_G (k+1)) \ar[u]^{\xi^*_k} 
    \ar[r]^<<<<<<<<{-\cap [M_G (k+1)]}_\simeq &
    H_{i + n(k+1) + a} (M_G (k))  
      \ar[u]_{\xi^!_k} 
 } \]
commutes up to sign, since
\[ \xi^!_k (\omega \cap [M_G (k+1)])
  = (-1)^{(m-i)n} \xi^*_k (\omega) \cap (\xi^!_k [M_G (k+1)])
  = (-1)^{(m-i)n} \xi^*_k (\omega) \cap [M_G (k)]  \]
for $\omega \in H^{m-i} (M_G (k+1))$. To eliminate the sign, we can
without loss of generality choose an embedding $G\subset \Or (n)$ with 
$n$ even.
Then the sequence of isomorphisms
$\{ -\cap [M_G (k)] \}_{k=1,2,\ldots}$
constitutes an isomorphism of inverse systems, and thus an
isomorphism
\begin{equation} \label{equ.equivpoincduality} 
\xymatrix{ 
H^{m-i}_G (M) \ar[r]^{-\cap [M]_G}_\simeq & 
   H^G_i (M). 
 } \end{equation}
on inverse limits, given by capping with the equivariant
fundamental class of $M$. This is \emph{equivariant Poincar\'e duality}.

In the special case where $X=M$ is a manifold, the
homological Gysin Stabilization Lemma \ref{lem.gysinstabilization}
can be deduced from the cohomological stability
(\ref{equ.cohomgxiscohomxgkforkgri}) by a simple argument based on equivariant
Poincar\'e duality:
\begin{lemma} \label{lem.nonsingstabilization}
(Nonsingular Stabilization.)
If $X=M$ is an $m$-dimensional closed oriented $G$-manifold, then
\[ H^G_j (M; \rat) \cong
   H_{j+nk+a} (M_G (k); \rat)  \]
for $k > m-j$.
\end{lemma}
\begin{proof}
By Lemma \ref{lem.xikiskconnected}, 
the inclusion $\xi_k: M_G (k) \hookrightarrow M_G (k+1)$ 
induces an isomorphism
\[ (\xi_k)_*: H_i (M_G (k)) \longrightarrow H_i (M_G (k+1))  \]
for $i<k$. By the universal coefficient theorem,
$\xi^*_k: H^i (M_G (k+1);\rat) \to
    H^i (M_G (k);\rat)$
is an isomorphism for $i<k$.
Using (\ref{equ.cohomgxiscohomxgkforkgri}), 
equivariant Poincar\'e duality for $M,$ and ordinary
Poincar\'e duality for $M_G (k)$, and writing $j = m-i,$
\[
H^G_j (M;\rat)
 \cong H^i_G (M;\rat) 
 \cong H^i  (M_G (k);\rat) 
 \cong H_{m-i+nk+a}  (M_G (k);\rat)     
\]
for $k>i=m-j$.
\end{proof}

\subsection{The Equivariant Homology of a Point}

We discuss the special case where $X$ is a point.
Then $X_G (k) = EG_k \times_G \pt = BG_k$ so that
\[ H^G_i (\pt) = \underset{\longleftarrow_k}{\lim}
  H_{i+ nk + a} (BG_k).
 \]
In this case, with $m=0$, equivariant Poincar\'e duality
(\ref{equ.equivpoincduality}) shows that
\[  H^G_{-i} (\pt) \cong H^{i}_G (\pt) = H^i (BG). \] 

\begin{example} \label{exple.circlegroup}
We illustrate the approximations as $k\to \infty$ of the
equivariant homology $H^G_* (\pt)$ for the circle group $G=S^1=\SO(2)$.
For this group, we may take $n=2$. Then
\[ BS^1_k = \frac{\SO (2+k)}{\SO (2)\times \SO(k)} \]
is the oriented Grassmannian $\wgr_2 (\real^{2+k})$ of oriented $2$-planes
in $\real^{2+k}$. Its dimension is $\dim BS^1_k = 2k$.
The integral cohomology of these Grassmannians has been computed
by Van\v{z}ura in \cite{vanzura}; see also 
Mimura-Toda \cite[p. 129, Remark 4.8]{mimuratoda}.
The following table displays generators for the groups
$H^i (BS^1_k;\intg)$ up to $k=8$:
\[
\begin{array}{|c|c|c|c|c|c|c|c|c|} \hline
    & k=1 & k=2      & k=3     & k=4 & k=5 & 
       k=6 & k=7 & k=8 \\ \hline
H^0 & \cellcolor{Gray}  1 &  \cellcolor{Gray} 1      & \cellcolor{Gray}  1  
   & \cellcolor{Gray}  1 & \cellcolor{Gray}  1 & 
        \cellcolor{Gray} 1 & \cellcolor{Gray}  1 & \cellcolor{Gray}  1 \\ \hline
H^2 & t_1 & t_2, u_2 & \cellcolor{Gray} t_3     & \cellcolor{Gray} t_4 
   & \cellcolor{Gray} t_5 & 
      \cellcolor{Gray} t_6 & \cellcolor{Gray} t_7 & \cellcolor{Gray} t_8 \\ \hline            
H^4 &   0 & t_2 u_2  & u_3     & t^2_4, u_4 & \cellcolor{Gray} t^2_5 & 
       \cellcolor{Gray} t^2_6 & \cellcolor{Gray} t^2_7 & \cellcolor{Gray} t^2_8 \\ \hline
H^6 &   0 &   0      & t_3 u_3 & t_4 u_4 & u_5 & 
       t^3_6, u_6 & \cellcolor{Gray} t^3_7 & \cellcolor{Gray} t^3_8 \\ \hline
H^8 &   0 &   0      &   0     & t^2_4 u_4 & t_5 u_5 & 
       t_6 u_6 & u_7 & t^4_8, u_8 \\ \hline                     
H^{10} &   0 &   0      &   0     &   0 & t^2_5 u_5 & 
       t^2_6 u_6 & t_7 u_7 & t_8 u_8 \\ \hline    
H^{12} &   0 &   0      &   0     &   0 &   0 & 
       t^3_6 u_6 & t^2_7 u_7 & t^2_8 u_8 \\ \hline                 
H^{14} &   0 &   0      &   0     &   0 &   0 & 
         0 & t^3_7 u_7 & t^3_8 u_8 \\ \hline                  
H^{16} &   0 &   0      &   0     &   0 &   0 & 
         0 &   0 & t^4_8 u_8 \\ \hline  
\end{array}
\]
The cohomology in odd degrees vanishes.
The generator $t_k$ is the Euler class $e(\widetilde{\gamma}_2)$ 
of the tautological oriented $2$-plane 
bundle $\widetilde{\gamma}_2$ over $\wgr_2 (\real^{2+k})$.
The restriction 
$\beta^*_k: H^2 (BS^1_{k+1}) \to H^2 (BS^1_k)$ 
maps $t_{k+1}$ to $t_k$ because the Euler class is
natural and the restriction of the tautological bundle is the tautological bundle.
In the proof of the Nonsingular Stabilization Lemma \ref{lem.nonsingstabilization},
we noted that
\[ \xi^*_k: H^i (M_G (k+1);\rat) \longrightarrow
    H^i (M_G (k);\rat)  \]
is an isomorphism for $k>i$. For $M=\pt$, this means that
\[ \beta^*_k: H^i (BG_{k+1};\rat) \longrightarrow
    H^i (BG_{k};\rat)  \]
is an isomorphism for $k>i$. In the table, these stable fields are shaded in gray;
they do not contain any of the unstable classes $u_k$.
The inverse limit
\[ H^i_{S^1} (\pt;\rat) 
 = \underset{\longleftarrow_k}{\lim}~
      \big( H^i  (BS^1_k;\rat) 
      \stackrel{\beta^*_k}{\longleftarrow} H^i (BS^1_{k+1};\rat) \big) 
\]
is $\rat [t] = H^* (\cplx \mathbb{P}^\infty;\rat),$ $BS^1 \simeq \cplx \mathbb{P}^\infty$.

Let us move on to homology. The $S^1$-equivariant homology groups of a space $X$
are by definition
\[
H^{S^1}_i (X) 
 = \underset{\longleftarrow_k}{\lim}
  H_{i+ nk + \frac{1}{2}n(n-1)-d} (X_{S^1} (k)) 
 = \underset{\longleftarrow_k}{\lim}
  H_{i+ 2k} (X_{S^1} (k)).  
\]
For $X=\pt$,
\[ H^{S^1}_i (\pt) = \underset{\longleftarrow_k}{\lim}
  H_{i+2k} (BS^1_k). \]
Let $\tau_k = t_k \cap [BS^1_k]$ denote the Poincar\'e dual of the Euler class $t_k$ 
and let $\mu_k$ denote the Poincar\'e dual
of the unstable cohomology class $u_k$.
The following table displays generators for the homology groups
$H_i (BS^1_k;\intg)$:
\[
\begin{array}{|c|c|c|c|c|c|c|c|c|} \hline
    & k=1 & k=2      & k=3     & k=4 & k=5 & 
       k=6 & k=7 & k=8 \\ \hline
H_0    & \tau_1 & \tau_2 \mu_2 & \tau_3 \mu_3 & \tau^2_4 \mu_4 & \tau^2_5 \mu_5 & 
       \tau^3_6 \mu_6 & \tau^3_7 \mu_7 & \tau^4_8 \mu_8 \\ \hline
H_2    & \cellcolor{Gray1} [BS^1_1] & \tau_2, \mu_2 & \mu_3 & \tau_4 \mu_4 & \tau_5 \mu_5 & 
       \tau^2_6 \mu_6 & \tau^2_7 \mu_7 & \tau^3_8 \mu_8 \\ \hline
H_4    &   0 & \cellcolor{Gray1} [BS^1_2] & \cellcolor{Gray2} \tau_3 & \tau^2_4, \mu_4 & \mu_5 & 
       \tau_6 \mu_6 & \tau_7 \mu_7 & \tau^2_8 \mu_8 \\ \hline
H_6    &   0 &   0 & \cellcolor{Gray1} [BS^1_3] & \cellcolor{Gray2} \tau_4 
   & \cellcolor{Gray3} \tau^2_5 & 
       \tau^3_6, \mu_6 & \mu_7 & \tau_8 \mu_8 \\ \hline
H_8    &   0 &   0 &   0 & \cellcolor{Gray1} [BS^1_4] & \cellcolor{Gray2} \tau_5 & 
       \cellcolor{Gray3} \tau^2_6 & \tau^3_7 & \tau^4_8, \mu_8 \\ \hline
H_{10}    &   0 &   0 &   0 &   0 & \cellcolor{Gray1} [BS^1_5] & 
       \cellcolor{Gray2} \tau_6 & \cellcolor{Gray3} \tau^2_7 & \tau^3_8 \\ \hline
H_{12}    &   0 &   0 &   0 &   0 &   0 & 
       \cellcolor{Gray1} [BS^1_6] & \cellcolor{Gray2} \tau_7 
        & \cellcolor{Gray3} \tau^2_8 \\ \hline                                               
H_{14}    &   0 &   0 &   0 &   0 &   0 & 
         0 & \cellcolor{Gray1} [BS^1_7] & \cellcolor{Gray2} \tau_8 \\ \hline              
H_{16}    &   0 &   0 &   0 &   0 &   0 & 
         0 &   0 & \cellcolor{Gray1} [BS^1_8] \\ \hline
\end{array}
\]
The Gysin restrictions $\beta^!_k: H_{i+2} (BS^1_{k+1};\rat) \to H_i (BS^1_k;\rat)$
act in the table by moving one field up and one field left.
By the Gysin Stabilization Lemma \ref{lem.nonsingstabilization},
\[ H^G_j (M; \rat) \cong
   H_{j+nk+\frac{1}{2} n(n-1) -d} (M_G (k); \rat)  \]
for $k > m-j$, where $M$ is an $m$-dimensional closed oriented manifold.
For $M=\pt$ and $G=S^1$,
the Gysin restriction $\beta^!_k: H_{i+2} (BS^1_{k+1};\rat) \to H_i (BS^1_k;\rat)$
is an isomorphism for $i > m + (n-1)k + a = k$, and
\[ H^{S^1}_j (\pt; \rat) \cong H_{j+2k} (BS^1_k; \rat)  \]
for $k > -j$.
The groups $H^{S^1}_j (\pt) =0$ vanish for $j > \dim \pt =0$.
In equivariant degree $j=0$, the group
\[ H^{S^1}_0 (\pt) = \underset{\longleftarrow_k}{\lim}
  H_{0+2k} (BS^1_k) \cong H_2 (BS^1_1) \cong H_4 (BS^1_2) \cong 
   H_6 (BS^1_3) \ldots \cong \intg  \]
is generated by the sequence of fundamental classes
$[BS^1_1] \mapsfrom [BS^1_2] \mapsfrom [BS^1_3] \mapsfrom \ldots$.
For $j=-1$, the group
\[ H^{S^1}_{-1} (\pt) = \underset{\longleftarrow_k}{\lim}
  H_{-1+2k} (BS^1_k) \cong H_3 (BS^1_2) \cong H_5 (BS^1_3) \cong 
   H_7 (BS^1_4) \ldots =0  \]
vanishes, as is the case for all odd $j$.
 In equvariant degree $j=-2$, the group
\[ H^{S^1}_{-2} (\pt) = \underset{\longleftarrow_k}{\lim}
  H_{-2+2k} (BS^1_k) \cong H_4 (BS^1_3) \cong H_6 (BS^1_4) \cong 
   H_8 (BS^1_5) \ldots \cong \intg  \]
is generated by $\tau_3 \mapsfrom \tau_4 \mapsfrom \tau_5 \mapsfrom \ldots$,
and
\[ H^{S^1}_{-4} (\pt) = \underset{\longleftarrow_k}{\lim}
  H_{-4+2k} (BS^1_k) \cong H_6 (BS^1_5) \cong H_8 (BS^1_6) \cong 
   H_{10} (BS^1_7) \ldots \cong \intg  \]
has generator $\tau^2_5 \mapsfrom \tau^2_6 \mapsfrom \tau^2_7 \mapsfrom \ldots$.
These stability regions are shaded in gray in the homology table.
\end{example}

\subsection{Independence of Choices}
\label{ssec.equivhomolindepofchoices}

The above construction of equivariant homology involves a choice
of model for the manifolds $EG_k$ and $BG_k$. We shall prove that $H^G_* (X)$
is independent of these choices.
Thus let $G \subset \Or (n')$ be another embedding 
into some orthogonal group, yielding a smooth, free, compact,
$(k-1)$-connected $G$-manifold $E'G_k$ of dimension
\[ \dim E'G_k = n' k + \smlhf n' (n'-1). \]
The associated orbit space is a smooth closed manifold $B'G_k$.
As we did for $G\subset \Or (n)$,
we assume that the embedding $G\subset \Or (n')$ has been chosen
so that the action of $G$ on $E'G_k$ preserves the orientation 
(Proposition \ref{prop.gactsonegkorientpres}) and
$B'G_k$ is orientable (Proposition \ref{prop.biinvorientgrpisorupondeloop}). 
The space 
$X'_G (k) = E'G_k \times_G X$
fibers over $B'G_k$ with projection $q'_k: X'_G (k) \to B'G_k$
and fiber $X$.

\begin{thm} \label{thm.equivhomolindepofgrpemb}
For $k > m-j,$ bundle transfer induces an identification
\begin{equation} \label{equ.bundletransferidentif} 
H_{j + nk + \frac{1}{2} n(n-1) -d} (X_G (k);\intg) \cong  
   H_{j + n' k + \frac{1}{2} n'(n'-1) -d} (X'_G (k);\intg).  
\end{equation}   
\end{thm}
\begin{proof}
We will briefly write $E_k = EG_k,$ $E'_k = E'G_k,$
$B_k = BG_k,$ $B'_k = B'G_k.$
The product manifold
$\widehat{E}_k := E_k \times E'_k$
is $(k-1)$-connected, since both factors are.
The group $G$ acts diagonally and freely on $\widehat{E}_k$.
We denote the orbit space of this action by
$\widehat{B}_k = E_k \times_G E'_k.$
Then $\widehat{B}_k$ is a closed smooth manifold.
The projection $E_k \times E'_k \to E_k$ is $G$-equivariant
and induces a fiber bundle
\[ E'_k \longrightarrow 
   \widehat{B}_k = E_k \times_G E'_k 
   \stackrel{p_B}{\longrightarrow} E_k/G = B_k \]
with structure group $G$. 
As the action of $G$ on the fiber $E'_k$ preserves the orientation,
the fiber bundle $p_B$ is oriented.
Since $B_k$ is oriented as well, we deduce that the
total space $\widehat{B}_k$ is oriented.
The pseudomanifold 
\[ \widehat{X}_G (k) = \widehat{E}_k \times_G X \]
fibers over $\widehat{B}_k$ with projection 
$\widehat{q}_k: \widehat{X}_G (k) \to \widehat{B}_k$
and fiber $X$.
Since $\widehat{B}_k$ is oriented, and the action of $G$ on $X$
preserves the orientation, the total space $\widehat{X}_G (k)$
and the fiber bundle $\widehat{q}_k$ is oriented. 
The projection $E_k \times E'_k \times X \to E_k \times X$ is $G$-equivariant
and induces an oriented fiber bundle
\[ E'_k \longrightarrow 
   \widehat{X}_G (k) = (E_k \times E'_k) \times_G X 
   \stackrel{p_E}{\longrightarrow} E_k \times_G X = X_G (k). \]
We thus have a commutative diagram of fiber bundles
\begin{equation} \label{equ.qkpeishqkpb}
\xymatrix@R=15pt@C=18pt{
 & X \ar[d] & X \ar[d] \\
E'_k \ar[r] & \widehat{X}_G (k) \ar[d]_{\widehat{q}_k} \ar[r]^{p_E}
   & X_G (k) \ar[d]^{q_k} \\ 
E'_k \ar[r] & \widehat{B}_k \ar[r]_{p_B}
   & B_k 
} \end{equation}
This diagram is in fact cartesian, as we will show next.
The universal property for the pullback
\[ \widehat{B}_k \times_{B_k} X_G (k)
   = (E_k \times_G E'_k) \times_{E_k /G} (E_k \times_G X) \]
yields a map
$\phi: \widehat{X}_G (k) = (E_k \times E'_k \times X)/G
  \to \widehat{B}_k \times_{B_k} X_G (k)$
given by
\[ \phi [e,e',x] = (\widehat{q}_k [e,e',x],~
                    p_E [e,e',x])
                = ([e,e'],~ [e,x]).  \]
We will construct a map
$\psi: \widehat{B}_k \times_{B_k} X_G (k)
  \longrightarrow (E_k \times E'_k \times X)/G.$
Let $([e_1, e'],[e_2, x])$ be an element of the fiber product.  
Thus
$[e_1] = p_B [e_1, e'] = q_k [e_2, x] = [e_2],$
that is, $e_1$ and $e_2$ are in the same $G$-orbit in $E_k$.
Therefore, there exists a $g\in G$ such that $ge_1 = e_2$.
The element $g$ is unique, since $G$ acts freely on $E_k$.
Note that then $[e_2, x] = [ge_1, x] = [e_1, g^{-1} x]$.
We set
\[ \psi ([e_1, e'],[e_2, x]) := [e_1, e', g^{-1} x].  \]
It can then be verified that $\psi$ is well-defined, and
that the maps $\phi$ and $\psi$ are inverse to each other.
Thus, via these identifications,
the space $\widehat{X}_G (k)$ is the fiber product of 
$\widehat{B}_k$ and $X_G (k)$ over $B_k$. This shows that
Diagram (\ref{equ.qkpeishqkpb}) is cartesian.

The oriented fiber bundle $p_E$ has an associated transfer map
\[ p^!_E: H_* (X_G (k);\intg) \longrightarrow 
          H_{*+ n' k + \frac{1}{2} n' (n'-1)} (\widehat{X}_G (k);\intg). \]
By Lemma \ref{lem.transferisisom}, the transfer
\[
p^!_E: H_{j + nk + \frac{1}{2} n(n-1) -d} (X_G (k);\intg) 
  \longrightarrow 
     H_{j + nk + \frac{1}{2} n(n-1) + n' k + \frac{1}{2} n' (n'-1) -d} 
           (\widehat{X}_G (k);\intg) 
\]
is an isomorphism for $k > m-j$. 
Now, the model $X'_G (k)$ is similarly the base of the
oriented fiber bundle
\[ E_k \longrightarrow 
   \widehat{X}_G (k) = (E_k \times E'_k) \times_G X 
   \stackrel{p_{E'}}{\longrightarrow} E'_k \times_G X = X'_G (k), \]
whose transfer
\[
p^!_{E'}: H_{j + n' k + \frac{1}{2} n' (n'-1) -d} (X'_G (k);\intg) 
  \longrightarrow 
     H_{j + nk + \frac{1}{2} n(n-1) + n' k + \frac{1}{2} n' (n'-1) -d} 
           (\widehat{X}_G (k);\intg) 
\]
is an isomorphism for $k > m-j$.
The transfer identification (\ref{equ.bundletransferidentif})
is then given by the isomorphism
$(p^!_{E'})^{-1} \circ p^!_E$.
\end{proof}

The transfer identification of the above 
Theorem \ref{thm.equivhomolindepofgrpemb} commutes
with the passage from $k$ to $k+1$, as we will show next.
The equivariant inclusion $E'G_k \hookrightarrow E'G_{k+1}$
induces a normally nonsingular inclusion
$\xi'_k: X'_G (k) \hookrightarrow X'_G (k+1)$.
\begin{prop} \label{prop.transferidentcompatblktokp1}
The diagram
\[ \xymatrix{
H_{j + n(k+1) + \frac{1}{2} n(n-1) -d} (X_G (k+1);\intg) 
   \ar@{=}[d]_{\rotatebox{90}{$\sim$}} \ar[r]^{\xi^!_k}
 & H_{j + nk + \frac{1}{2} n(n-1) -d} (X_G (k);\intg) \ar@{=}[d]^{\rotatebox{90}{$\sim$}} \\
H_{j + n' (k+1) + \frac{1}{2} n'(n'-1) -d} (X'_G (k+1);\intg) \ar[r]_{\xi'^!_k}
 & H_{j + n' k + \frac{1}{2} n'(n'-1) -d} (X'_G (k);\intg),  
} \]
whose vertical isomorphisms are given by the transfer identifications
of Theorem \ref{thm.equivhomolindepofgrpemb} ($k>m-j$), commutes.
\end{prop}
\begin{proof}
We use the notation of the proof of Theorem \ref{thm.equivhomolindepofgrpemb}.
The commutative diagram of $G$-equivariant maps
\[ \xymatrix{
E_k \times X \ar@{^{(}->}[r] & E_{k+1} \times X \\
E_k \times E'_k \times X \ar[u]^{\operatorname{proj}} \ar@{^{(}->}[r] 
  & E_{k+1} \times E'_{k+1} \times X \ar[u]_{\operatorname{proj}}
} \]
induces on orbit spaces a commutative diagram
\[ \xymatrix{
X_G (k) \ar@{^{(}->}[r]^{\xi_k} & X_G (k+1) \\
\widehat{X}_G (k) \ar[u]^{p_E} \ar@{^{(}->}[r]_{\widehat{\xi}_k} 
  & \widehat{X}_G (k+1). \ar[u]_{p_E}
} \]
The analog of Diagram (\ref{dia.egktoegkplusone}) is the morphism of
principal $G$-bundles given by the cartesian diagram
\begin{equation} \label{dia.hatofegktoegkplusone}
\xymatrix{
\widehat{E}_k \ar[d]_{\widehat{p}_k} \ar@{^{(}->}[r] 
  & \widehat{E}_{k+1} \ar[d]^{\widehat{p}_{k+1}} \\
\widehat{B}_k \ar@{^{(}->}[r]^{\widehat{\beta}_k} & \widehat{B}_{k+1},
} \end{equation}
and the analog of Diagram (\ref{dia.xgktoxgkplusone}) is the morphism of
$X$-fiber bundles given by the cartesian diagram
\[ \xymatrix{
\widehat{X}_G (k) \ar[d]_{\widehat{q}_k} 
         \ar@{^{(}->}[r]^{\widehat{\xi}_k} & 
       \widehat{X}_G (k+1) \ar[d]^{\widehat{q}_{k+1}} \\
\widehat{B}_k \ar@{^{(}->}[r]^{\widehat{\beta}_k} & \widehat{B}_{k+1},
} \]
whose bottom horizontal map
$\widehat{\beta}_k$ is a smooth embedding of $\widehat{B}_k$ as a
submanifold of $\widehat{B}_{k+1}$.
The normal bundle $\widehat{\nu}_k$ of this embedding
is oriented since $\widehat{B}_k$ and $\widehat{B}_{k+1}$ are oriented.
By Lemma \ref{lem.normalbundlesofxgk}, $\widehat{\xi}_k$ is normally nonsingular
with oriented normal bundle $\widehat{q}^*_k \widehat{\nu}_k$.
Hence it admits a Gysin restriction $\widehat{\xi}^!_k$ which fits into
the diagram
\[ \xymatrix{
H_{j + n(k+1) + a} (X_G (k+1)) 
   \ar[d]_{p^!_E} \ar[r]^{\xi^!_k}
 & H_{j + nk + a} (X_G (k)) 
   \ar[d]^{p^!_E} \\
  H_{j + (n+n')(k+1) + a + \frac{1}{2} n'(n'-1)} 
       (\widehat{X}_G (k+1)) 
   \ar[r]_{\widehat{\xi}^!_k}
 & H_{j + (n+n')k + a + \frac{1}{2} n'(n'-1)} 
       (\widehat{X}_G (k)),
} \]
with $a=\frac{1}{2} n(n-1) -d$.
The diagram commutes by functoriality of Gysin maps.
By symmetry, we also have
$p^!_{E'} \xi'^!_k = \widehat{\xi}^!_k p^!_{E'}$.
Therefore,
\[ (p^!_{E'})^{-1} p^!_E \xi^!_k
   = (p^!_{E'})^{-1} \widehat{\xi}^!_k p^!_E 
   = \xi'^!_k (p^!_{E'})^{-1} p^!_E. \]
(Both instances of $p^!_{E'}$ are simultaneously invertible,
since $k>m-j$ implies $k+1>m-j$.)
\end{proof}
By the proposition, 
the transfer identification
induces an isomorphism between the inverse systems 
\[ \{ H_{j + nk + \frac{1}{2} n(n-1) -d} (X_G (k);\intg) \}_{k=1,2,\ldots}
\text{ and }
\{ H_{j + n'k + \frac{1}{2} n'(n'-1) -d} (X'_G (k);\intg) \}_{k=1,2,\ldots}.
\]
This shows that equivariant homology is well-defined independent
of choices.

\section{The Equivariant $L$-Class of Compact Pseudomanifolds}

Let $G$ be a compact Lie group of dimension $d$.
We assume that $G$ is bi-invariantly orientable.
This is for example the case if $G$ is connected, or finite, or abelian. 
Suppose that $G$ acts smoothly on
a smooth manifold $M$ which contains a compact stratified oriented Witt pseudomanifold
$(X,\Sa)$ of dimension $m$ as a $G$-invariant Whitney stratified subset
such that the induced action of $G$
on $X$ is compatible with the stratification $\Sa$ and preserves the
orientation of $X$.
Then Theorem \ref{thm.xgkwhitneystratpsdmfdwitt} 
together with Proposition \ref{prop.xgkwhitneystratpsdmfdoriented}
shows that the $L$-classes
\[ L_i (X_G (k)) \in H_i (X_G (k); \rat) \]
are well-defined following Goresky-MacPherson \cite{gmih1} and Siegel \cite{siegel}.
For $k=1,2,\ldots,$ we have the fiber bundles
\[ X \to X_G (k) \stackrel{q_k}{\longrightarrow} BG_k,  \]
where $BG_k$ is a smooth compact manifold. 
The tangent bundle $TBG_k$ of $BG_k$ has associated Hirzebruch $L$-classes
$L^* (TBG_k) \in H^* (BG_k;\rat)$ in cohomology.
These can be pulled back to yield invertible classes
$q^*_k L^* (TBG_k) \in H^* (X_G (k);\rat)$.

\begin{defn} \label{def.stagekequivlclass}
For $k=1,2,\ldots$, the \emph{stage-$k$ equivariant $L$-class} of the $G$-space
$X$ is defined to be
\[ L^G_{*,k} (X) := q^*_k L^* (TBG_k)^{-1} \cap L_* (X_G (k)) 
   \in  H_* (X_G (k);\rat). \]
\end{defn}

We will show that $\{ L^G_{*,k} (X) \}_{k=1,2,\ldots}$
is an element in the inverse limit that defines the $G$-equivariant homology of $X$.
In order to do so, we use the author's Verdier-Riemann-Roch formula
for $L$-classes of singular spaces:

\begin{thm} \label{thm.myVRRforL}
(Verdier-Riemann-Roch for Gysin restriction of $L$-classes, 
\cite[Thm. 3.18, p. 1300]{banaglnyjm}.)
Let $g: Y \hookrightarrow Z$ be a normally nonsingular inclusion of closed
oriented PL Witt pseudomanifolds. Let $\nu$ be the oriented 
normal vector bundle of $g$. Then
\[ g^! L_* (Z) = L^* (\nu) \cap L_* (Y). \]
\end{thm}

\begin{prop} \label{prop.stagekp1lmapstostagekl}
Under the Gysin restriction
$\xi^!_k: H_{*+n} (X_G (k+1)) \longrightarrow H_* (X_G (k)),$
the stage-$(k+1)$ equivariant $L$-class of $X$ maps to the stage-$k$
equivariant $L$-class, i.e.
\[ \xi^!_k L^G_{*,k+1} (X) = L^G_{*,k} (X).  \]
\end{prop}
\begin{proof}
We will use diagram (\ref{dia.xgktoxgkplusone}).
The restriction of $TBG_{k+1}$ to $BG_k$ splits as
$\beta^*_k TBG_{k+1} = TBG_k \oplus \nu_k$, where $\nu_k$ is the normal
bundle of $BG_k$ in $BG_{k+1}$. Hence,
\begin{align*}
\xi^!_k L^G_{*,k+1} (X)
&= \xi^!_k (q^*_{k+1} L^* (TBG_{k+1})^{-1} \cap L_* (X_G (k+1))) \\
&= \xi^*_k (q^*_{k+1} L^* (TBG_{k+1})^{-1}) \cap \xi^!_k L_* (X_G (k+1)) \\
&= (q^*_k \beta^*_k L^* (TBG_{k+1})^{-1}) \cap \xi^!_k L_* (X_G (k+1)) \\
&= (q^*_k L^* (TBG_k \oplus \nu_k)^{-1}) \cap \xi^!_k L_* (X_G (k+1)), 
\end{align*}
so that, using the multiplicativity of Hirzebruch's $L$-class,
\begin{equation} \label{equ.xikshriekofstagekp1l}
\xi^!_k L^G_{*,k+1} (X)
= q^*_k L^* (TBG_k)^{-1} q^*_k L^*(\nu_k)^{-1} \cap \xi^!_k L_* (X_G (k+1)).
\end{equation}
Appendix \ref{sec.plstructures} establishes the PL context required for applying
the VRR-formula of Theorem \ref{thm.myVRRforL}.
By Lemma \ref{lem.normalbundlesofxgk}, the normal bundle $\mu_k$ of the
normally nonsingular inclusion $\xi_k$ is given by the oriented vector bundle
$\mu_k = q^*_k \nu_k.$
Thus by the VRR-Theorem \ref{thm.myVRRforL},
\[
\xi^!_k L_* (X_G (k+1)) = L^* (\mu_k) \cap L_* (X_G (k))
  = q^*_k L^* (\nu_k) \cap L_* (X_G (k)).
\]
Plugging this expression into (\ref{equ.xikshriekofstagekp1l}),
the factor $q^*_k L^* (\nu_k)$ cancels its inverse and we find
\[
\xi^!_k L^G_{*,k+1} (X)
= q^*_k L^* (TBG_k)^{-1} \cap L_* (X_G (k)) = L^G_{*,k} (X).
\]
\end{proof}
This proposition shows that the sequence $\{ L^G_{*,k} (X) \}_{k=1,2,\ldots}$
is an element of the inverse limit
\[ \underset{\longleftarrow_k}{\lim}
  H_{*+ nk + a} (X_G (k); \rat) = H^G_* (X;\rat).
\]
\begin{defn} \label{def.equivlclasssingular}
The \emph{equivariant $L$-class} of the $G$-space
$X$ is defined to be
\[ L^G_* (X) := \{ L^G_{*,k} (X) \}_{k=1,2,\ldots}
   \in  H^G_* (X;\rat). \]
\end{defn}
The index $j$ in $L^G_j (X)$ thus indicates \emph{equivariant}
degree $j$.
Let us describe the equivariantly homogeneous components more
explicitly in equivariant degrees $m$ and $m-4$. 
Recall that $a = \smlhf n(n-1) - d.$
If $N$ is a smooth manifold, we find it convenient to write the
cohomological $L$-class of $N$ as
\[ L^* (N) = L^* (TN) = 1+ L^4 (TN) + L^8 (TN) + \cdots
   \in H^* (N;\rat) \]
with $L^j (TN) \in H^j (N;\rat)$, i.e. the superscript $j$ in $L^j (N)$
indicates the cohomological degree in which the class resides.
The inverse class $L^* (N)^{-1}$ is the $L$-class of the stable normal
bundle of $N$. Its components will be written as
$L^* (N)^{-1} = 1+ L^4 (TN)^{-1} + L^8 (TN)^{-1} + \cdots
   \in H^* (N;\rat)$
with $L^j (TN)^{-1} \in H^j (N;\rat)$.
The stage-$k$ equivariant $L$-class is
\[ L^G_{*,k} (X)
 = \big( 1+ q^*_k L^4 (BG_k)^{-1} + q^*_k L^8 (BG_k)^{-1} + \ldots \big) \]
\[ \hspace{3cm} \cap
   \big( L_{m+nk+a} (X_G (k)) + L_{m-4+nk+a} (X_G (k)) + \ldots \big), \]
with $L_{m+nk+a} (X_G (k)) = [X_G (k)]$ the fundamental class.
By the Gysin Stabilization Lemma \ref{lem.gysinstabilization},
$\xi^!_k: H_{i+n} (X_G (k+1);\rat) \to H_i (X_G (k);\rat)$
is an isomorphism for $i > m + (n-1)k + a$ 
(at least when $n\geq 3$ and $k\geq 2$), and
$H^G_j (X; \rat) \cong
   H_{j + nk + a} (X_G (k); \rat)$
for $k > m-j$.
For the top equivariant degree $j=m$,
stability holds from $k=2$ on. The equivariant groups are given by
\[ H^G_m (X;\rat)
  \cong H_{m+2n+a} (X_G (2);\rat) 
  \cong H_{m+3n+a} (X_G (3);\rat) 
  \cong H_{m+4n+a} (X_G (4);\rat)
  \cong \ldots. \]
The homogeneous component $L^G_m (X) \in H^G_m (X;\rat)$ of $L^G_* (X)$
in equivariant degree $m$ is then given by
\small
\[ L_{m+2n+a} (X_G (2)) = [X_G (2)] \mapsfrom
   L_{m+3n+a} (X_G (3)) = [X_G (3)] \mapsfrom
   L_{m+4n+a} (X_G (4)) = [X_G (4)] \mapsfrom \ldots. \]
\normalsize  
Thus $L^G_m (X) = [X]_G \in H^G_m (X;\rat)$.
In the next relevant equivariant degree $j=m-4$,
stability holds from $k=5$ on and the equivariant groups are  
\[ H^G_{m-4} (X;\rat)
  \cong H_{m-4+5n+a} (X_G (5);\rat) 
  \cong H_{m-4+6n+a} (X_G (6);\rat) 
  \cong \ldots. \]
The homogeneous component $L^G_{m-4} (X) \in H^G_{m-4} (X;\rat)$ of $L^G_* (X)$
in equivariant degree $m-4$ is the sequence
\[ L_{m-4+5n+a} (X_G (5)) +
     q^*_5 L^4 (BG_5)^{-1} \cap L_{m+5n+a} (X_G (5)) \]
\[ \hspace{1cm} \mapsfrom
   L_{m-4+6n+a} (X_G (6)) + 
     q^*_6 L^4 (BG_6)^{-1} \cap L_{m+6n+a} (X_G (6))  \] 
\[  \hspace{2cm} \mapsfrom
   L_{m-4+7n+a} (X_G (7)) +  
     q^*_7 L^4 (BG_7)^{-1} \cap L_{m+7n+a} (X_G (7))  \]
\[   \mapsfrom \ldots. \]

The above construction of the equivariant $L$-class involves a choice
of model for the manifolds $EG_k$ and $BG_k$. We shall prove that
the class is independent of these choices.
Thus, as in Section \ref{ssec.equivhomolindepofchoices}, 
let $G \subset \Or (n')$ be another embedding into some
orthogonal group, yielding a smooth, free, compact,
$(k-1)$-connected $G$-manifold $E'G_k$ of dimension
$n' k + \smlhf n' (n'-1).$
The associated orbit space is denoted by $B'G_k$; it is an
oriented compact smooth manifold.
The Witt space 
$X'_G (k) = E'G_k \times_G X$
fibers over $B'G_k$ with projection $q'_k: X'_G (k) \to B'G_k$
and fiber $X$.
The associated stage-$k$ equivariant $L$-class is
\[ L'^G_{*,k} (X) := q'^*_k L^* (TB'G_k)^{-1} \cap L_* (X'_G (k)) 
   \in  H_* (X'_G (k);\rat). \]
The proof of Theorem \ref{thm.stageklclassindepofgrpemb} below rests 
on the author's Verdier-Riemann-Roch formula
for $L$-classes of singular spaces under bundle transfer
homomorphisms. We recall the statement.
\begin{thm}  \label{thm.myVRRlclassbundletransfer}
(Verdier-Riemann-Roch for bundle transfer of $L$-classes, 
\cite[Thm. 8.1]{banaglbundletransfer}.)
Let $B$ be a closed PL Witt pseudomanifold
and let $F$ be a closed oriented PL manifold.
Let $\xi$ be an oriented PL $F$-block bundle over $B$
with total space $X$ and
oriented stable vertical normal PL microbundle $\mu$ over $X$.
Then $X$ is a Witt space and
the associated block bundle transfer 
$\xi^!: H_* (B;\rat)\to H_{*+\dim F} (X;\rat)$ sends the
Goresky-MacPherson $L$-class of $B$ to the product
\[ \xi^! L_* (B) = L^* (\mu) \cap L_* (X). \]
\end{thm}
This theorem holds in a rather general PL context. It is applicable
in the present paper, since our actions on singular spaces are the
restriction of smooth actions on ambient smooth manifolds. As mentioned
before, further PL details are provided in Appendix \ref{sec.plstructures}.
\begin{lemma} \label{lem.verticalnormalbundlepullsback}
Suppose that $p: Y\to B$ is a fiber bundle with smooth compact manifold fiber $F$,
compact Lie structure group $G$, and finite CW complex base $B$.
Let $f:B' \to B$ be a continuous map. Then the stable vertical normal
bundle $\nu^\vee_{Y'}$ of the pullback $F$-bundle $p': Y' = f^* Y \to B'$
is isomorphic to the pullback
$\overline{f}^* \nu^\vee_Y$ of the stable vertical normal bundle 
$\nu^\vee_Y$ of $Y$ over $B$,
where $\overline{f}: Y' \to Y$ is the map of $F$-fiber bundles covering $f$.
\end{lemma}
The proof of this lemma is straightforward and can for example be based on 
the description of the vertical normal bundle given in the proof of the
homological bundle transfer stabilization Lemma \ref{lem.transferisisom}.

\begin{thm} \label{thm.stageklclassindepofgrpemb}
For $k > m-j,$ the bundle transfer identification
(\ref{equ.bundletransferidentif}),
\[ 
H_{j + nk + \frac{1}{2} n(n-1) -d} (X_G (k);\rat) \cong  
   H_{j + n' k + \frac{1}{2} n'(n'-1) -d} (X'_G (k);\rat),  
\]   
maps the stage-$k$ equivariant $L$-class $L^G_{*,k} (X)$
to $L'^G_{*,k} (X)$.
\end{thm}
\begin{proof}
We use the notation of the proof of Theorem \ref{thm.equivhomolindepofgrpemb}.
In view of Diagram (\ref{equ.qkpeishqkpb}),
the stage-$k$ equivariant $L$-class 
transfers under the bundle projection $p_E$ to 
\begin{align*}
p^!_E L^G_{*,k} (X)
&= p^!_E (q^*_k L^* (TBG_k)^{-1} \cap L_* (X_G (k))) \\
&= (p^*_E q^*_k L^* (TBG_k)^{-1}) \cap p^!_E L_* (X_G (k)) \\
&= (\widehat{q}^*_k p^*_B L^* (TBG_k)^{-1}) \cap p^!_E L_* (X_G (k)).
\end{align*}
By the author's Verdier-Riemann-Roch theorem
for $L$-classes, Theorem \ref{thm.myVRRlclassbundletransfer},
\[ p^!_E L_* (X_G (k))
   = L^* (\nu^\vee_E) \cap L_* (\widehat{X}_G (k)), \]
where $\nu^\vee_E$ is the stable vertical normal bundle associated to $p_E$.
Applying Lemma \ref{lem.verticalnormalbundlepullsback} 
to the cartesian Diagram (\ref{equ.qkpeishqkpb})
(with $p=p_B$ and $f=q_k$), we obtain
a vector bundle isomorphism
$\nu^\vee_E = \widehat{q}^*_k \nu^\vee_B,$
where $\nu^\vee_B$ is the stable vertical normal bundle associated to $p_B$.
Note that
\[ T\widehat{B}_k \oplus \nu^\vee_B \cong
   p^*_B TB_k \oplus \epsilon^N, \]
where $\epsilon^N$ denotes a trivial vector bundle of sufficiently high 
rank $N$.   
Therefore,
\begin{align*}
p^!_E L^G_{*,k} (X)
&= (\widehat{q}^*_k p^*_B L^* (TBG_k)^{-1}) 
   \cup \widehat{q}^*_k L^* (\nu^\vee_B) \cap L_* (\widehat{X}_G (k)) \\
&= \widehat{q}^*_k (p^*_B L^* (TBG_k)^{-1}
   \cup L^* (\nu^\vee_B)) \cap L_* (\widehat{X}_G (k)) \\
&= \widehat{q}^*_k L^* (T\widehat{B}_k)^{-1} \cap L_* (\widehat{X}_G (k)) 
  =: \widehat{L}^G_{*,k} (X).   
\end{align*}
By the same token,
$p^!_{E'} L'^G_{*,k} (X) = \widehat{L}^G_{*,k} (X),$ which implies
\[ ((p^!_{E'})^{-1} \circ p^!_E) (L^G_{*,k} (X))
   = L'^G_{*,k} (X),    \]
as was to be shown.
\end{proof}
Thus the transfer identification
induces an isomorphism between the inverse systems 
\[ \{ H_{j + nk + a} (X_G (k);\rat) \}_{k=1,2,\ldots}
\text{ and }
\{ H_{j + n'k + a'} (X'_G (k);\rat) \}_{k=1,2,\ldots} \]
that sends
the equivariant $L$-class $\{ L^G_{*,k} (X) \}_k$
to the class $\{ L'^G_{*,k} (X) \}_k$.
This shows that the equivariant $L$-class is well-defined independent
of choices.

\section{The Nonsingular Case}

The topological approach to equivariant characteristic classes of smooth
manifolds is based on the following principle; see e.g. 
Tu \cite[p. 242]{tuintrolectequivcoh}.
Let $X$ be a $G$-space and $\xi$ a $G$-equivariant vector bundle with
equivariant projection $\pi: E \to X$. Then $\pi$ induces a map
$\pi_G: E_G \to X_G$ on homotopy quotients which is the projection of a vector bundle 
$\xi_G$ over $X_G$. 
Given a theory of (nonequivariant) cohomological characteristic classes 
$c^i \in H^i$, one defines the corresponding \emph{equivariant characteristic class}
to be
\[ c^i_G (\xi) := c^i (\xi_G) \in H^i (X_G) = H^i_G (X). \]
Suppose that $X=M$ is a smooth $G$-manifold, where $G$ is a compact Lie group.
Then $G$ acts on the tangent bundle $TM$ of $M$ via the differential of the action
on $M$ and thus $TM$ is a $G$-equivariant vector bundle. 
Hence the above principle applies and yields equivariant classes
\[ c^i_G (M) := c^i_G (TM) = c^i ((TM)_G) \in H^i_G (M). \]
In particular, the \emph{equivariant cohomological $L$-class of $M$} is given by
\[ L^i_G (M) = L^i ((TM)_G) \in H^i_G (M; \rat). \]
In terms of finite dimensional approximations,
the homotopy quotient can be written as
\[ (TM)_G = EG \times_G TM = \bigcup_k (EG_k \times_G TM).  \]
Thinking of $H^i_G (M)$ as the inverse limit
\[ \underset{\longleftarrow_k}{\lim}~
      \big( H^i  (M_G (k)) \stackrel{\xi^*_k}{\longleftarrow} H^i (M_G (k+1)) \big) \]
(Remark \ref{rem.equivcohominvlim}),
$L^i_G (M)$ can be described by the sequence 
$\{ L^i (EG_k \times_G TM) \}_k$, since
the vector bundle map
\[ \xymatrix{
EG_k \times_G TM \ar[d] \ar@{^{(}->}[r]^{\widetilde{\xi}_k}
  & EG_{k+1} \times_G TM \ar[d] \\
EG_k \times_G M \ar@{^{(}->}[r]_{\xi_k}
  & EG_{k+1} \times_G M   
} \]
is a fiberwise isomorphism and thus
\[ EG_k \times_G TM = \xi^*_k (EG_{k+1} \times_G TM),~ 
   L^i (EG_k \times_G TM) = \xi^*_k L^i (EG_{k+1} \times_G TM). \]

\begin{thm} \label{thm.manifoldcase}
Let $G$ be a compact Lie group which bi-invariantly orientable,
e.g. $G$ connected or finite or abelian,
and let $X=M$ be a smooth oriented closed $G$-manifold such
that the action is orientation preserving.
Then the equivariant homological $L$-class $L^G_* (M)$ of Definition
\ref{def.equivlclasssingular} is the equivariant Poincar\'e dual of the 
equivariant cohomological $L$-class $L^*_G (M),$ that is,
\[ L^G_* (M) = L^*_G (M) \cap [M]_G. \]
\end{thm}
\begin{proof}
The equivariant Poincar\'e duality isomorphism (\ref{equ.equivpoincduality})
is induced by the sequence of isomorphisms
$\{ -\cap [M_G (k)] \}_{k=1,2,\ldots}$.
With respect to the $M$-fiber bundle $q_k: M_G (k) \to BG_k$,
the tangent bundle of $M_G (k)$ splits as
$TM_G (k) \cong T^\vee M_G (k) \oplus q^*_k TBG_k,$
where $T^\vee M_G (k)$ is the vertical tangent bundle.
This bundle admits the description
\[ T^\vee M_G (k) = T^\vee (EG_k \times_G M) \cong EG_k \times_G TM.   \]
Now $EG_k \times_G TM =: (TM)_G (k)$ approximates the homotopy quotient
$(TM)_G = EG \times_G TM$.
For sufficiently large $k$, the stage-$k$ homological equivariant $L$-class of $M$ can
then be calculated as
\begin{align*}
L^G_{*,k} (M) 
&= q^*_k L^* (TBG_k)^{-1} \cap L_* (M_G (k)) \\
&= q^*_k L^* (TBG_k)^{-1} \cap (L^* (TM_G (k)) \cap [M_G (k)]) \\
&= q^*_k L^* (TBG_k)^{-1} \cap (L^* (T^\vee M_G (k) \oplus q^*_k TBG_k) \cap [M_G (k)]) \\
&= L^* (T^\vee M_G (k)) \cap [M_G (k)] \\
&= L^* ((TM)_G (k)) \cap [M_G (k)]. 
\end{align*}
\end{proof}

\begin{example} \label{exple.equivlofpoint}
For a point $M=\pt$, the vector bundle $(TM)_G$ is the
rank zero bundle over $M_G = BG$.
Therefore, $L^* ((TM)_G) = 1 \in H^* (BG;\rat)$ and by Theorem \ref{thm.manifoldcase},
\[ L^G_* (\pt) = L^*_G (\pt) \cap [\pt]_G = 1\cap [\pt]_G = [\pt]_G
  \in H^G_* (\pt;\rat) \cong H^{-*}_G (\pt) = H^{-*} (BG;\rat). \]
Concerning stage $k$, $M_G (k) = BG_k$ and $q_k$ is the identity. 
Hence the equivariant stage-$k$ class is
\[
L^G_{*,k} (\pt) 
 = q^*_k L^* (TBG_k)^{-1} \cap L_* (BG_k) 
 = L^* (TBG_k)^{-1} \cup L^* (TBG_k) \cap [BG_k] 
 = [BG_k],
\]  
which represents the equivariant fundamental class $[\pt]_G$ in the limit.
\end{example}

\section{Functoriality in the Group Variable}
\label{sec.functgroupvar}

Let $G$ be a bi-invariantly orientable compact Lie group acting on the
oriented pseudomanifold $X$, preserving its orientation.
An inclusion $G' \subset G$ of a closed bi-invariantly orientable subgroup $G'$
induces a restriction map
\begin{equation} \label{equ.restrfromgtogprime}
H^G_i (X) \longrightarrow H^{G'}_i (X)  
\end{equation}
on equivariant homology, as we shall recall next.
Let $G\subset \Or (n)$ be an embedding such that 
$G$ acts orientation preservingly on $EG_k$
and the manifolds $BG_k$ are orientable 
(Proposition \ref{prop.biinvorientgrpisorupondeloop}; $k\geq 2$).
We fix an orientation for every $BG_k$.
The embedding $G\subset \Or (n)$ induces by composition
an embedding $G' \subset \Or (n)$.
The underlying manifolds of $EG'_k$ and $EG_k$ are equal, namely
$\Or (n+k) / 1_n \times \Or (k)$.
Thus $G'$ acts orientation preservingly on $EG'_k = EG_k$
and the manifolds $BG'_k = EG_k /G'$ are orientable, $k\geq 2$.
We consider the fiber bundle
\[ G/G' \hookrightarrow BG'_k 
   \stackrel{\rho_k}{\longrightarrow} BG_k \]
given by the quotient map 
$BG'_k = EG_k /G' \to EG_k /G = BG_k.$
Since both the base and total space of $\rho_k$ are oriented,
this is an oriented fiber bundle.
As $EG_k /G' = EG_k \times_G (G/G')$, the structure group
of $\rho_k$ is $G$, and its underlying principal $G$-bundle
is $EG_k \to BG_k$.
More generally, the quotient map
\[ X_{G'} (k) = (EG_k \times X)/G' \longrightarrow 
        (EG_k \times X)/G = X_G (k) \]   
is the projection of a fiber bundle
\[ G/G' \hookrightarrow X_{G'} (k) 
   \stackrel{\pi_k}{\longrightarrow} X_G (k). \]
These projections fit into a cartesian diagram of fiber bundles
\begin{equation} \label{equ.gmodgprime}
\xymatrix@R=15pt@C=18pt{
 & X \ar[d] & X \ar[d] \\
G/G' \ar[r] & X_{G'} (k) \ar[d]_{q'_k} \ar[r]^{\pi_k}
   & X_G (k) \ar[d]^{q_k} \\ 
G/G' \ar[r] & BG'_k \ar[r]_{\rho_k}
   & BG_k. 
} \end{equation}
In particular, the fiber bundle $\pi_k$ is oriented, since
it is the pullback under $q_k$ of the oriented fiber bundle $\rho_k$. 
The dimension of $G$ is denoted by $d$, and the dimension of $G'$
shall be denoted by $d'$.
The oriented projection $\pi_k$ induces a bundle transfer homomorphism
\[ 
  \pi^!_k: H_{i+ nk + \frac{1}{2}n(n-1)-d} (X_G (k))
   \longrightarrow
   H_{i+ nk + \frac{1}{2}n(n-1)-d + (d-d')} (X_{G'} (k)).
\]
Under the passage from $k$ to $k+1$, these commute with the
Gysin restrictions $\xi^!_k$ by functoriality of Gysin maps.
Thus, taking inverse limits as $k\to \infty$, we obtain the 
desired restriction map (\ref{equ.restrfromgtogprime}).

\begin{thm} \label{thm.equivlchangeingroup}
Let $G' \subset G$ be a closed bi-invariantly orientable subgroup 
of the compact bi-invariantly orientable Lie group $G$
and let $X$ be a compact Witt $G$-pseudomanifold such that the $G$-action 
preserves the orientation of $X$.
Then the restriction map $H^{G}_* (X;\rat) \to H^{G'}_* (X;\rat)$
sends $L^{G}_* (X)$ to $L^{G'}_* (X)$.
\end{thm}
\begin{proof}
Using Diagram (\ref{equ.gmodgprime}), we compute the image of the 
stage-$k$ $G$-equivariant $L$-class
$L^G_{*,k} (X)$ under $\pi^!_k$:
\begin{align*}
\pi^!_k L^G_{*,k} (X)
&= \pi^!_k (q^*_k L^* (TBG_k)^{-1} \cap L_* (X_G (k))) \\
&= (\pi^*_k q^*_k L^* (TBG_k)^{-1}) \cap \pi^!_k L_* (X_G (k)) \\
&= (q'^*_k \rho^*_k L^* (TBG_k)^{-1}) \cap \pi^!_k L_* (X_G (k)).
\end{align*}
By the author's VRR-type theorem
for $L$-classes, Theorem \ref{thm.myVRRlclassbundletransfer},
\[ \pi^!_k L_* (X_G (k))
   = L^* (\nu^\vee) \cap L_* (X_{G'} (k)), \]
where $\nu^\vee$ is the stable vertical normal bundle associated to $\pi_k$.
Applying Lemma \ref{lem.verticalnormalbundlepullsback} 
to the cartesian Diagram (\ref{equ.gmodgprime})
(with $p=\rho_k$ and $f=q_k$), we obtain
a vector bundle isomorphism
$\nu^\vee = q'^*_k \nu^\vee_B,$
where $\nu^\vee_B$ is the stable vertical normal bundle associated to $\rho_k$.
Note that
$TBG'_k \oplus \nu^\vee_B \cong
   \rho^*_k TBG_k \oplus \epsilon^N,$
where $\epsilon^N$ denotes a trivial vector bundle of sufficiently high 
rank $N$. Therefore,
\begin{align*}
\pi^!_k L^G_{*,k} (X)
&= (q'^*_k \rho^*_k L^* (TBG_k)^{-1}) 
   \cup q'^*_k L^* (\nu^\vee_B) \cap L_* (X_{G'} (k)) \\
&= q'^*_k (\rho^*_k L^* (TBG_k)^{-1}
   \cup L^* (\nu^\vee_B)) \cap L_* (X_{G'} (k)) \\
&= q'^*_k L^* (TBG'_k)^{-1} \cap L_* (X_{G'} (k)) 
 = L^{G'}_{*,k} (X).   
\end{align*}
\end{proof}

\section{The Trivial Group and the Nonequivariant $L$-Class}

If $G = \{ 1 \}$ is the trivial group, then the Borel space is
$X_G = X_{\{ 1 \}}= E\{ 1 \} \times_{\{ 1 \}} X = E\{ 1 \} \times X \simeq X$.
In fact, one may of course take $E\{ 1 \}$ to be a point.
The trivial group embeds as $\{ 1 \} = \Or (0)$ into the orthogonal group with $n=0$.
The manifold approximations $EG_k$ used in this paper are then given by
$E\{ 1 \}_k = \Or(0+k) / (1_0 \times \Or (k)) = \pt,$
and $B\{ 1 \}_k = \pt$.
Thus the pseudomanifold approximations $X_G (k)$ are
$X_{\{ 1 \}} (k) = (E\{ 1 \}_k \times X)/ \{ 1 \} = X$ and
the equivariant homology $H^G_* (X)$ is
\[ H^{\{ 1 \}}_i (X)
    = \underset{\longleftarrow_k}{\lim}
  H_{i+ nk + \frac{1}{2}n(n-1)-d} (X_{\{ 1 \}} (k))
   = \underset{\longleftarrow_k}{\lim}
  H_{i} (X) = H_i (X), \]
the ordinary, nonequivariant, homology of $X$.
The stage-$k$ $\{ 1 \}$-equivariant $L$-class of $X$ is given by
\begin{align*} 
L^{\{ 1 \}}_{*,k} (X) 
&= q^*_k L^* (TB\{ 1 \}_k)^{-1} \cap L_* (X_{\{ 1 \}} (k)) \\
&= \id^* L^* (\pt \times \real^0)^{-1} \cap L_* (X) = L_* (X)
   \in  H_* (X;\rat), 
\end{align*}
the Goresky-MacPherson $L$-class of $X$. Hence
\[  L^{\{ 1 \}}_* (X) = L_* (X). \]

\begin{prop} \label{prop.equivlmapstononequivl}
Let $G$ be any compact bi-invariantly orientable Lie group acting on $X$.
The restriction map $H^G_* (X) \to H^{\{ 1 \}}_* (X) \cong H_* (X)$ 
induced by the inclusion of
the trivial group into $G$ sends $L^G_* (X)$ to the Goresky-MacPherson
class $L_* (X)$.
\end{prop}
\begin{proof}
This is essentially a corollary to Theorem \ref{thm.equivlchangeingroup}.
However, the construction of the restriction map in Section \ref{sec.functgroupvar}
requires both groups to be embedded in the same $\Or (n)$.
Thus in the context of the restriction map, we may technically not use the model
$X_{\{ 1 \}} (k) = X$ as above, but need to use the model
\[ X'_{\{ 1 \}} (k) = E'_k \times_{\{ 1 \}} X = E'_k \times X,~
    E'_k = \Or (n+k) / (1_n \times \Or (k)),
\]
where $G\subset \Or (n)$.
The restriction map is thus technically a map
\[ H^G_j (X) \to
   \underset{\longleftarrow_k}{\lim}
  H_{j+ nk + \frac{1}{2}n(n-1)} (X'_{\{ 1 \}} (k)).
\]
It sends
$L^{G}_* (X)$ to $L'^{\{ 1 \}}_* (X)$
according to Theorem \ref{thm.equivlchangeingroup}.
Let $p_E: X'_{\{ 1 \}} (k) \to X = X_{\{ 1 \}} (k)$ be the projection
to the second factor.
By Theorem \ref{thm.stageklclassindepofgrpemb}, the
bundle transfer $p^!_E$ induces for $k > m-j$ an identification
(\ref{equ.bundletransferidentif}) 
\[
p^!_E: H_j (X;\rat) 
  = H_{j} (X_{\{ 1 \}} (k);\rat) \stackrel{\simeq}{\longrightarrow}  
   H_{j + nk + \frac{1}{2} n(n-1)} (X'_{\{ 1 \}} (k);\rat),  
\]
under which the Goresky-MacPherson $L$-class $L_* (X)= L^{\{ 1 \}}_* (X)$
is mapped to $L'^{\{ 1 \}}_{*,k} (X)$. This finishes the proof.
We remark on the side that in the model $X'_{\{ 1 \}} (k)$, the
stage-$k$ class takes the form
\[
L'^{\{ 1 \}}_{*,k} (X) 
 = q'^*_k L^* (TB'\{ 1 \}_k)^{-1} \cap L_* (X'_{\{ 1 \}} (k)) 
 = [E'_k] \times L_* (X)
   \in  H_* (X'_{\{ 1 \}} (k);\rat). 
\]
\end{proof}
The proposition shows that the equivariant $L$-class contains all of the information
of the (nonequivariant) Goresky-MacPherson $L$-class.

\section{Products}

Let $G,G'$ be compact bi-invariantly orientable Lie groups.
Then their direct product $P := G \times G'$ is compact and
bi-invariantly orientable. Let $M$ be a smooth $G$-manifold
and $M'$ a smooth $G'$-manifold.
Let $X\subset M$ be an oriented $G$-invariant pseudomanifold and $X'\subset M'$ an 
oriented $G'$-invariant pseudomanifold. (The Whitney stratification assumptions are
as in the previous sections; the actions preserve orientations and strata.)
Then $M\times M'$ is a smooth $P$-manifold under the action
$(g,g')\cdot (x,x') = (gx, g'x')$. The product
$X \times X' \subset M\times M'$ is a $P$-invariant subspace and
a Whitney stratified pseudomanifold. The action of $P$ on $X\times X'$
preserves the product strata and the product orientation.
Equivariant rational homology satisfies a Künneth-theorem:

\begin{prop} \label{prop.equivariantkunneth}
(Equivariant Künneth Theorem.)
The homological cross product induces an isomorphism
\begin{equation} \label{equ.equivariantkunneth} 
\bigoplus_{p+q=j} H^{G}_p (X;\rat) \otimes H^{G'}_q (X';\rat)
   \cong H^{G\times G'}_j (X\times X';\rat). 
\end{equation}   
\end{prop}
\begin{proof}
Using bi-invariant orientability, choose embeddings 
$G \subset \SO (n)\subset \Or (n)$ and $G' \subset \SO (n')\subset \Or (n')$
in orthogonal groups such that the manifolds $BG_k, BG'_k$ are
oriented. The product group $P$ is embedded in
$\Or (n+n')$ via
\[ \xymatrix@R=15pt@C=15pt{
P=G\times G' \ar@{^{(}->}[r] & \SO (n) \times \SO (n') \ar@{^{(}->}[r] 
  \ar@{^{(}->}[d]
  & \Or (n) \times \Or (n') \ar@{^{(}->}[d] \\
& \SO (n+n') \ar@{^{(}->}[r] & \Or (n+n'). 
} \]
We set $Y := X \times X'$. 
The argument involves the $(k-1)$-connected free closed smooth manifolds
$EG_k,~ EG'_k,~ EP_k,$
their quotients
$BG_k = EG_k/G,~ BG'_k = EG'_k/G',~ BP_k = EP_k/P,$
as well as the approximating Borel pseudomanifolds
\[ X_G (k) = EG_k \times_G X,~ X'_{G'} (k) = EG'_k \times_{G'} X',~ 
     Y_P (k) = EP_k \times_P Y.   \]
Note that the embedding $G\times G' \subset \Or (n+n')$
is such that $G\times G'$ acts orientation preservingly on 
$EP_k \cong \SO (n+n'+k)/(1_{n+n'} \times \SO (k))$
$\cong \Or (n+n'+k)/(1_{n+n'} \times \Or (k))$ and, since $P$ is
bi-invariantly oriented,
$BP_k$ is oriented for every $k\geq 2$ by Lemma \ref{lem.totalspaceorientbaseorient}.
The pseudomanifold $Y_P (k)$ is oriented according to
Proposition \ref{prop.xgkwhitneystratpsdmfdoriented}.
It will be convenient to consider another model of $Y_P (k)$,
namely
$\overline{Y}_P (k) := \overline{E}_k \times_P Y,$
where 
$\overline{E}_k := EG_k \times EG'_k.$
The smooth closed manifold $\overline{E}_k$ is $(k-1)$-connected since both of its
factors are, and it is a free (smooth) $P$-space under the action
$(g,g')\cdot (e,e') = (ge, g' e')$. 
We endow $\overline{E}_k$ with the product orientation, for then the
action of $P$ on it preserves that orientation.
Let
$\overline{B}P_k = \overline{E}_k/P$ denote the orbit space, a smooth closed manifold.
The advantage of $\overline{Y}_P (k)$ over $Y_P (k)$ is that
the well-defined map
\[ (EG_k \times EG'_k) \times_{G \times G'} (X\times X')
   \longrightarrow (EG_k \times_G X) \times (EG'_k \times_{G'} X'),  \]
\[ (G\times G')(e,e',x,x') \mapsto (G(e,x), G'(e',x')),  \]   
determines a canonical identification
\[ \overline{Y}_P (k) = X_G (k) \times X'_{G'} (k) \]
with inverse
$(G(e,x), G'(e',x')) \mapsto (G\times G')(e,e',x,x').$
For $X=\pt =X'$ we have in particular a canonical identification
$\overline{B}P_k = BG_k \times BG'_k.$
We give both $\overline{B}P_k$ and $\overline{Y}_P (k)$ the product orientation.
The notation concerning dimensions will be
$m = \dim X$, $m' = \dim X'$, $d=\dim G$ and $d' = \dim G'$ so that
$\dim Y = m+m'$ and $\dim P = d+d'$.
We write $a = \smlhf n(n-1) -d,$ $a' = \smlhf n'(n'-1) -d'$.
By the Gysin Stabilization Lemma \ref{lem.gysinstabilization},
$H^G_p (X; \rat) \cong
   H_{p+nk+ a} (X_G (k); \rat)$
for $k > m-p,$ and
$H^{G'}_q (X'; \rat) \cong
   H_{q+n'k+ a'} (X'_{G'} (k); \rat)$
for $k > m'-q$.
The ordinary cross product is a map
\[
 H_{p+nk+ a} (X_G (k); \rat) \otimes
 H_{q+n'k+ a'} (X'_{G'} (k); \rat)
 \longrightarrow H_r (X_G (k) \times X'_{G'} (k);\rat),
\]
with $r = (p+q) +(n+n')k+ (a+a').$
Suppose that $j > m+m' = \dim Y$.
Then the right-hand side $H^P_j (Y;\rat)$ of (\ref{equ.equivariantkunneth})
vanishes (Remark \ref{rem.equivabovedimvanishing}).
If $p+q=j$, then $p>m$ or $q>m'$ so that $H^G_p (X;\rat)=0$
or $H^{G'}_q (X';\rat)=0$. Hence the left-hand side vanishes as well
and the equivariant Künneth formula holds.

Suppose then that $j \leq m+m'$ and take $k > m+m' -j$.
The ordinary, nonequivariant, Künneth theorem asserts that the direct sum
of cross products
\begin{equation} \label{equ.nonequivkunn}
\bigoplus_{p+q=j}
 H_{p+nk+a} (X_G (k); \rat) \otimes
 H_{q+n'k+a'} (X'_{G'} (k); \rat)
 \stackrel{\simeq}{\longrightarrow} H_r (X_G (k) \times X'_{G'} (k);\rat)
\end{equation}
is an isomorphism.
The sum may range only over $p,q$ with $p\leq m$ and $q\leq m'$,
since other terms do not contribute to the sum.
Then $k> m+m' - (p+q) = (m-p) + (m'-q) \geq \max (m-p,~ m'-q)$ so that
(\ref{equ.nonequivkunn}) may be written as
\[
\bigoplus_{p+q=j}
 H^G_{p} (X; \rat) \otimes
 H^{G'}_{q} (X'; \rat)
 \stackrel{\simeq}{\longrightarrow} H_r (\overline{Y}_P (k);\rat).
\]
It remains to identify the right-hand group with the $P$-equivariant
homology of $Y$.
By the Gysin Stabilization Lemma \ref{lem.gysinstabilization},
\[ H^P_j (Y; \rat) \cong
   H_{j+(n+n')k+\frac{1}{2} (n+n')(n+n'-1) -(d+d')} (Y_P (k); \rat)  \]
since we took $k > (m+m')-j$.
Now, an isomorphism
\begin{equation} \label{equ.isoypkybarpk}
(p^!_{E'})^{-1} \circ (p^!_E):
H_{j+(n+n')k+\frac{1}{2} (n+n')(n+n'-1) -(d+d')} (Y_P (k); \rat)
 \cong H_r (\overline{Y}_P (k);\rat)
\end{equation}
is given by the composition of two transfer isomorphisms as
in Theorem \ref{thm.equivhomolindepofgrpemb}.
Instead of Diagram (\ref{equ.qkpeishqkpb}), one uses the cartesian diagrams
of fiber bundles
\begin{equation} \label{dia.hatandypk}
\xymatrix{
 & Y \ar[d] & Y \ar[d] \\
\overline{E}_k \ar[r] & \widehat{Y}_P (k) \ar[d]_{\widehat{q}_k} \ar[r]^{p_E}
   & Y_P (k) \ar[d]^{q_k} \\ 
\overline{E}_k \ar[r] & \widehat{B}P_k \ar[r]_{p_B}
   & BP_k 
} \end{equation}
and
\begin{equation} \label{dia.hatandoverlypk}
\xymatrix{
 & Y \ar[d] & Y \ar[d] \\
EP_k \ar[r] & \widehat{Y}_P (k) \ar[d]_{\widehat{q}_k} \ar[r]^{p_{E'}}
   & \overline{Y}_P (k) \ar[d]^{\overline{q}_k} \\ 
EP_k \ar[r] & \widehat{B}P_k \ar[r]_{p_{\overline{B}}}
   & \overline{B}P_k. 
} \end{equation}
These two diagrams involve the following auxiliary objects:
The product manifold
$\widehat{E}P_k := EP_k \times \overline{E}_k$
is $(k-1)$-connected, since both factors are.
The group $P$ acts diagonally and freely on $\widehat{E}P_k$
with orbit space 
$\widehat{B}P_k := EP_k \times_P \overline{E}_k,$
a closed smooth manifold.
The projection $EP_k \times \overline{E}_k \to EP_k$ is $P$-equivariant
and induces a fiber bundle
\[ \overline{E}_k \longrightarrow 
   \widehat{B}P_k = EP_k \times_P \overline{E}_k 
   \stackrel{p_B}{\longrightarrow} EP_k/P = BP_k \]
with structure group $P$. 
As the action of $P$ on the fiber $\overline{E}_k$ preserves the orientation,
the fiber bundle $p_B$ is oriented.
Since the base $BP_k$ is oriented as well, we deduce that the
total space $\widehat{B}P_k$ is oriented.
The pseudomanifold 
$\widehat{Y}_P (k) := \widehat{E}P_k \times_P Y$
fibers over $\widehat{B}P_k$ with projection 
$\widehat{q}_k: \widehat{Y}_P (k) \to \widehat{B}P_k$
and fiber $Y$.
The projection $EP_k \times \overline{E}_k \times Y \to EP_k \times Y$ 
is $P$-equivariant and induces a fiber bundle
\[ \overline{E}_k \longrightarrow 
   \widehat{Y}_P (k) = (EP_k \times \overline{E}_k) \times_P Y 
   \stackrel{p_E}{\longrightarrow} EP_k \times_P Y = Y_P (k). \]
These are the bundles that appear in
(\ref{dia.hatandypk}).
The pullback of an oriented fiber bundle such as $p_B$ is oriented.
Thus $p_E$ is an oriented fiber bundle and therefore has an associated transfer.
The projection $EP_k \times \overline{E}_k \to \overline{E}_k$ 
is $P$-equivariant and induces a fiber bundle
\[ EP_k \longrightarrow 
   \widehat{B}P_k = EP_k \times_P \overline{E}_k 
   \stackrel{p_{\overline{B}}}{\longrightarrow} 
    \overline{E}_k/P = \overline{B}P_k \]
with structure group $P$. 
Since both the base and the total space of $p_{\overline{B}}$ are oriented,
this is an oriented fiber bundle.
The equivariant projection $EP_k \times \overline{E}_k \times Y 
 \to \overline{E}_k \times Y$ 
induces a fiber bundle
\[ EP_k \longrightarrow 
   \widehat{Y}_P (k) = (EP_k \times \overline{E}_k) \times_P Y 
  = EP_k \times_P (\overline{E}_k \times Y)
   \stackrel{p_{E'}}{\longrightarrow} \overline{E}_k \times_P Y = 
     \overline{Y}_P (k). \]
These bundles appear in (\ref{dia.hatandoverlypk}).
Since $p_{E'}$ is the pullback of the oriented bundle $p_{\overline{B}}$,
it is itself oriented and admits a transfer.
The fact that the two transfers $p^!_E$ and $p^!_{E'}$ are isomorphisms
is due to the homological transfer stabilization 
Lemma \ref{lem.transferisisom}.
The identification (\ref{equ.isoypkybarpk}) is compatible with 
the structure Gysin restrictions
of the inverse systems, as can be seen by an argument that proceeds as
the proof of Proposition \ref{prop.transferidentcompatblktokp1}.
\end{proof}

\begin{thm} \label{thm.productequivl}
Under the equivariant homological Künneth isomorphism (\ref{equ.equivariantkunneth}),
the equivariant $L$-class satisfies the product formula
\[ L^{G \times G'}_* (X\times X') 
      = L^G_* (X) \times L^{G'}_* (X'). \]   
\end{thm}
\begin{proof}
We use the notation of the proof of the
equivariant Künneth theorem, Proposition \ref{prop.equivariantkunneth}.
In particular, $P = G\times G',$ $Y = X\times X'$.
A $P$-equivariant stage-$k$ class in the homology of $\overline{Y}_P (k)$
is given by
$\overline{L}^P_{*,k} (Y) :=
      \overline{q}^*_k L^* (T\overline{B}P_k)^{-1} 
        \cap L_* (\overline{Y}_P (k)).$ 
By the nonequivariant product formula for the
Goresky-MacPherson $L$-class (\cite{blm}, \cite{woolf}),
\[ L_* (\overline{Y}_P (k))
  = L_* (X_G (k) \times X'_{G'} (k))
  = L_* (X_G (k)) \times L_* (X'_{G'} (k)). \]
Using the decomposition
\[ \xymatrix{
\overline{Y}_P (k) \ar[d]_{\overline{q}_k} \ar@{=}[r]
 & X_G (k) \times X'_{G'} (k) \ar[d]^{q_k \times q'_k} \\
\overline{B}P_k \ar@{=}[r]
 & BG_k \times BG'_k, 
} \]
we compute
\[
\overline{q}^*_k L^* (T\overline{B}P_k)
 = (q_k \times q'_k)^* L^* (TBG_k \times TBG'_k)
 = q^*_k L^* (TBG_k) \times q'^*_k L^* (TBG'_k).
\]
Therefore,
\begin{align*}
\overline{L}^P_{*,k} (Y)
&= (q^*_k L^* (TBG_k)^{-1} \times q'^*_k L^* (TBG'_k)^{-1}) 
    \cap (L_* (X_G (k)) \times L_* (X'_{G'} (k))) \\
&= (q^*_k L^* (TBG_k)^{-1} \cap L_* (X_G (k))) 
   \times (q'^*_k L^* (TBG'_k)^{-1} \cap L_* (X'_{G'} (k))) \\
&= L^G_{*,k} (X) \times L^{G'}_{*,k} (X').       
\end{align*}
On the other hand,
\[ (p^!_{E'})^{-1} \circ (p^!_E) L^P_{*,k} (Y)
    = \overline{L}^P_{*,k} (Y) \]
as in Theorem \ref{thm.stageklclassindepofgrpemb}, using the VRR-formula
Theorem \ref{thm.myVRRlclassbundletransfer} and Diagrams
\ref{dia.hatandypk} and \ref{dia.hatandoverlypk}.
\end{proof}

\section{The Case of a Trivial Action}

Suppose that the group $G$ acts trivially on $X$.
Then the Borel space of the action is $X_G = EG \times_G X = BG \times X$.
Its pseudomanifold approximations are given by
$X_G (k) = (EG_k \times X)/G = (EG_k /G)\times X = BG_k \times X$.
We may identify the $G$-space $X$ with the
$(G\times \{ 1 \})$-space $\pt \times X$. 
Then the equivariant Künneth theorem, Proposition \ref{prop.equivariantkunneth},
shows that
\[ H^G_* (X) = H^{G\times 1}_* (\pt \times X)
   \cong H^G_* (\pt) \otimes H_* (X).  \]
In Section \ref{ssec.equivfundclass}, we discussed the equivariant fundamental class.
For the point, this is a class
$[\pt]_G \in H^G_0 (\pt).$

\begin{prop} \label{prop.trivialaction}
If $G$ acts trivially on $X$, then
\[ L^G_* (X) = [\pt]_G \times L_* (X) \in H^G_* (X;\rat) \cong
                 H^G_* (\pt;\rat) \otimes H_* (X;\rat). \]
\end{prop}
\begin{proof}
This is a direct consequence of the product Theorem
\ref{thm.productequivl}, which asserts here that
\[ L^G_* (X) = L^{G \times 1}_* (\pt \times X) 
      = L^G_* (\pt) \times L^{\{ 1 \}}_* (X). \]  
In Proposition \ref{prop.equivlmapstononequivl}, we saw that 
$L^{\{ 1 \}}_* (X) = L_* (X)$, while Example \ref{exple.equivlofpoint}
shows that $L^G_* (\pt) = [\pt]_G$.
\end{proof}

\section{The Case of Free Actions}
\label{sec.freeactions}

We will show that if $G$ acts freely on $X$, then the equivariant
$L$-class $L^G_* (X)$ can be identified with the Goresky-MacPherson
class $L_* (X/G)$.

Suppose that $G$ acts freely on $M$, and $X\subset M$ is an invariant
compact oriented Witt pseudomanifold as before. The group $G$ is bi-invariantly
oriented and its action preserves the orientation of $X$.
Then $X/G$ is a Whitney stratified subset of the smooth manifold $M/G$
by Theorem \ref{thm.orbitspaceoffreeisbregular}, and a
pseudomanifold by Lemma \ref{lem.quotientoffreeispseudomfd}.
The orbit space $X/G$ is compact, since $X$ is.
The smooth principal $G$-bundle $M \to M/G$ restricts to a principal $G$-bundle
$X \to X/G$, and the induced smooth fiber bundle
$EG_k \hookrightarrow M_G (k) \to M/G$
restricts to a fiber bundle
\[ EG_k \hookrightarrow X_G (k) \stackrel{\pi_k}{\longrightarrow} X/G. \]
We note that the fiber of $\pi_k$ is a smooth manifold.
The pseudomanifolds $X_G (k)$ and $X/G$ are oriented.
To establish the Witt condition for $X/G$, we will employ the following
principle:
\begin{lemma} \label{lem.wittdesuspension}
(Witt Desuspension Principle.)
If $W$ is an oriented topologically stratified pseudomanifold such that the canonical
morphism $IC^{\bar{m}}_{W \times \real^j} \to
IC^{\bar{n}}_{W \times \real^j}$ of intersection chain sheaves
is an isomorphism (in the derived category) for some $j\geq 0$,
then $W$ satisfies the Witt condition.
\end{lemma}
\begin{proof}
Using Deligne's iterated construction of intersection chain sheaves
(given by alternating pushforwards $Ri_{k*}$ and truncations $\tau_{\leq \bar{p}(k)}$),
the canonical morphism $\phi_W: IC^{\bar{m}}_W \to IC^{\bar{n}}_W$
is induced by inclusions $\tau_{\leq \bar{m}(k)} (-)\subset \tau_{\leq \bar{n}(k)} (-)$.
The constant sheaf $\real_{W-\Sigma}$ on the top stratum of $W$ pulls back
to the constant sheaf $\pi^* \real_{W-\Sigma} = \real_{(W-\Sigma)\times \real^j}$
under the restriction of the stratified projection $\pi: W\times \real^j \to W$
over the top stratum. By base change, the exact functor $\pi^*$ commutes with the 
pushforwards $Ri_{k*}$,
and it also commutes with truncation.
Hence there is a commutative diagram
\[ \xymatrix@R=10pt{
IC^{\bar{m}}_{W \times \real^j} \ar[r]^{\phi_{W \times \real^j}}_\simeq 
   \ar@{=}[d]_{\rotatebox{90}{$\sim$}}
 & IC^{\bar{n}}_{W \times \real^j} \ar@{=}[d]_{\rotatebox{90}{$\sim$}} \\
\pi^* IC^{\bar{m}}_W \ar[r]_{\pi^* (\phi_W)} 
 & \pi^* IC^{\bar{n}}_W,  
} \]
which shows that $\pi^* (\phi_W)$ is an isomorphism in the derived category.
The stalk isomorphism
\[ \xymatrix@C60pt{
(\pi^* IC^{\bar{m}}_W)_{(w,t)} 
  \ar[r]^{(\pi^* (\phi_W))_{(w,t)}} &
   (\pi^* IC^{\bar{n}}_W)_{(w,t)} 
} \]
at a point $(w,t) \in W \times \real^j$ identifies naturally with the stalk map
$(\phi_W)_{w}: (IC^{\bar{m}}_W)_{w} \to (IC^{\bar{n}}_W)_{w}.$
Consequently, $(\phi_W)_w$ is an isomorphism for every $w\in W$, and thus $\phi_W$ is an isomorphism. 
\end{proof}

\begin{lemma}
The orbit space $X/G$ satisfies the Witt condition.
\end{lemma}
\begin{proof}
Since the Witt condition is local, it suffices to cover $X/G$ by
open subsets $U_\alpha$, each of which is Witt.
Now let $\{ U_\alpha \}$ be an open cover of $X/G$ such that
the fiber bundle $\pi_k$ trivializes over $U_\alpha$.
The preimages $\pi^{-1}_k (U_\alpha)$ are thus over $U_\alpha$ homeomorphic to 
$U_\alpha \times EG_k$, with $EG_k$ a smooth manifold. 
Recall that the total space $X_G (k)$ is Witt 
(by Theorem \ref{thm.xgkwhitneystratpsdmfdwitt}).
The preimages are open in $X_G (k)$, hence themselves Witt.
Since intersection homology is topologically invariant,
the spaces $U_\alpha \times EG_k$ are Witt.
Then $U_\alpha$ is Witt by the desuspension principle,
Lemma \ref{lem.wittdesuspension}.
\end{proof}
The Goresky-MacPherson $L$-class $L_* (X/G)$ is thus defined.
The commutative diagram 
\[ \xymatrix@C=15pt@R=20pt{
 X_G (k) \ar[rd]_{\pi_k} \ar@{^{(}->}[rr]^{\xi_k} & & X_G (k+1) \ar[ld]^{\pi_{k+1}} \\
 & X/G &
} \]
yields a commutative diagram of transfer homomorphisms
\[ \xymatrix{
 H_{j+n(k+1)+a} (X_G (k+1);\rat) \ar[rr]^{\xi^!_k} 
   & & H_{j+nk+a} (X_G (k);\rat) \\
 & H_{j-d} (X/G;\rat), \ar[lu]^{\pi^!_{k+1}} \ar[ru]_{\pi^!_k} &
} \]
$a=\smlhf n(n-1)-d$.
Thus the maps $\pi^!_k$ induce a map
\[ \pi^!: H_{j-d} (X/G;\rat) \longrightarrow H^G_j (X;\rat)  \]
into the inverse limit.
By Lemma \ref{lem.transferisisom}, 
$\pi^!_k: H_{j-d} (X/G;\rat) \to H_{j+nk+\frac{1}{2} n(n-1) -d} (X_G (k);\rat)$
is an isomorphism for $k > m-j$ and it follows that
\begin{equation} \label{equ.pishriekxmodgtohgx}
\pi^!: H_{j-d} (X/G;\rat) \cong H^G_j (X;\rat) 
\end{equation}
is an isomorphism.

\begin{lemma} \label{lem.tveepikqktbgkstablyiso}
The vertical tangent bundle $T^\vee \pi_k$ of $\pi_k$ and 
the vector bundle $q_k^* (TBG_k)$ are
stably isomorphic over $X_G (k)$ ($q_k: X_G (k) \to BG_k$).
\end{lemma}
\begin{proof}
The underlying principal bundle of
the $EG_k$-fiber bundle $\pi_k: X_G (k) \to X/G$ 
is the principal $G$-bundle $X \to X/G$.
Therefore, the vertical tangent bundle of $\pi_k$ is given by
$T^\vee \pi_k = (TEG_k) \times_G X.$ 
Since $G$ acts freely on $EG_k$, the quotient $(TEG_k)/G$ of the
$G$-vector bundle $TEG_k$ is a vector bundle over $EG_k/G = BG_k$.
There is a canonical isomorphism 
\[ q_k^* ((TEG_k)/G) = (TEG_k \times X)/G = T^\vee \pi_k, \]
of vector bundles over $X_G (k)$, and there is a decomposition
\[ (TEG_k)/G \cong T(EG_k/G) \oplus (T^\vee p_k)/G
    = TBG_k \oplus (T^\vee p_k)/G, \]
where $p_k: EG_k \to EG_k/G = BG_k$ is the orbit projection.
Now, the vector bundle $(T^\vee p_k)/G$ is in fact trivial, as we will
show next:
First, $T^\vee p_k$ is trivial as an ordinary vector bundle, 
being the vertical tangent bundle
of the projection of a principal $G$-bundle.
But more is true: In fact, $T^\vee p_k$ is trivial as a $G$-vector bundle,
which implies that the quotient $(T^\vee p_k)/G$ is trivial as an
(ordinary) vector bundle. Let us elaborate on this.
Let $v$ be a vector in the Lie algebra $\mathfrak{g}$ of $G$.
We can form the fundamental vector field $\sigma (v)$ on $EG_k$ given
at a point $x\in EG_k$ by
\[ \sigma(v)_x := \left. \frac{d}{dt}\right|_{t=0}~ x\cdot \exp (tv) \in T_x EG_k,    \]
where $\exp: \mathfrak{g} \to G$ is the exponential map, see
\cite{tudiffgeo}. This defines a
map $\sigma: \mathfrak{g} \to \mathfrak{X}(EG_k)$, where $\mathfrak{X}(EG_k)$
denotes the Lie algebra of smooth vector fields on $EG_k$. The map 
$\sigma$ is $\real$-linear.
In fact, $\sigma(v)_x$ is a vertical vector, i.e. an element
of $T^\vee_x (p_k)$. 
A smooth map $\tau: EG_k \times \mathfrak{g} \longrightarrow T^\vee p_k$,
which lies over $BG_k$ and which is fiberwise a linear isomorphism, 
is given by $\tau (x,v) := \sigma (v)_x.$ 
Let $\Ad: G \to GL (\mathfrak{g})$ be the adjoint representation.
A (right) $G$-module structure on $\mathfrak{g}$ is given by
$v \cdot g := (\Ad~ g^{-1}) (v).$
Then a right $G$-action on $EG_k \times \mathfrak{g}$ is given by 
$(x,v) \cdot g:= (xg, v\cdot g).$
Thus $EG_k \times \mathfrak{g}$ is a trivial $G$-vector bundle.
The group $G$ acts on $T^\vee p_k$ by the differential $R_{g*}$ of the
right translation $R_g: EG_k \to EG_k,$ $R_g (x)=xg$. The equation
\[ R_{g*} \sigma (v) = \sigma ((\Ad~ g^{-1}) v) \] 
for all $v\in \mathfrak{g}$
implies that the trivialization $\tau$ is $G$-equivariant.
In light of the fact that the functors $(-)/G$ and $p^*_k$ are inverse
to each other (using freeness) and set up an equivalence of categories between
$G$-vector bundles on $EG_k$ and vector bundles on $BG_k$,
this finishes the proof that $(T^\vee p_k)/G$ is a trivial vector bundle.
We conclude that
\[ T^\vee \pi_k = q_k^* ((TEG_k)/G) =
  q_k^* (TBG_k \oplus (T^\vee p_k)/G)
  = q_k^* (TBG_k) \oplus q_k^* ((T^\vee p_k)/G),
\]
with $q_k^* ((T^\vee p_k)/G)$ a trivial vector bundle.
\end{proof}

For free actions, we may now express the equivariant $L$-class in terms of the 
Goresky-MacPherson class.
\begin{thm} \label{thm.freeaction}
If $G$ acts freely on $X$, then the canonical identification
(\ref{equ.pishriekxmodgtohgx}) maps the Goresky-MacPherson class $L_* (X/G)$ to the
equivariant $L$-class $L^G_* (X)$.
\end{thm}
\begin{proof}
By the VRR-Theorem \ref{thm.myVRRlclassbundletransfer},
$\pi^!_k L_* (X/G) 
   =  L^* (T^\vee \pi_k)^{-1} \cap L_* (X_G (k)).$
(Note that all the PL requirements of the VRR-Theorem are met, since
the fiber $EG_k$ is smooth, hence PL and $\pi_k$ is the restriction of
a smooth fiber bundle, so in particular has an underlying PL $EG_k$-block bundle, 
which is oriented.
The spaces $X_G (k)$ and $X/G$ are PL since they are Whitney stratified.)
The stage-$k$ equivariant $L$-class of $X$ is thus given by
\begin{align*} 
L^G_{*,k} (X) 
&= q^*_k L^* (TBG_k)^{-1} \cap L_* (X_G (k)) \\
&= q^*_k L^* (TBG_k)^{-1} \cup L^* (T^\vee \pi_k) \cap \pi^!_k L_* (X/G). 
\end{align*}
By Lemma \ref{lem.tveepikqktbgkstablyiso},
$L^* (q^*_k TBG_k)^{-1} \cup L^* (T^\vee \pi_k) = 1,$ which implies
$L^G_{*,k} (X) = \pi^!_k L_* (X/G).$
\end{proof}

\section{Appendix: PL Structures}
\label{sec.plstructures}

This technical appendix establishes the PL context
required for an application of the author's VRR-theorems.
The essential point is that the action of $G$ on $X$ is by assumption the
restriction of a \emph{smooth} action of $G$ on $M$.
Smooth fiber bundles are PD-homeomorphic to PL bundles.
The relevant bundles used in the paper are thus restrictions of PL
bundles, hence again PL bundles.  
We will freely use the theory of blocktransversality as 
introduced by Rourke and Sanderson in \cite{rosablockbundles2}
for submanifolds, and extended by Stone \cite{stone}
to general polyhedra $P,Q$ in PL manifolds $N$. The latter concept has
been further clarified by McCrory's transversality based
on cone complexes, \cite{mccrory}, which is equivalent to blocktransversality as
McCrory has shown. 
We shall only need the case where one of the polyhedra, say $Q$, is
a locally flat PL submanifold, in which case blocktransversality
of $P$ and $Q$ roughly means that $Q$ has a normal block bundle
in $N$ which $P$ intersects in a union of blocks.

Suppose that 
$B$ is a smooth manifold and $A,X\subset B$ smooth submanifolds.
Let $\pi: E\to B$ be a surjective smooth submersion.
If the smooth manifold $\pi^{-1} (X)$ is (smoothly) transverse to
the smooth manifold $\pi^{-1} (A)$ in $E$, then
$X$ is transverse to $A$ in $B$.
This observation from differential topology has the following analog
in the PL category.
\begin{lemma} \label{lem.pltransvuptodown}
Let $B$ be a PL manifold, $A\subset B$ a locally flat closed PL submanifold
and $X\subset B$ a compact polyhedron. Let $\pi: E\to B$ be a
PL fiber bundle of PL manifolds with (nonempty) compact PL manifold fiber.
If the polyhedron $\pi^{-1} (X)$ is blocktransverse to
the PL submanifold $\pi^{-1} (A)$ in $E$, then
$X$ is blocktransverse to $A$ in $B$. In particular,
$X\cap A$ is PL normally nonsingular in $X$ with PL normal disc block bundle
given by the restriction of the PL normal block bundle of $A$ in $B$.
\end{lemma}
\begin{proof}
Let $\nu_A$ be the normal (closed-disc-) block bundle of the locally
flat PL submanifold $A$ in $B$. Let $E(\nu_A)$ denote the total space
of $\nu_A$.
Since the PL fiber bundle $\pi$ is transverse regular to any
submanifold of $B$, 
the preimage $\pi^{-1} (A)$ is a PL submanifold of $E$
and its normal bundle is the pullback
$\pi^*_A (\nu_A)$, where $\pi_A: \pi^{-1} (A) \to A$ denotes the
restriction of $\pi$.
The total space $E(\pi^*_A \nu_A)$ can be identified with
$\pi^{-1} (E(\nu_A)),$ since the manifold $\pi^{-1} (E(\nu_A))$ collapses
to $\pi^{-1} (A)$ (using local PL triviality), hence constitutes an abstract regular neighborhood of $\pi^{-1} (A)$. Any abstract regular neighborhood can be
equipped with a block bundle structure (by \cite[Thm. 4.3]{rosablockbundles1}),
and by the uniqueness theorem for normal block bundles
(\cite[Thm. 4.4]{rosablockbundles1}), there is an isotopy fixing $\pi^{-1} (A)$ which
realizes an isomorphism between the two block bundles.
By assumption, $\pi^{-1} (X)$ is blocktransverse to
the PL submanifold $\pi^{-1} (A)$ in $E$, that is,
\begin{equation} \label{equ.blocktransupstairs}
\pi^{-1} (X) \cap E(\pi^*_A \nu_A)
 = E((\pi^*_A \nu_A)|_{\pi^{-1} (X) \cap \pi^{-1} (A)}).
\end{equation}
The left hand side is identified with
$\pi^{-1} (X) \cap \pi^{-1} (E(\nu_A))
  = \pi^{-1} (X \cap E(\nu_A)).$
Using the commutative diagram
\[ \xymatrix@R=18pt{
\pi^{-1} (X\cap A) \ar[d]_{\pi_\cap} \ar@{^{(}->}[r]
 & \pi^{-1} (A) \ar[d]^{\pi_A} \\
X\cap A \ar@{^{(}->}[r] & A,
} \]
the right hand side of (\ref{equ.blocktransupstairs}) identifies with
$E(\pi^*_\cap (\nu_A|_{X\cap A}))
= \pi^{-1} (E(\nu_A|_{X\cap A})).$
Therefore, (\ref{equ.blocktransupstairs}) asserts that
$\pi^{-1} (X \cap E(\nu_A))
  = \pi^{-1} (E(\nu_A|_{X\cap A})).$
The surjectivity of $\pi$ guarantees that 
$\pi (\pi^{-1} (S)) =S$ for any subset $S\subset B$.
Thus
\[
  X \cap E(\nu_A)
  = \pi (\pi^{-1} (X \cap E(\nu_A)))
  = \pi (\pi^{-1} (E(\nu_A|_{X\cap A})))
  = E(\nu_A|_{X\cap A}),
\]
which means that $X$ is blocktransverse to $A$.
The equation also shows that $X\cap A$ has a PL tubular
neighborhood in $X$ (namely $X\cap E(\nu_A)$) which is
the total space of a block bundle, and that block
bundle is $\nu_A|_{X\cap A}$.
\end{proof}

In the following, we use the notation introduced in Section 
\ref{sec.whitneyapproxtoborel}, i.e. 
a compact Lie group $G$ acts smoothly on
a manifold $M$ and 
$(X,\Sa),$ $X\subset M,$ is a Whitney stratified subset
such that $X$ is $G$-invariant and the induced action of $G$
on $X$ is compatible with the stratification $\Sa$.
\begin{lemma} \label{lem.normalbundlesofxgk}
There exists a polyhedron $\widehat{M}_G (k+1)$ with
subpolyhedra $\widehat{M}_G (k)$ and $\widehat{X}_G (k+1)$, and
a PD-homeomorphism $\mu_G: \widehat{M}_G (k+1) \cong M_G (k+1)$ such that
\begin{enumerate}
\item $\mu_G (\widehat{M}_G (k)) = M_G (k)$, 
     $\mu_G (\widehat{X}_G (k+1)) = X_G (k+1)$,
\item $\widehat{M}_G (k+1)$ is a PL manifold and
  $\widehat{M}_G (k) \subset \widehat{M}_G (k+1)$ is a 
  locally flat PL submanifold,
\item the PL submanifold $\widehat{M}_G (k)$ and the polyhedron $\widehat{X}_G (k+1)$
  are blocktransverse in $\widehat{M}_G (k+1)$ with intersection
  \[ \widehat{M}_G (k) \cap \widehat{X}_G (k+1)
       = \widehat{X}_G (k),~ \mu_G (\widehat{X}_G (k)) = X_G (k),  \]
\item the inclusion 
   $\widehat{\xi}_k: \widehat{X}_G (k) \subset \widehat{X}_G (k+1)$
   is PL normally nonsingular with normal PL block bundle 
   $(q_k \mu_G)^* \widehat{\nu}_k$, 
   where $\widehat{\nu}_k$ is the underlying PL block bundle of the 
   normal vector bundle $\nu_k$ of the smooth inclusion 
   $\beta_k: BG_k \hookrightarrow BG_{k+1}$.
   In particular, the PL normal block bundle of $\widehat{\xi}_k$ is the 
   underlying block bundle of the vector bundle $(q_k \mu_G)^* \nu_k$.
\end{enumerate}
\end{lemma}
\begin{proof}
By Whitehead's smooth triangulation theorem,
there exists a compact polyhedron $\widehat{E}_{k+1}$ together
with a PD-homeomorphism $\epsilon: \widehat{E}_{k+1} \to EG_{k+1}$.
According to Verona's \cite[Thm. 7.8]{verona},
this may be done in such a way that the smooth submanifold
$EG_k \subset EG_{k+1}$ corresponds to a subpolyhedron
$\widehat{E}_k = \epsilon^{-1} (EG_k)$ so that $\epsilon$ restricts
to a PD homeomorphism $\widehat{E}_k \to EG_k$.
The subpolyhedron $\widehat{E}_k$ is a locally flat PL submanifold of the
PL manifold $\widehat{E}_{k+1}$.
Again using Whitehead's triangulation theorem,
there exists a compact polyhedron $\widehat{M}$ together
with a PD-homeomorphism $\mu: \widehat{M} \to M$.
Then $\widehat{M}$ is a PL manifold.
Since $X$ is Whitney stratified in $M$, there exists a subpolyhedron
$\widehat{X} \subset \widehat{M}$ such that
$\mu: \widehat{M} \to M$ restricts to a PD-homeomorphism
$\widehat{X} \to X$ (Goresky \cite{goreskytriang}, Verona \cite{verona}).
Taking products, we obtain a commutative diagram
\[ \xymatrix@R=10pt@C=15pt{
\widehat{E}_k \times \widehat{M} \ar@{^{(}->}[r]
  & \widehat{E}_{k+1} \times \widehat{M} \\
\widehat{E}_k \times \widehat{X} \ar@{^{(}->}[u] \ar@{^{(}->}[r] 
 & \widehat{E}_{k+1} \times \widehat{X} \ar@{^{(}->}[u]  
} \]
of polyhedra and PL embeddings.
The polyhedron $\widehat{E}_k \times \widehat{M}$ is a locally flat PL submanifold.
We note that
$\widehat{E}_k \times \widehat{X}
  = (\widehat{E}_{k+1} \times \widehat{X}) 
    \cap (\widehat{E}_k \times \widehat{M}).$
The PL submanifold $\widehat{E}_k \times \widehat{M}$
and the polyhedron $\widehat{E}_{k+1} \times \widehat{X}$ are
blocktransverse in the PL manifold $\widehat{E}_{k+1} \times \widehat{M}$,
as we shall show next.
Let $\widehat{\nu}_E$ be the normal block bundle of the locally flat
PL submanifold $\widehat{E}_k \subset \widehat{E}_{k+1}$.
Then the normal block bundle $\widehat{\nu}$ of the
PL submanifold
$\widehat{E}_k \times \widehat{M} \subset 
 \widehat{E}_{k+1} \times \widehat{M}$ is given by
$\widehat{\nu} = \widehat{\nu}_E \times \widehat{M}$.
We must show that
$(\widehat{E}_{k+1} \times \widehat{X}) 
  \cap E(\widehat{\nu}) =
  E(\widehat{\nu}|_{(\widehat{E}_{k+1} \times \widehat{X})
                        \cap (\widehat{E}_k \times \widehat{M})}).$
Indeed,
\begin{align*}
(\widehat{E}_{k+1} \times \widehat{X}) 
  \cap E(\widehat{\nu}) 
&= (\widehat{E}_{k+1} \times \widehat{X}) 
  \cap (E(\widehat{\nu}_E) \times \widehat{M}) 
 = (\widehat{E}_{k+1} \cap E(\widehat{\nu}_E)) 
   \times (\widehat{X} \cap \widehat{M}) \\
&= E(\widehat{\nu}_E) \times \widehat{X} 
 = E((\widehat{\nu}_E \times \widehat{M})|_{\widehat{E}_k \times \widehat{X}}) 
 = E(\widehat{\nu}|_{\widehat{E}_k \times \widehat{X}}) \\
&= E(\widehat{\nu}|_{(\widehat{E}_{k+1} \times \widehat{X})
                        \cap (\widehat{E}_k \times \widehat{M})}).
\end{align*}
By Whitehead's triangulation theorem,
there exists a compact polyhedron $\widehat{M}_G (k+1)$ together
with a PD-homeomorphism $\mu_G: \widehat{M}_G (k+1) \to M_G (k+1)$.
This polyhedron is then a PL manifold.
Consider the smooth principal $G$-bundle 
$\gamma: EG_{k+1} \times M \to M_G (k+1)$.
We forget that this bundle is principal and only consider its
underlying smooth fiber bundle.
The unique map 
$\pi: \widehat{E}_{k+1} \times \widehat{M}\to \widehat{M}_G (k+1)$
that makes the diagram
\begin{equation} \label{dia.pigamma} 
\xymatrix@R=15pt{
\widehat{E}_{k+1} \times \widehat{M} \ar[d]_\pi \ar[r]^{\epsilon \times \mu}
& EG_{k+1} \times M \ar[d]^\gamma \\
\widehat{M}_G (k+1) \ar[r]_{\mu_G} & M_G (k+1)
} \end{equation}
commutative, is a PL fiber bundle, whose fiber is the unique PL manifold
that underlies the smooth manifold $G$ (forgetting the group structure),
see e.g. Lurie \cite[Thm. 1]{lurie}.
The image of a compact polyhedron under a PL map is a compact polyhedron.
Thus we obtain the compact subpolyhedra
\[ \xymatrix@R=10pt@C=15pt{
\pi (\widehat{E}_k \times \widehat{M}) \ar@{^{(}->}[r]
  & \widehat{M}_G (k+1) \\
\pi (\widehat{E}_k \times \widehat{X}) \ar@{^{(}->}[u] \ar@{^{(}->}[r] 
 & \pi (\widehat{E}_{k+1} \times \widehat{X}). \ar@{^{(}->}[u]  
} \]
We note that by the commutativity of (\ref{dia.pigamma}),
the PD-homeomorphism $\mu_G$
restricts to homeomorphisms
\[ \widehat{M}_G (k) := \pi (\widehat{E}_k \times \widehat{M}) \to 
  M_G (k) = \gamma (EG_k \times M), \]
\[ \widehat{X}_G (k) := \pi (\widehat{E}_k \times \widehat{X}) \to 
  X_G (k) = \gamma (EG_k \times X), \]
\[ \widehat{X}_G (k+1):= \pi (\widehat{E}_{k+1} \times \widehat{X}) \to 
  X_G (k+1) = \gamma (EG_{k+1} \times X). \]
This provides polyhedral models for $M_G (k)$, $X_G (k)$, and $X_G (k+1)$
in $\widehat{M}_G (k+1)$, which satisfy
\begin{align*}
\mu_G (\widehat{M}_G (k))
&= \mu_G (\pi (\widehat{E}_k \times \widehat{M})) 
  = \gamma ((\epsilon \times \mu)(\widehat{E}_k \times \widehat{M})) \\
&= \gamma ((\epsilon (\widehat{E}_k) \times \mu(\widehat{M})) 
  = \gamma (EG_k \times M) 
 = M_G (k),
\end{align*}
and similarly for $\widehat{X}_G (k)$ and $\widehat{X}_G (k+1)$.
Using Diagram (\ref{dia.pigamma}), one verifies that 
\[ \pi^{-1} (\widehat{M}_G (k)) = \widehat{E}_k \times \widehat{M},~
  \pi^{-1} (\widehat{X}_G (k)) = \widehat{E}_k \times \widehat{X},~ 
  \pi^{-1} (\widehat{X}_G (k+1)) = \widehat{E}_{k+1} \times \widehat{X}. \]
Under the homeomorphism $\mu_G:\widehat{M}_G (k+1) \to M_G (k+1)$,
the equality
$X_G (k) = M_G (k) \cap X_G (k+1)$
implies that
$\widehat{X}_G (k) = \widehat{M}_G (k) \cap \widehat{X}_G (k+1).$
We showed earlier that the 
polyhedron $\pi^{-1} (\widehat{X}_G (k+1))$ is blocktransverse to the
PL manifold $\pi^{-1} (\widehat{M}_G (k))$ in the 
PL manifold $\widehat{E}_{k+1} \times \widehat{M}$.
Therefore, the PL manifold $\widehat{M}_G (k)$ is blocktransverse to the
polyhedron $\widehat{X}_G (k+1)$ by Lemma \ref{lem.pltransvuptodown}.
The lemma furthermore asserts that
$\widehat{X}_G (k) = \widehat{M}_G (k) \cap \widehat{X}_G (k+1)$
is PL normally nonsingular in $\widehat{X}_G (k+1)$ with PL normal disc block bundle
given by the restriction of the PL normal block bundle of the locally flat
PL submanifold
$\widehat{M}_G (k)$ in $\widehat{M}_G (k+1)$.
The morphism (\ref{dia.egktoegkplusone}) of principal $G$-bundles
induces a morphism 
\begin{equation} \label{dia.mgktomgkplusone}
\xymatrix@R=15pt{
M_G (k) \ar[d]_{s_k} \ar@{^{(}->}[r]^{\eta_k} & M_G (k+1) \ar[d]^{s_{k+1}} \\
BG_k \ar@{^{(}->}[r]^{\beta_k} & BG_{k+1}
} \end{equation}
of associated smooth $M$-fiber bundles.
The closed inclusion $\beta_k$ is a smooth embedding of $BG_k$ as a
submanifold of $BG_{k+1}$.
Let $\nu_k$ be the normal bundle of this embedding; this is a smooth
vector bundle of rank $n$. 
Since Diagram (\ref{dia.egktoegkplusone}) is cartesian, 
Diagram (\ref{dia.mgktomgkplusone}) is cartesian as well.
A smooth submersion is transverse to any smooth submanifold of its target.
Thus $s_{k+1}$ is transverse to $BG_k$. 
This implies that the normal vector bundle $\nu_\eta$ of the transverse inverse
image $s^{-1}_{k+1} (BG_k) = s^{-1}_k (BG_k) = M_G (k)$ is
the pullback $\nu_\eta = s^*_k \nu_k$.
Given a vector bundle $\xi$, let $\widehat{\xi}$ denote the
underlying block bundle, given in terms of classifying maps
by composition with the canonical map 
$\BO (n) \to \BPL(n) \to \BBPL(n)$.
If $f$ is a continuous map, then $f^* (\widehat{\xi}) \cong (f^* \xi)^\wedge$.
The PL normal block bundle of $\widehat{M}_G (k)$ in $\widehat{M}_G (k+1)$
is then $\mu^*_G \widehat{\nu}_\eta = \mu^*_G s^*_k \widehat{\nu}_k$
by Lashof-Rothenberg \cite[Thm. (7.3)]{lashofrothenberg}.
(In fact, their theorem shows a bit more, namely that $\widehat{M}_G (k)$ has a normal
PL microbundle which triangulates $\mu^*_G \nu_\eta$; via the map
$\BPL(n) \to \BBPL(n)$, every PL microbundle has an underlying PL block bundle.)
Thus the PL normal block bundle of 
$\widehat{X}_G (k)$ in $\widehat{X}_G (k+1)$
is the restriction of $(s_k \mu_G)^* \widehat{\nu}_k$ to 
$\widehat{X}_G (k)$.
This restriction is $(q_k \mu_G)^* \widehat{\nu}_k$ in view of the
commutative diagram
\[ \xymatrix@R=20pt{
\widehat{M}_G (k) \ar[r]^{\mu_G} &
  M_G (k) \ar[rd]^{s_k} & \\
\widehat{X}_G (k) \ar@{^{(}->}[u] \ar[r]_{\mu_G} &
  X_G (k) \ar@{^{(}->}[u] \ar[r]_{q_k} & BG_k.   
} \]
Since $(q_k \mu_G)^* \widehat{\nu}_k \cong ((q_k \mu_G)^* \nu_k)^\wedge,$
the PL normal block bundle of
$\widehat{X}_G (k)$ in $\widehat{X}_G (k+1)$ is the underlying
block bundle of a vector bundle.
\end{proof}

\end{document}